\documentclass[reqno,a4paper,10pt]{amsart}
\usepackage{amssymb,enumerate}
	\usepackage{a4wide} \usepackage[all]{xy}\usepackage{imakeidx}
	\makeindex 

\usepackage[shortlabels]{enumitem}
\def\com#1{}

\newcommand{\Z}{\ensuremath{{\rm{Z}}}}

\newcommand{\R}{\mathrm {R}}
\newcommand{\Mat}{\mathrm {Mat}}
\newcommand{\RU}{\mathrm {R}_u}

\newcommand{\aaa}{{\mathbf a}}
\newcommand{\bbb}{{\mathbf b}}

\newcommand{\Se}{{\mathsf{Set}}}
\newcommand{\BB}{{\mathbf B}}
\newcommand{\CC}{{\mathbf C}}
\newcommand{\DD}{{\mathbf D}}

\newcommand{\GG}{{\mathbf G}}
\newcommand{\GA}{{\GG_\aaa}}
\newcommand{\GB}{{\GG_\bbb}}
\newcommand{\HH}{{\mathbf H}}
\newcommand{\LL}{{\mathbf L}}
\newcommand{\KK}{{\mathbf K}}
\newcommand{\MM}{{\mathbf M}}
\newcommand{\NN}{{\mathbf N}}
\newcommand{\PP}{{\mathbf P}}
\newcommand{\bS}{{\mathbf S}}
\newcommand{\UU}{{\mathbf U}}
\newcommand{\VV}{{\mathbf V}}
\newcommand{\TT}{{\mathbf T}}

\newcommand{\XX}{{\mathbf X}}
\newcommand{\YY}{{\mathbf Y}}

\newcommand{\mo}{{\mathsf{mod}}}

\newcommand{\Mmo}{{\text{-}\mathsf{Mod}}}
\newcommand{\mmo}{{\text{-}\mathsf{mod}}}
\newcommand{\HO}{{\mathsf{Ho}^b}}
\newcommand{\rH}{{\mathrm{H}}}

\newcommand{\bbC}{{\mathbb{C}}}
\newcommand{\bbR}{{\mathbb{R}}}
\newcommand{\bbF}{{\mathbb{F}}}
\newcommand{\SF}{\mathrm {Spec}(\bbF)}
\newcommand{\FH}{{\mathbb{F}H}}

\newcommand{\Fm}{{\mathbb{F}^\times}}

\newcommand{\bbQ}{{\mathbb{Q}}}

\newcommand{\bbZ}{{\mathbb{Z}}}
\newcommand{\bbN}{{\mathbb{N}}}

\newcommand{\tA}{\mathsf A}
\newcommand{\tB}{\mathsf B}
\newcommand{\tC}{\mathsf C}
\newcommand{\tD}{\mathsf D}

\newcommand{\fsl}{{\mathfrak{sl}}}
\newcommand{\asl}{\widehat{\mathfrak{sl}}}

\newcommand{\cA}{{\mathcal A}}
\newcommand{\cC}{{\mathcal C}}
\newcommand{\cE}{{\mathcal E}}

\newcommand{\cF}{{\mathcal F}}

\newcommand{\cG}{{\mathcal G}}

\newcommand{\cL}{{\mathcal L}}
\newcommand{\HGU}{{\mathcal H}_\bbF(G,U)}
\newcommand{\cN}{{\mathcal N}}
\newcommand{\OO}{{\mathcal O}}

\newcommand{\cP}{{\mathcal P}}
\newcommand{\cS}{{\mathcal S}}

\newcommand{\Aut}{\operatorname{Aut}}
\newcommand{\diag}{\operatorname{diag}}
\newcommand{\End}{\operatorname{End}}
\newcommand{\Res}{\operatorname{Res}}
\newcommand{\Ind}{\operatorname{Ind}}
\newcommand{\id}{\operatorname{id}}

\newcommand{\tr}{{\operatorname{tr}}}

\newcommand{\Irr}{\operatorname{Irr}}
\newcommand{\IBr}{\operatorname{IBr}}
\newcommand{\Bl}{\operatorname{Bl}}

\newcommand{\Br}{\operatorname{Br}}
\newcommand{\Alp}{\operatorname{Alp}}
\newcommand{\CF}{\operatorname{CF}}

\newcommand{\conj}{\operatorname{conj}}
\newcommand{\op}{{\operatorname{opp}}}
\newcommand{\hoo}{{\operatorname{hook}}}
\newcommand{\pr}{{\operatorname{pr}}}

\newcommand{\Hom}{\operatorname{Hom}}

\newcommand{\reg}{\operatorname{reg}}

\newcommand{\SC}{{\operatorname{sc}}}
\newcommand{\ad}{{\operatorname{ad}}}
\newcommand{\GL}{\operatorname{GL}}
\newcommand{\SL}{\operatorname{SL}}
\newcommand{\PSL}{\operatorname{PSL}}
\newcommand{\PGL}{\operatorname{PGL}}

\def\Sym#1{\operatorname{\frak S}_{#1}}
\def\Alt{\operatorname{\frak A}}

\newcommand{\mm}{{-1}}
\newcommand{\W}{{\operatorname W}}
\newcommand{\Op}{{\mathsf{ Open}}}

\def\OP#1{{\operatorname O}_{p}(#1)}
\def\norm#1#2{{\operatorname N}_{#1}(#2)}
\def\cent#1#2{{\operatorname C}_{#1}(#2)}
\def\ccent#1#2{{\operatorname C}^\circ_{#1}(#2)}
\def\zent#1{{\operatorname Z}(#1)}
\def\czent#1{{\operatorname Z}^\circ(#1)}
\def\ser#1#2{{\mathcal E}(#1 , #2)}
\def\lser#1#2{{\mathcal E}_\ell(#1 , #2)}

\newcommand{\w}{\widetilde}
\newcommand{\wh}{\widehat}
\newcommand{\ov}{\overline }

\newcommand{\ovF}{\overline \bbF}

\newcommand{\tw}[1]{{}^#1}

\newcommand{\inn}{\subseteq}\newcommand{\nni}{\supseteq}
\newcommand{\deq}{\mathrel{\mathop:}=}
\newcommand{\GF}{{{\GG^F}}}
\newcommand{\GD}{{{\GG^*}}}
\newcommand{\LD}{{{\LL^*}}}
\newcommand{\FD}{{{F^*}}}
\newcommand{\LF}{{{\LL^F}}}

\newcommand{\pp}{{p'}}
\newcommand{\lp}{{\ell '}}
\def\Lu#1#2{{{\rm R}^{#2}_{#1}}}
\def\slu#1#2{{{}^*{\rm R}^{#2}_{#1}}}

\let\al=\alpha
\let\dd=\delta
\let\eps=\epsilon
\let\th=\theta
\let\si=\sigma
\let\zz=\zeta
\let\la=\lambda

\newcommand{\PGT}{{\Phi(\GG,\TT)}}
\newcommand{\WGT}{{W(\GG,\TT)}}

\let\co=\colon

\newtheorem{thm}{Theorem}[section]
\newtheorem{lem}[thm]{Lemma}
\newtheorem{cor}[thm]{Corollary}
\newtheorem{pro}[thm]{Proposition}
\newtheorem{prop}[thm]{Proposition}

\theoremstyle{definition}
\newtheorem{exm}[thm]{Example}
\newtheorem{defn}[thm]{Definition}
\newtheorem{rem}[thm]{Remark}

\def\Th#1{Theorem~\ref{#1}}
\numberwithin{figure}{section} 

\makeatletter
\def\restr#1|#2{\left.#1\right\rceil_{#2}}
\def\spann<#1>{\left\langle#1\right\rangle}

\def\Spann<#1>{\Spann@h#1@}
\def\Spann@h#1|#2@{\left\langle\left.#1\vphantom{#2}\hskip.1em\right.\mid\relax #2 \right\rangle}
\def\Set#1{\Set@h#1@}
\def\Set@h#1|#2@{\left\{\left.#1\vphantom{#2}\hskip.1em\,\right.
\mid\relax #2\right\}}
\def\set#1{\set@h#1@}
\def\set@h#1@{\left\{#1\right\}}

\makeatother

\begin{document}

\title{Local methods for blocks of finite simple groups}

\date{October 20, 2017}

\author{Marc Cabanes}
\address{CNRS, Institut de Math\'ematiques de Jussieu-Paris Rive Gauche,
Batiment Sophie Germain,
75205 Paris Cedex 13, France.}
\email{marc.cabanes@imj-prg.fr}

\thanks{ }



\begin{abstract}
 We review old and new results about the modular representation theory of finite reductive groups with a strong emphasis on local methods. This includes subpairs, Brauer's Main Theorems, fusion, Rickard equivalences. In the defining characteristic we describe the relation between $p$-local subgroups and parabolic subgroups, then give classical consequences on simple modules and blocks, including the Alperin weight conjecture in that case. In the non-defining characteristics, we sketch a picture of the local methods pioneered by Fong-Srinivasan in the determination of blocks and their ordinary characters. This includes the relationship with Lusztig's twisted induction and the determination of defect groups. We conclude with a survey of the results and methods by Bonnaf\'e-Dat-Rouquier giving Morita equivalences between blocks that preserve defect groups and the local structures. 
\end{abstract}

\maketitle

\tableofcontents


\centerline{\sc Introduction}

This survey aims at presenting in an almost self contained fashion some key results in the representation theory of finite quasi-simple groups that can be related to some global-local principle. For finite group theorists local information means information relating to normalizers of nilpotent subgroups. The typical situation is when given a finite group $G$ and a prime number $p$, one wants to guess information about $G$ from information of the same kind about subgroups $N$ normalizing a non normal $p$-subgroup of $G$. Those $N$ are sometimes called $p$-local subgroups. One has $N\lneq G$ so the process looks like somehow reducing the questions we might have about $G$ to questions about more tractable subgroups. This is particularly apparent in the classification of finite simple groups (CFSG, 1955--1980) where, at least in the earliest stages, $2$-local subgroups were systematically used to sort out simple groups by the structure of centralizers of involutions.

But what is the relevance of all that to representations, in particular of quasi-simple groups ? We try to give very concrete answers here. 

It is clear that in the years of the classification it was strongly believed that the $p$-local information on $G$ should determine many aspects of linear representations of $G$ in characteristic $p$. A short textbook by J.L. Alperin appeared in 1986 with the title ``Local representation theory" [Alper]. The main themes: Green's vertex theory, Brauer's morphism and defect groups, the case of cyclic defect and its consequence on the module category $B\text{-}\mo$ of the block $B$. On the other hand, the theme of ``simple groups and linear representations" was at that time recalling mainly the spectacular applications of character theory (both modular and ordinary) to CFSG, especially through Glauberman's Z*-theorem. This was exemplified by the influential survey [Da71] or the textbooks [Feit], [Nava].

Today the perspective has changed a little. CFSG and the wealth of knowledge on representations of finite groups of Lie type (see the survey [Ge17] of this volume) or symmetric groups make that representations of (quasi-)simple groups are becoming the main subject. The development of combinatorial representation theory and the recent interpretations in terms of categorification (see [ChRo08], [DuVV15], [DuVV17]), seem to hint at a situation where $p$-blocks of finite group algebras are classified regarding their module categories $\mo$ and other associated categories (derived $\tD^b$, homotopy $\HO$ or stable) {\it before} the relevant $p$-local information is known. The latter can even possibly be a consequence of the equivalence of blocks as rings (see Puig's theorem (\Th{Pu99} below) and its use in [BoDaRo17]). On the other hand, the notion of fusion systems and the topological questions or results it provides (see for instance [AschKeOl], [Craven]), have given a new perspective to the determination of local structure both for groups and blocks. 

We try here to sum up the relevance of local methods in representations of quasi-simple groups, essentially for groups of Lie type. 
 In the defining characteristic we describe the relation between $p$-local subgroups and parabolic subgroups, then we give classical consequences on simple modules and blocks, including the Alperin's weight conjecture. In the non-defining characteristics, we sketch a picture of the local methods pioneered by Fong-Srinivasan in the determination of blocks and their ordinary characters. On the method side one will find Brauer's three Main Theorems, Alperin-Brou\'e subpairs, both revolving around the Brauer morphism which will reappear also when discussing Rickard equivalences. On the side of results, we describe the relationship between blocks and Lusztig's twisted induction including the determination of defect groups. We also recall applications to Brauer's height zero conjecture (Kessar-Malle) and Brou\'e's abelian defect conjecture. In all cases we try to give many proofs at least for ``main cases" (leaving aside bad primes). We conclude with a survey of the results and methods of Bonnaf\'e-Dat-Rouquier ([BoRo03], [BoDaRo17]).
 
 The exposition follows a route prescribed by the groups we study. Abstract methods on blocks are only introduced when needed. The basics about $p$-local subgroups and fusion are in sections 1.C-D, $p$-blocks appear for the first time in 1.E with Brauer's first and third Main Theorems, Alperin's weight conjecture is recalled in 3.C, sections 5.A-D recall the general strategy to find the splitting of $\Irr (G)$ ($G$ a finite group) into blocks and the defect groups as an application of Brauer's second Main Theorem. Categorifications are evoked in 5.E, Rickard equivalences in 9.C.
 
This text grew out of the course and talks I gave in July and September 2016 during the program ``Local representation theory and simple groups" at CIB Lausanne. I heartily thank the organizers for giving me the opportunity to speak in those occasions and publish in this nice proceedings volume. 

\medskip\noindent{\bf On background and notation.}
We use freely the standard results and notation of basic module theory (see first chapter of [Benson]). For characters and block theory we refer to [NagaoTsu] and [AschKeOl, Ch. IV] but restate most theorems used with references. For categories and homological algebra, we refer to the first part of [Du17] whose notations we follow\index{$\tD^b$}\index{$\HO$}. For varieties, algebraic groups and finite groups of Lie type, our notations are the ones of [DigneMic] and [CaEn]. We borrow as much as possible from [Ge17] and [Du17], but since their algebraic group is denoted $G$ and $\Bbb G$ respectively, we felt free to stick to our good old $\GG$. I also thank Lucas Ruhstorfer for his careful reading, suggestions and references.

\bigskip

{ }

\bigskip

{ }

\bigskip

{ }
\bigskip

{ }

\centerline{\bf I. DEFINING CHARACTERISTIC}



{}
\bigskip

We first construct the finite groups $\GG^F$ that will be the main subject of this survey. Symmetric groups are also evoked in Sections 5.E-G and 10. The groups $\GF$ are commonly called finite groups of Lie type or finite reductive groups. In order to simplify the exposition we will not try to cover the Ree and Suzuki groups, nor speak of finite BN-pairs. We will even sometimes assume that $F$ induces no permutation of the roots (``untwisted groups") and refer to the bibliography for the original theorems in their full generality.


{}

\section{$p$-local subgroups and parabolic subgroups}

The groups and subgroups we will study are defined as follows (see [CaEn], [Carter2], [DigneMic], [MalleTe], [Sri], [Spr]).

Let $p$ be a prime and $\bbF\deq\ovF_p$ the algebraic closure of the field with $p$ elements.

Let $\GG$ be a connected algebraic group over $\bbF$.\index{$\GG$} We assume that it is defined over a finite subfield $\bbF_q$ ($q$ a power of $p$) thus singling out a Frobenius endomorphism $F\co\GG\to\GG$.\index{$F$} The group of fixed points $$\GF =\{g\in \GG\mid F(g)=g \}$$ is a finite group.

\begin{rem}\label{ReRe} Our way of defining things may be less concrete than saying that  $\GG$ is a subgroup of some $\GL_n(\bbF)$ ($n\geq 1$) defined by polynomial equations (on the matrix entries) with all coefficients in the finite subfield $\bbF_q$. This is indeed equivalent to the definition we gave, but the more intrinsic definition is generally preferred and also leads to a more compact notation. Subgroups of $\GG$ that are $F$-stable are also very important. 
\end{rem}

\begin{exm}\label{ExRe}
(a) The group $\GL_n(\bbF)$ is such a group $\GG$. It is defined over any finite subfield and the map $F\co\GG\to\GG$ raising matrix entries to the $q$-th power gives $\GF=\GL_n(\bbF_q)$. Note that any element of $\GL_n(\bbF)$ has finite order and that the Jordan decomposition $g=g_ug_{ss}$ of matrices coincides with the decomposition $g=g_pg_\pp$ into $p$-part and $p'$-part. This also defines a notion of {\bf unipotent/semi-simple} elements\index{unipotent elements}\index{semi-simple elements} and {\bf Jordan decomposition}\index{Jordan decomposition} inside any algebraic group $\GG$ over $\bbF$.  

(b) The group $U_n(\bbF)$ consisting of upper triangular unipotent matrices is clearly defined over $\bbF_p$ and is stable under $F$ defined in (a). Note that any element of $U_n(\bbF)$ has finite order a power of $p$.

(c) The group $D_n(\bbF)\cong (\Fm)^n$ consists of invertible diagonal matrices. Every element there has order prime to $p$. 

(d) Groups of type $\GF$ are rarely finite simple groups. For instance, $\SL_n(\bbF_q)$ is such a group with $\GG=\SL_n(\bbF)$ but in general it is not possible to find a connected group $\GG$ such that $\GG^F$ is isomorphic to $\PSL_n(\bbF_q)$. Even factoring out the center of $\SL_n(\bbF)$ would produce a $\PGL_n(\bbF)$ whose subgroup of fixed points under $F$ is $\cong \PGL_n(\bbF_q)$, a non-simple group! But realizing $\SL_n(\bbF_q)$, a perfect central extension of our simple group, is preferable for our representation theoretic purposes. Of course any representation of $\PSL_n(\bbF_q)$ identifies with a representation of $\SL_n(\bbF_q)$ trivial on its center.
\end{exm}

The unipotent radical $\RU(\HH)$\index{$\RU(\GG)$} of an algebraic group $\HH$ is the maximal connected normal unipotent subgroup of $\HH$ (see [Hum1, 19.5]).

The groups $\GG$ we study are assumed to be {\bf reductive}\index{reductive group}, i.e. $\RU(\GG)=\{1\}$. This implies essentially that $\Z(\GG)$ consists of semi-simple elements and $\GG_\ad\deq \GG/\Z(\GG)$ is a direct product of abstract simple groups [MalleTe, \S 8.4]. The factors in the direct product are in fact taken in a list obtained by the classification of simple algebraic groups, due to Chevalley and depending on root systems in the usual list.

We now introduce some subgroups of fundamental importance.

\subsection{Parabolic subgroups and Levi subgroups: reductive groups.}

Each group $\GG$ as above contains closed subgroups $\BB=\RU(\BB)\TT\geq\TT$ called a {\bf Borel subgroup} and a {\bf maximal torus}\index{$\BB$}\index{$\TT$}. Borel means connected solvable and maximal as such\index{Borel subgroup}. Torus means isomorphic to some $D_n(\bbF)$ ($n\geq 0$) as in Example~\ref{ExRe}.(c)\index{maximal torus}\index{torus}. Moreover the normalizer of $\TT$ is such that the Weyl group $$W_\GG(\TT)\deq \norm{\GG}{\TT}/\TT$$ is finite\index{Weyl group} and $$S\deq\{s\in W_\GG(\TT)\mid \BB\cup \BB s\BB\text{ is a subgroup }  \}$$ generates $W_\GG(\TT)$. When $w\in W_\GG(\TT)$ the expression $\BB w\BB$ means the set of products $b_1xb_2$ with $b_i\in\BB$ and $x\in w$ where the latter is a class mod $\TT$. Since $\TT\leq\BB$, $\BB w\BB$ is a single double coset with regard to $\BB$. The pair $(W_\GG(\TT),S)$ \index{$S$}satisfies the axioms of {\bf Coxeter systems} (see {[Hum2]}).

One has the {\bf Bruhat decomposition}\index{Bruhat decomposition} \begin{equation}\label{Bruhat}
\GG =\bigcup_{w\in W_\GG(\TT)}\BB w\BB \text{\ \ \ (a disjoint union).}
\end{equation} \def\sing#1{{\bf #1}\index{#1}}

One classically defines the \sing{root system} $\Phi(\GG,\TT)$\index{$\Phi(\GG,\TT)$} as a finite subset of the $\bbZ$-lattice $X(\TT )\deq \Hom (\TT ,\Fm)$\index{$X(\TT)$} (algebraic morphisms). It is stable under the action of $W_\GG(\TT)$. The actual definition of roots refers to the Lie algebra of $\GG$ and roots also define certain unipotent subgroups. One has a family of so-called \sing{root subgroups} $$(\XX_\al)_{\al\in\Phi(\GG,\TT)}$$ \index{$\XX_\al$}ranging over all minimal connected unipotent subgroups of $\GG$ normalized by $\TT$.

One has $^w\XX_\al =\XX_{w(\al)}$ for any $\al\in\PGT$, $w\in\WGT$.

 There is a basis of the root system $\Delta\inn\Phi(\GG,\TT)$ \index{$\Delta$}of cardinality the rank of $\Phi (\GG,\TT)$ in $\bbR\otimes _\bbZ X(\TT)$. One has $\Phi (\GG,\TT)=\Phi (\GG,\TT)^+\sqcup\Phi (\GG,\TT)^-$\index{$\Phi (\GG,\TT)^+$} where $\Phi (\GG,\TT)^+=\Phi (\GG,\TT)\cap\bbR^+\Delta$, $\Phi (\GG,\TT)^-=-\Phi (\GG,\TT)^+$.

One has $\XX_\al\leq \BB$ if and only if $\al\in\Phi (\GG,\TT)^+$ (which defines $\Phi (\GG,\TT)^+$ and therefore $\Delta$ from $\BB$).

When $\al\in\PGT^+$, $s\in S$, the condition $\XX_{-\al}\leq \BB\cup\BB s\BB$ implies $\al\in\Delta$. This establishes a bijection \begin{equation}\label{deltas}
\delta\co S\to \Delta\ \ \ \ s\mapsto \dd_s.
\end{equation}

The \sing{commutator formula} in a simplified version says the following for any linearly independent $\al,\beta\in\PGT$ 
\begin{equation}\label{ComFla}
[\XX_\al ,\XX_\beta]\leq\spann <\XX_{i\al+j\beta}\mid i,j\in\bbZ_{>0}\ ,\ i\al+j\beta\in\PGT >.\end{equation}

One calls \sing{parabolic subgroups} of $\GG$ the ones containing a conjugate of $\BB$. Denoting $W_I\deq\spann <I>\leq\WGT$ for $I\inn S$\index{$W_I$}\index{$\PP_I$}, the subgroups of $\GG$ containing $\BB$ are in bijection with subsets of $S$ by the map \begin{equation}\label{PI}I\mapsto \PP_I\deq \BB W_I\BB.\end{equation}

Note that $\PP_\emptyset =\BB$, $\PP_S=\GG$.

One has a semi-direct decomposition called the \sing{Levi decomposition}
\begin{equation}\label{Levi}\PP_I=\RU(\PP_I)\rtimes \LL_I\end{equation} where $\LL_I\deq\TT \spann <\XX_\al\mid\al\in \PGT\cap\bbR\delta(I)>$\index{$\LL_I$}\index{$\UU_I$}, a reductive group with same maximal torus as $\GG$, Borel subgroup $\BB\cap\LL_I$ and root system $\PGT\cap\bbR\delta(I)$.

One denotes $\UU_I\deq \RU(\PP_I)=\spann <\XX_\al\mid \al\in\PGT^+,\ \al\not\in \bbR\delta(I) >$.

\begin{exm}\label{GL}
	The case of $\GG =\GL_n(\bbF)$. Then $\BB=\TT\UU$ is the group of upper triangular matrices, $\UU =\RU(\BB)$ the group of upper unipotent matrices (see Example~\ref{ExRe}.(b)). It is not difficult to see that $\norm{\GG}{\TT} $ is the subgroup of monomial matrices (each row and column has a single non-zero entry) and $\norm{\GG}{\TT} /\TT$ identifies with the subgroup of permutation matrices $\cong\Sym n$, where $S$ corresponds to the set of transpositions of consecutive integers $\{s_1\deq (1,2), \dots ,s_{n-1}\deq(n-1,n)\}$. The roots $\PGT =\{\al_{(i,j)}\mid 1\leq i,j\leq n,\ i\not=j \}$ are defined as elements of $X(\TT)$ by \begin{equation}\label{GLal}\al_{(i,j)}\co\TT\to\Fm\ ,\  \diag (t_1,\dots ,t_n) \mapsto t_it_j^\mm .\end{equation} The elements of $\PGT^+$, resp. $\Delta$, are defined by the condition $i<j$, resp. $j=i+1$. 
	
	When $\al\in\PGT$ corresponds to $(i,j)$ then $\XX_\al$ is the subgroup of matrices $\id_n+\la E_{i,j}$ ($\la\in\bbF$) where $E_{i,j}$ is the elementary matrix with 1 as $(i,j)$ entry and 0 elsewhere. 
	
	If $I\inn S$, let us write $S\setminus I=\{s_{n_1}, s_{n_1+n_2}, \dots , s_{n_1+n_2+\cdots +n_{k-1} } \}$ with $n_1,n_2,\dots ,n_{k-1}\geq 1$ and define $n_k=n-(n_1+n_2+\cdots +n_{k-1})$. Then 
the elements of $\PP_I=\UU_I\LL_I$ decompose as
	$$ \begin{array}{c}
	n_1\  \\ n_2\  \\ \vdots \\ n_k\  
	\end{array}  \left( \begin{array}{cccc}
*	&\multicolumn{1}{|c}{*}&*&* \\ \cline{1-2}
0	&\multicolumn{1}{|c}{*}&\multicolumn{1}{|c}{*}&*  \\ \cline{2-2}
0	&0&\ddots & \\ \cline{4-4}
0	&0&0&\multicolumn{1}{|c}{*}	\end{array} \right)\ =\ \left( \begin{array}{cccc}
\id_{n_1}	&\multicolumn{1}{|c}{*}&*&* \\ \cline{1-2}
0	&\multicolumn{1}{|c}{\id_{n_2}}&\multicolumn{1}{|c}{*}&*  \\ \cline{2-2}
0	&0&\ddots & \\ \cline{4-4}
0	&0&0&\multicolumn{1}{|c}{\id_{n_k}}	\end{array} \right)
\left( \begin{array}{cccc}
*	&\multicolumn{1}{|c}{0}&0&0 \\ \cline{1-2}
0	&\multicolumn{1}{|c}{*}&\multicolumn{1}{|c}{0}&0  \\ \cline{2-2}
0	&0&\ddots & \\ \cline{4-4}
0	&0&0&\multicolumn{1}{|c}{*}	\end{array} \right)
$$

Note that $\LL_I\cong \GL_{n_1}(\bbF)\times \GL_{n_2}(\bbF)\times \cdots \times \GL_{n_k}(\bbF)$.
\end{exm}

\subsection{Parabolic subgroups and Levi subgroups: finite groups}

All the above can be taken $F$-stable: $F(\BB)=\BB$, $F(\TT)=\TT$. Then one denotes $B=\BB^F$, $T=\TT^F$, $N=\norm\GG\TT^F$ and $W=N/T=(W_\GG(\TT))^F$\index{$B$}\index{$T$}\index{$N$}\index{$\ov S$}. The later is generated by the set $$\ov S\deq\{w_\omega\mid \omega\in S/\spann <F>   \}\longleftrightarrow S/\spann <F> $$ where $\omega$ ranges over $F$-orbits in $S$ and if $I\inn S$, $w_I$ denotes the element of maximal $S$-length in $W_I$. From (\ref{Bruhat}) one gets a Bruhat decomposition  
\begin{equation}\label{fBru}
\text{  $    G=\bigcup_{w\in W}BwB  $\ ,\ \  a disjoint union. }
\end{equation} For $\ov J\inn \ov S$ corresponding to an $F$-stable subset $J\inn S$, the subgroups $\PP_J$, $\LL_J$ are $F$-stable, $P_{\ov J}\deq\PP^F_J= BW_J^FB$ and $L_{\ov J}\deq \LL^F_J$. Moreover $U_{\ov J}\deq \UU_J^F=\OP{P_{\ov J}}$. One has $$P_{\ov J}=U_{\ov J}\rtimes L_{\ov J}.$$ The roots are also acted upon by $F$ and the quotient set $\Phi (G,T)$ has properties similar to $\PGT$. Similar ideas allow to associate to them $p$-subgroups $(X_\al)_{\al\in \Phi (G,T)}$ that satisfy consequently an analogue of the commutator formula (\ref{ComFla}) seen above.

The relevance to simple groups starts with the following (see [MalleTe, 12.5]).

\begin{thm}\label{Tits} Assume $\ov S$ has no non-trivial partition into commuting subsets. Assume $G$ is perfect (i.e. $[G,G]=G$). Then $G/\Z(G)$ is simple.
\end{thm}

Recall that a \sing{quasi-simple group} is a finite perfect group $H$ such that $H/\zent H$ is simple. A \sing{universal covering} of a simple group $V$ is a quasi-simple group $H$ of maximal order such that $H/\zent H\cong V$. 

The \sing{classification of finite simple groups (CFSG)} (see [GoLySo], [Asch, \S 47]) tells us that finite non-abelian simple groups are either \begin{enumerate}[$\bullet$]
	\item alternating groups $\Alt_n$ ($n\geq 5$),
	\item groups of Lie type $\GF/\zent{\GF}$ as above, \index{finite groups of Lie type}
	\item or among the 26 so-called {\bf sporadic groups}.
\end{enumerate}

 Remarkably enough, simple groups of Lie type have universal coverings that are of type $\GF$ (short of 17 exceptions, see [GoLySo, 6.1.3]).

When dealing with finite groups $\GF$, an important tool is \sing{Lang's theorem}. It tells us that if $\CC$ is a connected closed $F$-stable subgroup of $\GG$, then $x\mapsto x^\mm F(x)$ is surjective from $\CC$ to itself.

\subsection{$p$-local subgroups and simple groups of characteristic $p$ type}

The proof of the classification of finite simple groups makes crucial use of the notions of $2$-local subgroups and of simple groups of characteristic $2$ type, this last one to separate simple groups of even and odd characteristic. The notions have also been defined for any prime (see [Asch, Ch. 48]). 

We fix here a prime $p$.

\begin{defn}\label{pLoc} Let $H$ be a finite group. A $p$-local subgroup \index{local subgroup} of $H$ is any normalizer $\norm HQ$ where $1\not= Q\leq H$ is a non trivial $p$-subgroup of $H$.
\end{defn}

\begin{defn}\label{radp} Let $p$ be a prime and $H$ a finite group. A radical $p$-subgroup \index{radical subgroup} of $H$ is any $p$-subgroup $Q$ of $H$ such that $$Q=\OP{\norm HQ}.$$
\end{defn}

Note that a Sylow $p$-subgroup of $H$ is always $p$-radical.

\begin{exm}\label{UIs} Let $G=\GF$ as in the last section, let $I\inn \ov S$. Then $U_I$ defined in 1.B satisfies $\norm G{U_I}=P_I$ and $U_I=\OP{P_I}$. Both properties are a consequence of the commutator formula. This proves that the $U_I$'s are $p$-radical subgroups.
\end{exm}

\begin{pro}\label{Maxp} The maximal $p$-local subgroups of a finite group $H$ satisfying $\OP{H}=\{1\}$ are normalizers of radical $p$-subgroups.
\end{pro}

\begin{proof} For any subgroup $M\leq H$, we clearly have $\norm H{\OP M}\geq M$. Applying this to our maximal $p$-local subgroup $M$ we get that either $M=\norm H{\OP M}$ or $\OP M\lhd H$ and therefore $\OP M=\{1\}$. But the second case is impossible by the definition of $p$-local subgroups. 
\end{proof}

\begin{defn}\label{chap} Let $H$ be a finite {\bf simple} group and $p$ be a prime. Then $H$ is said to be of characteristic $p$ type\index{characteristic $p$ type}, if and only if 
	\begin{equation}\label{char}
	\text{  $   \cent X {\OP X}\leq \OP X $ }
	\end{equation} for any $p$-local subgroup $X$ of $H$. 
	
	This is equivalent to (\ref{char}) holding for any {\it maximal} $p$-local subgroup.
\end{defn}

The second statement in the above definition is a non trivial one. We refer to the proof of [Asch, 31.16], using among other things Thompson's A$\times$B lemma.

One will of course check here that our groups $\GF$ give rise to simple groups of characteristic $p$ type, see [Asch, 47.8.(3)]. 

We take $G=\GF$ as in Sect. 1.B above. Recall the subgroups $P_I=U_I\rtimes L_I$ for $I\inn \ov S$.

\begin{thm}[{[BoTi65]}]\label{BTits}  \begin{enumerate}[\rm(1)]
		\item The $p$-radical subgroups of $G$ are the $\{{}^g(U_I)\mid g\in G,\ I\inn \ov S  \}$ with $U_I=\OP {P_I}$, $P_I=\norm G{U_I}$.
		\item If $g\in G$ and $I,J\inn\ov S$ are such that $^g(U_I)=U_J$, then $I=J$ and $g\in P_I$.
		\item If $\ov S\supsetneq I$ and $G/\Z (G)$ is simple then $\cent G{U_I}\leq \Z (G) U_I$.
	\end{enumerate}
\end{thm}

\begin{cor}\label{charp} If $G/\Z(G) $ is simple then it has characteristic $p$ type.
\end{cor}

We finish this subsection by giving some ideas in the proof of 
 Theorem~\ref{BTits}.
First the theorem has an equivalent in $\GG$ as follows (Platonov 1966, see [Hum1, 30.3]).

\begin{lem}\label{Plato} In $\GG$, if $V$ is a closed subgroup of $\UU$, then the sequence $\VV_0=V$, $\VV_{i+1}\deq \VV_i \RU(\norm \GG{\VV_i})$ is an ascending sequence stabilizing at some $\RU(\PP(V))$ where $\PP(V)$ is a parabolic subgroup of $\GG$. 
\end{lem}

Note that if $V$ is $F$-stable then all $\VV_i$'s and therefore $\PP(V)$ itself are $F$-stable. Once written as $^g\PP_I$ for $g\in \GG$ and $I\inn S$, using $F$-stability one gets $^{g^\mm F(g)}\PP_{F(I)}=\PP_I   $. By the argument we are going to use for (2) of the Theorem, this implies $F(I)=I$ and $g^\mm F(g)\in\PP_I$. Lang's theorem then allows to assume that $g=g'h$ where $g'\in G$ and $h\in \PP_I$, so that $\PP(V)={}^{g'}\PP_I$ with $F(I)=I$ and $g'\in G$.

Assume moreover $V$ $p$-radical in $G$. The inclusions $V\leq \RU(\PP(V))$ and $\norm \GG V\leq \PP(V)$ imply $\norm{\RU(\PP(V))^F}V\lhd\norm GV$. But $\RU(\PP(V))^F$ is a  $p$-subgroup of $G$, so $p$-radicality of $V$ implies $\norm{\RU(\PP(V))^F}V =V$. But $V\leq {\RU(\PP(V))^F}$ is an inclusion of $p$-groups so we must have indeed $V = {\RU(\PP(V))^F}$. Using the above this gives $V={}^{g'}U_{\ov I}$, hence the claim (1).

For the claim (2), writing $g\in BwB$ thanks to Bruhat decomposition (\ref{fBru}) and using that $B$ normalizes both $U_I$ and $U_J$, one finds that $^wU_I=U_J$. Assume for simplicity that $F$ acts trivially on $W(\GG,\TT)$ and $S$. Our equality implies on roots that $$w(\PGT^+\setminus \PGT^+_I)\inn \PGT^+ .$$ But a basic property of Weyl groups acting on roots tells us that any element of $W(\GG,\TT)$ decomposes as $w=w'v$ where $v\in \spann <I>$ and $w'(\PGT^+_I)\inn\PGT^+$. But then $w'(\PGT^+)=\PGT^+$, therefore $w'=1$, $w=v\in \spann <I>$ and $g\in P_I$.

(3) Using (\ref{fBru}) again and arguing on roots it is easy to show that $\cent G{U_I}\leq B$. We then check that under our assumptions, $\cent B{U_I}\leq \Z (G)$. We show it for $I=\emptyset$ and refer to [Asch, 47.5.3] for the general case. Given the semi-direct product structure $B=U\rtimes T$ with $U$ the Sylow $p$-subgroup of $B$, it is not difficult to see that our claim about $\cent BU$ reduces to the inclusion $\cent TU\leq \Z(G)$. For $s\in \ov S$, let $C_s=\cent{L_s}{U_s}$. It is normalized by $X_s$, $U_s$ (hence $U$), but also by $s$ and we have seen $C_s\leq B$. So $$C_s\leq L_s\cap B\cap B^s=L_s\cap TU_s=T.$$ So $C_s=\cent T{U_s}$ normalizes  $U$, hence centralizes it since $U\cap C_s=\{1\}$. So $C_s=\cent TU$. We deduce that $\cent TU$ is normalized by any $s\in S$ and by $T$, hence by $N$. On the other hand $B=TU\leq \cent G{\cent TU}$, so the latter is a parabolic subgroup normalized by $N$, hence normal in $G$. By our hypothesis, it has to equal $G$, hence the inclusion $\cent TU\leq \Z(G)$.

\qed

\subsection{Consequences on fusion}

The problem of $p$-fusion in finite groups is as follows. Let $Q$ be a Sylow $p$-subgroup of a finite group $H$. One wants to know when subsets of $Q$ can be $H$-conjugate. 

More generally one defines the ``fusion system" \index{fusion system}$\cF_Q(H)$ as follows

\begin{defn}[{[AschKeOl, I.1.1]}]\label{FusCat} For $Q$ a Sylow $p$-subgroup of $H$, the fusion system $\cF_Q(H)$ has objects the subgroups of $Q$ and if $Q_1,Q_2\leq Q$, one defines $$\Hom_{\cF_Q(H)}(Q_1,Q_2)\inn\Hom(Q_1,Q_2),$$ the former consisting of maps \begin{eqnarray*}
		\ad_{h,Q_1,Q_2}\ \colon Q_1&\to&Q_2  \\
		x&\mapsto&hxh^\mm \end{eqnarray*}  for $h\in H$ with $^hQ_1\leq Q_2$.
\end{defn}

A theorem by Alperin (1967) first showed that this category is generated by certain elementary operations, see [Asch, 38.1].

A \sing{tame intersection} of Sylow $p$-subgroups of $H$ is a $p$-subgroup of type $Q_1\cap Q_2$ with $Q_1,Q_2$ both Sylow $p$-subgroups of $H$ and $\norm{Q_1}{Q_1\cap Q_2}$, $\norm{Q_2}{Q_1\cap Q_2}$ both Sylow $p$-subgroups of $\norm{H}{Q_1\cap Q_2}$. 
\begin{thm}[{[Al67]}]\label{AlpFus} Let $h\in H$ and $A\inn Q$ such that $A^h\inn Q$. Then there exist Sylow $p$-subgroups $Q_1,\dots,Q_n$ and elements $h_i\in \norm H{Q\cap Q_i}$ for $i=1,\dots ,n-1$ such that 
	\begin{enumerate}[\rm(i)]
		\item $h=h_1\dots h_{n-1}$,
		\item for any $i=1,\dots ,n$, $Q\cap Q_i$ is a tame intersection, and
		\item $A\inn Q\cap Q_1$, $A^{h_1} \inn Q\cap Q_2$, $\dots$ , $A^{h_1\dots h_{n-1}} \inn Q\cap Q_n$.
	\end{enumerate}
\end{thm}

This can be summed up as saying that normalizers of tame intersections $Q\cap Q'$ ($Q'$ another Sylow $p$-subgroup of $H$) generate $\cF_Q(H)$.

\begin{rem}\label{Tame}
	A tame intersection $Q_1\cap Q_2$ of Sylow subgroups is a $p$-radical subgroup (see Definition~\ref{radp}). Indeed $\OP {\norm H{Q_1\cap Q_2}}$ is included in both $\norm{Q_1}{Q_1\cap Q_2}$ and $\norm{Q_2}{Q_1\cap Q_2}$ by the tame intersection hypothesis, so included in $Q_1\cap Q_2$.
	In the case of groups $G=\GF$ it means that they are $G$-conjugates of subgroups $U_I$ ($I\inn \ov S$) thanks to \Th{BTits}.
\end{rem}

Alperin's theorem has been strengthened by Goldschmidt so as to find a minimal family of normalizers of so-called {\it essential } $p$-subgroups \index{essential $p$-subgroups} (see [AschKeOl, I.3.2])  which generates $\cF_Q(H)$. In the case of 
groups $G=\GF$, it gives the following [Puig76, Appendix I]. Recall that $U\deq\UU^F$ is a Sylow $p$-subgroup of $G$.

\begin{thm}[{}]\label{Pu76} 
The fusion system $\cF_{U}(G)$ is generated by minimal parabolic subgroups $P_{\{s\}}=B\cup BsB$ for $s$ ranging over $\ov S$.
\end{thm}
\subsection{Consequences on $p$-blocks}

We show that the condition of being of characteristic $p$ type has strong consequences on the $p$-blocks of our simple group.

\medskip

\noindent {\bf Blocks and the Brauer morphism.} Let us recall what are $p$-blocks \index{block}of a finite group $H$. 

We keep $\bbF$ the algebraic closure of $\bbF_p$ and consider the group algebra $\FH$. As for any finite dimensional algebra over a field, one has a maximal decomposition \begin{equation}\label{Blo}\FH\cong B_0\times B_1\times \cdots \times B_\nu\end{equation} as a direct product of $\bbF$-algebras. The corresponding two-sided ideals $B_i$ of $\FH$ are called the {\bf $p$-blocks} of $H$, one denotes $\Bl(H)=\{B_0, B_1,\dots ,B_\nu\}$. The unit element $b_i$ of each $B_i$ is a primitive idempotent of the center $\Z(\FH)$ and one has a bijection between $\Bl(H)$ and the primitive idempotents of $\Z(\FH)$ since any such idempotent $b$ defines the block $\FH b$. Any $\FH$-module $M$ decomposes as $M=\oplus_iB_iM$ as $FH$-module. So each indecomposable module has only one block acting non-trivially on it. This induces a partition $\IBr (H)=\sqcup_{i=0}^\nu \IBr (B_i)$ of the isomorphism classes of simple $\FH$-modules. The \sing{principal block} $B_0(H)$ is by definition the one corresponding to the trivial $\FH$-module\index{trivial module}, i.e. the line $\bbF$ with $H$ acting as identity, often denoted as $\bbF$ or $1$.

When $Q$ is a $p$-subgroup of $H$, the {\bf Brauer morphism }\index{Brauer morphism}\begin{align}\label{Brmo}
	\Br_Q\co\Z(\FH)&\to\Z(\bbF\cent HQ)\\
	\sum_{h\in H}\la_hh&\mapsto \sum_{h\in \cent HQ}\la_hh
\end{align} is a morphism of commutative algebras. The {\bf defect groups} \index{defect group}of a given block $B_i$ are the $p$-subgroups $Q$ of $H$ maximal for the property that $\Br_Q(b_i)\not= 0$. For a given $B_i$ they form a single $H$-conjugacy class. For $D\leq H$ a given $p$-subgroup of $H$, one denotes $\Bl(H\mid D)$ \index{$\Bl(H\mid D)$} the subset of $\Bl(H)$ consisting of blocks admitting $D$ as defect group.

The {\bf principal block} has defect group any Sylow $p$-subgroup of $H$. 

A block $B_i$ has defect group $\{1\}$ if and only if $B_i$ is a semi-simple algebra (in fact simple with $|\IBr(B_i)|=1$), this is called {\bf a block of defect zero}\index{defect zero} (defect was first defined as a numerical invariant corresponding to the exponent $d$ such that $|D|=p^d$).

Brauer's first and third so-called Main Theorems are as follows\index{Brauer's first Main Theorem}\index{Brauer's third Main Theorem}. One keeps $H$ a finite group.

\begin{thm}\label{BrThs} Let $Q$ be a $p$-subgroup of $H$.
	
	(i) The Brauer morphism $\Br_Q$ induces bijections $$\Bl(H\mid Q)\longleftrightarrow \Bl(\norm HQ\mid Q)\longleftrightarrow \Bl(Q\cent HQ\mid Q)/\norm HQ\text{-}\conj .$$
	
	(ii) Through the above, the principal blocks of $H$, $\norm HQ$ and $\cent HQ$ correspond.
	
\end{thm}

\medskip

\noindent {\bf Blocks in the defining characteristic.} Let us return to our finite reductive groups $G=\GF$, or more generally simple groups of characteristic $p$ type (see 1.C).

\begin{pro}[Dagger-Humphreys, see {[Hum3, \S 8.5]}]\label{blochp}
	Assume $H$ is a finite simple group of characteristic $p$ type. Then the non principal $p$-blocks of $H$ all have defect $\{1\}$.
\end{pro}
 \begin{proof}
 	Let $D$ be a defect group $\not=\{1\}$ of a $p$-block $B$ of $H$. By (i) of the above theorem, $\Bl(D\cent HD\mid D)\not=\emptyset$. The condition that $H$ has characteristic $p$ type implies that $\cent X{\OP{X}}\leq\OP{X}$ for $X=\norm HD$. But we clearly have $$D\cent HD\leq \OP X\cent X{\OP X},$$ so $D\cent HD$ is a $p$-group. A $p$-group has only one simple module over $\bbF$ (see [Benson, 3.14.1]), hence only one $p$-block, the principal block. So by (ii) of the above theorem, $B$ is the principal block of $H$.
 	\end{proof}

We will see later that blocks of defect $\{1\}$ actually exist in this case, corresponding to the so-called Steinberg module, see \Th{StM}.



{}
\bigskip

{}

\section{Yokonuma-Hecke algebras in characteristic $p$}

Iwahori-Hecke algebras are algebras similar to the group algebras of Coxeter groups $(W,S)$, only the quadratic relations $s^2=1$ ($s\in S$) have been there replaced by an equation $$s^2=(q-1)s+q$$ where $q$ is a scalar. See [GePf00]. This models the endomorphism algebras of induced representations $\Ind^G_B1$ for $G=\GF$ as before with Weyl group $W$ and $q$ the order of root subgroups $X_\al$.
  
\sing{Yokonuma-Hecke algebras} are a bit larger and serve as model for the endomorphism algebra of induced representations $\Ind^G_U1$.

\subsection{Self-injective endomorphism algebras, a theorem of Green}

We first present a general theorem of J.A. Green about certain $A$-modules where $A$ is a finite dimensional $\bbF$-algebra. Green's theorem shows that if $Y$ is an $A$-module such that $\End_A(Y)$ is self-injective then $\End_A(Y)$-modules give a lot of information on $A$-modules, in particular the simple submodules of $Y$. This will be applied to $A=\bbF G$ ($\bbF$ and $G$ as in Sect. 1.B), $Y=\Ind ^G_U\bbF$, so that $\End_{\bbF G}(\Ind ^G_U\bbF)$ is a Yokonuma-Hecke algebra in characteristic $p$.

In the following $Y$ is a finitely generated left module over the finite dimensional $\bbF$-algebra $A$, $E\deq \End_A(Y)^\op$ and one considers the functor $H_Y$\index{$H_Y$} sending an $A$-module $M$ to the $E$-module $H_Y(M)\deq \Hom_A(Y,M)$, $E$ acting through composition by elements of $\End_A(Y)$ on the right. Note that $H_Y(Y)=E$ the regular left module.

\begin{thm}\label{Green}
	Assume \begin{enumerate}[\rm(1)]
		\item there is a linear map $\la\co E\to\bbF$ such that for any $x\in E$, $x\not=0$, one has $\la(xE)\not=0\not=\la(Ex)$, and
		\item any simple $A$-module is both a submodule and a quotient of $Y$.
	\end{enumerate} Then the functor $H_Y$ sends simple $A$-modules to simple $E$-modules and this induces a bijection between isomorphism types of simple modules for both algebras.
\end{thm}

The relevance of self-injectivity (implied by the slightly stronger hypothesis (1) above, see [Benson, \S 1.6]) essentially lies in the following lemma where we keep the same assumptions.

\begin{lem}\label{VY} Let $V\inn E=\End_A(Y)$ an $E$-submodule. Denote $V.Y\deq \sum_{f\in V}f(Y)\inn Y$. Then $H_Y(V.Y)=V$ by taking the image of the latter inclusion. 
\end{lem}

\begin{proof}  We clearly have $V\inn \Hom_A(Y,V.Y)=H_Y(V.Y)$ as subspaces of $\Hom_A(Y,Y)=H_Y(Y)$. If the inclusion is strict, since it is an inclusion of $E$-modules, there is $$V\subsetneq U\inn H_Y(V.Y)\inn E ,$$ $E$-modules with $U/V$ simple. Hypothesis (1) implies that projective and injective $E$-modules coincide (see [NagaoTsu, 2.8.11]), so any finitely generated module injects into a free one and any simple module into the regular one. So we have a map $\phi\co U\to E$ of $E$-modules such that $\phi(U)\not= 0=\phi(V)$. By injectivity of the regular module $E$ the map $\phi$ extends into $\wh\phi\co E\to E$.
	
\centerline{	$\xymatrix{0\ar[r]&U\ar[d]_\phi\ar@{^{(}->}[r]&E\ar@{.>}[dl]^{\wh\phi}\\&E&}$} 

Such a map $\wh\phi$ is clearly of the form $H_Y(\phi ')$ for some $\phi '\co Y\to Y$ a map of $A$-modules. Now $\wh\phi (V)=0\not=\wh\phi(U)$ implies $\phi'(V.Y)=0\not=\phi'(U.Y)$. But on the other hand $V.Y= U.Y$ since $V.Y\inn U.Y\inn H_Y(V.Y)\inn V.Y$. This contradiction finishes the proof.
\end{proof}

The proof of the Theorem goes as follows.
Let $M$ be a simple $A$-module. Then $M$ is a factor and a submodule of $Y$ by (2), so $E=H_Y(Y)\nni H_Y(M)=\Hom_A(Y,M)\not= 0$. Let now $0\not=V\subsetneq H_Y(M)\inn E$ an $E$-submodule. By simplicity of $M$, $V.Y=M$, but the lemma tells us that $V=H_Y(V.Y)=H_Y(M)$. This shows that $H_Y(M)$ is simple. 

Moreover, every simple $E$-module $V$ is obtained that way since $V$ embeds in $E=H_Y(Y)$ as seen before, thus allowing to form $V.Y$ and the Lemma gives $V=H_Y(V.Y)$. If $M$ is a simple submodule of $V.Y$, then $$0\not= H_Y(M)\inn H_Y(V.Y)=V$$ so indeed $V=H_Y(M)$.

Eventually, if $M,M'$ are simple $A$-modules and $H_Y(M)\cong H_Y(M')$, then $M$ and $M'$ can be assumed to be submodules of $Y$, so that $H_Y(M)$ and $ H_Y(M')$ are seen as submodules of $E$. Now the isomorphism $H_Y(M')\to H_Y(M)$ extends to some map $E\to E$ that writes $H_Y(\phi)$ for $\phi\co Y\to Y$. The restriction of $\phi$ to $M$ gives a non zero map $M\to M'$, and therefore $M\cong M'$.
	
\centerline{	$\xymatrix{0\ar[r]&H_Y(M)\ar[d]_{\not=0}\ar@{^{(}->}[r]&E=H_Y(Y)\ar@{.>}[d]^{}\\0\ar[r]&H_Y(M')\ar@{^{(}->}[r]&E=H_Y(Y)}$} 

\begin{exm}\label{Hec}
Assume now that $H$ is a finite group and $X$ a subgroup, let $k$ be any commutative ring. The $kH$-module $Y=\Ind_X^Hk =kH\otimes_Xk$ is the permutation module on the set of cosets $\{hX\mid h\in H  \} $. Denote $\omega\in Y$ the element corresponding to the coset $X$ or $1\otimes 1\in kH\otimes_Xk$. If $M$ is a $kH$-module, one denotes by $M^X$ the space of fixed points under $X$. By Frobenius reciprocity, one can identify explicitly $$ \Hom_{kH}(Y,M)\xrightarrow{\sim}M^X\ ,\ f\mapsto f(\omega).$$ 

This can serve first to give a basis of $\End_{kH}(Y)\cong Y^X$ as a vector space. One has $\End_{kH}(Y)=\oplus_{n\in X\backslash H/X}k.a'_n$ where $a'_n$ \index{$a'_n$} is the $kH$-linear map $Y\to Y$ defined by \begin{equation}
\label{a's}
a'_{XhX}(\omega)=\sum_{y\in XhX\omega}y=\sum_{x\in X/X\cap {}^hX }xh\omega .\end{equation}

One has $\End_{kH}(Y)\cong \End_{kH}(Y)^\op$ by $a'_n\mapsto a_n\deq a'_{n^\mm}$.

Moreover, through the identification above the action of $a_{XhX}$ on $ M^X$ is by \begin{equation}
\label{aact}m\mapsto a_{XhX}(m)=\sum_{x\in X/X\cap X^h}xh^\mm m.
\end{equation}

\end{exm}

\subsection{Yokonuma-Hecke algebras: a presentation}

As said before we will apply Green's theorem to $A=\bbF G$, $Y=\Ind^G_U\bbF$ in the notations of Sect. 1.B. We recall the subgroups $B=\BB^F$, $U=\OP{B}$, $T=\TT^F$, $W=W(\GG,\TT)^F$, etc.. In order to simplify a bit we assume that $F$ acts trivially on $W(\GG,\TT)$, so that also $S=\ov S$.

One writes $U_-=U^{w_S}$ where $w_S$ is the longest element of $W$ with regard to $S$.

\begin{defn}\label{an's} Let $\HGU =\End_{kG}(\Ind^G_U\bbF)$\index{$\HGU $}.
	For $n\in N$, let $a_n\co Y\to Y$ defined by $a_n(1\otimes 1)=\sum_{u\in U\cap U_-^n}un^\mm\otimes 1$.
\end{defn}

For $s\in S$ corresponding to some $\dd\in \Delta$ through (\ref{deltas}), one defines $$T_s\deq T\cap \spann <X_\dd , X_{-\dd}>$$ and one can find some representative \index{$\dot s$}$$\dot s\in N\cap X_\dd X_{-\dd}X_\dd$$ (see [CaEn, 6.3.(i)]). Moreover for $s_1,s_2\in S$ and $r$ the order of the product $s_1s_2$ in $W$,  one has \begin{equation}\label{dotsi}
\dot s_1\dot s_2\dots =\dot s_2\dot s_1\dots (r \textrm{ terms on each side}).\end{equation}

\begin{thm}\label{Yoko} Let $n,n'\in N$, $s\in S$.
	\begin{enumerate}[\rm(1)]
		\item  $a_na_{n'}=a_{nn'}$ as soon as $l_S(nn'T)=l_S(nT)+l_S(n'T)$.
		\item The $a_t$'s for $t\in T$ generate a semi-simple subalgebra $\cong \bbF T$.
		\item $(a_{\dot s})^2= -|T_s|^\mm a_{\dot s}\sum_{t\in T_s}a_t$.
		\item $\HGU$ is presented as an algebra by symbols $a_n$ ($n\in N$) subject to the relations (1) and (3) above.
	\end{enumerate}
\end{thm}

\begin{proof} (1)
		 The additivity of lengths implies that $U\cap U_-^{nn'}=U\cap U_-^{n'}.(U\cap U_-^{n})^{n'}$ with uniqueness by considerations on roots. Note that this is the same argument as for the corresponding equation in Iwahori-Hecke algebras $\End_{\bbC G}(\Ind^G_B1)$. The equality $a_na_{n'}=a_{nn'}$ then follows by the definition of the $a_n$'s.
		 
	(2) Clear from the first point, noting that $T$ has order prime to $p$.
		
		(3) Let $\dd\in \Delta$ correspond to $s$ by (\ref{deltas}), so that $X_\dd =U\cap U_-^s$. The Bruhat decomposition  (\ref{fBru}) in $L_s$ implies that $\spann<X_\dd ,X_{-\dd}> =X_\dd T_s\cup X_\dd T_s\dot s X_\dd$. For $v\in X_\dd\setminus\{1\}$ one denotes $t(v)\in T_s$ such that $$ \dot s^\mm v\dot s^\mm \in X_\dd t(v)^\mm \dot s^\mm X_\dd .$$ From Definition~\ref{an's}, one gets clearly $$(a_{\dot s})^2(1\otimes 1)=\sum_{u\in X_\dd}\dot s^{-2}u\otimes 1+\sum_{u\in X_\dd\setminus\{1\}}a_{\dot st(u)}(1\otimes 1).$$ The first term is 0 since each $u$ acts trivially on $1\otimes_U 1$. The second term gives what is claimed once we check that the cardinality $|\dot sX_\dd\dot s\cap X_\dd\dot s tX_\dd|$ is the same for any $t\in T_s$. This is an easy check in the group $\spann <X_\dd , X_{-\dd}>$ which in our hypotheses is a quotient of $\cong \SL_2(q)$.
		
		(4) The proof is similar to the one for Iwahori-Hecke algebras [CurtisRei, \S 67].

	\end{proof}

\subsection{Yokonuma-Hecke algebras: simple modules}

\begin{pro}\label{lform} Let $n_S$ be an element of $N$ whose class mod $T$ is the element $w_S\in W$ of largest $S$-length. Let $$\la\co\HGU\to \bbF$$ the $\bbF$-linear map sending $a_{n_S}$ to 1 and $a_n$ for $n\in N$, $n\not=n_S$ to 0. Then $\la$ vanishes on no non-zero left or right ideal of $\HGU$. 
\end{pro}

\begin{proof}
	From Theorem~\ref{Yoko} it is clear that when $n,n'\in N$, the product $a_na_{n'}$ is always in $a_{nn'}+\sum_{n''}\bbF a_{n''}$ where the sum is over $n''\in N$ with $l_S(n''T)<l_S(nT)+l_S(n'T)$. Now if $0\not= x=\sum_{n\in N}\mu_na_n$ with $\mu_n\in \bbF$, let $n_0$ be such that $\mu_{n_0}\not=0$, with $l_S(n_0T)$ maximal as such. Then $a_{n_0}a_{n_0^\mm n_S}=a_{n_Sn_0^\mm}a_{n_0}=a_{n_S}$ and therefore $\la (xa_{n_0^\mm n_S})=\la(a_{n_Sn_0^\mm}x)=\mu_{n_0}\not=0$.
\end{proof}

\begin{defn}\label{adm} For $\th \co T\to \Fm$ a group morphism, let \index{$S_\th$}$S_\th\deq \{s\in S\mid \th(T_s)=1  \}$. One calls \sing{admissible pair} any pair $(\th ,I)$ where $\th{\in\Hom (T,\Fm)}$ and $I\inn S_\th$.
\end{defn}

\begin{thm}\label{simH} The simple $\HGU$-modules are one-dimensional. Seen as maps $\HGU\to\bbF$, they are of the form $\psi_{(\th ,I)}$ where $(\th ,I)$ is an admissible pair and $\psi_{(\th ,I)}$ is defined by \begin{enumerate}[\rm(a)]
		\item $\psi_{(\th ,I)}(a_t)=\th (t)$ for any $t\in T$
		\item $\psi_{(\th ,I)}(a_{\dot s})=-1$ for $s\in I$, 0 otherwise.
	\end{enumerate} 
\end{thm}

\begin{proof}
Let $V$ be a simple $\HGU$-module. The subalgebra $\oplus_{t\in T}\bbF a_t$ being commutative, semi-simple with $\bbF$ algebraically closed, $V$ decomposes as a sum of lines stable under the $a_t$'s. Let $L\inn V$ be such a line and $n_0\in N$ such that $a_{n_0}.L\not= 0$ and $l_S(n_0T)$ is maximal as such. One shows that $\bbF a_{n_0}.L$ is stable under $\HGU$. For $t\in T$, one has $a_t a_{n_0}.L =a_{tn_0}.L=a_{n_0}a_{t^{n_0}}.L=a_{n_0}.L$. For $s\in S$, if $l_S(sn_0T)=l_S(n_0T)+1$ then Theorem~\ref{Yoko}.(1) and maximality of $n_0$ imply $a_{\dot s}a_{n_0}L=a_{\dot s n_0}L=0$. If $l_S(sn_0T)=l_S(n_0T)-1$ then Theorem~\ref{Yoko}.(1) and (3) imply $a_{\dot s}a_{n_0}=a_{\dot s}a_{\dot s}a_{\dot s^\mm n_0}\in a_{n_0}(\oplus_{t\in T}a_t)$ hence $a_{\dot s}a_{n_0}L\inn a_{n_0}L$. We get our claim by noting that the $a_t$'s and the $a_{\dot s}$'s generate $\HGU$ by Theorem~\ref{Yoko}.(4).

The form of the $\bbF$-algebra morphisms $\HGU\to \bbF$ is easy to deduce from Theorem~\ref{Yoko}.(4).
\end{proof}



{}
\bigskip

{}

\section{Simple $\bbF G$-modules and $p$-blocks }

As announced before, we now apply \Th{Green} and the information gathered on $\HGU$ to simple $\bbF \GF$-module. This theory is due to J.A. Green (see [Gre78], [Tin79], [Tin80]). This provided a more conceptual framework to a classification of simple modules of split BN-pairs due to Curtis-Richen [Cu70], [Ri69] (see also [CaLu74]).

Among other properties of the simple $\bbF \GF$-modules we show the existence of a block of defect zero (see Proposition~\ref{blochp}) associated to the so-called Steinberg module.

The notations are the same as in the preceding chapter.

\subsection{Simple $\bbF G$-modules, the Steinberg module}

\begin{thm}\label{MU}
For any simple $\bbF G$-module $M$, the subspace of fixed points $M^U$ is a line. Moreover $M=\bbF U_-.M^U$. 
\end{thm}

\begin{thm}\label{MthI}
There is a bijection between the isomorphism types of simple $\bbF G$-modules and the set of admissible pairs (see Definition~\ref{adm}). 

Let the simple $\bbF G$-module $M$ correspond to the pair $(\th ,I)$ then \begin{enumerate}[\rm(i)]
	\item $T$ acts by $\th$ on the line $M^U$
	\item for $s\in S$ associated with $\dd\in \Delta$ and $m\in M^U$ one has $\sum_{u\in X_\dd}u\dot s.m=-m$ if $s\in I$, 0 if $s\in S\setminus I$.
	
	\item (Smith 1982 [Sm82]) if $J\inn S$ and $U_JL_J$ is the Levi decomposition of the parabolic subgroup $P_J$, then $M^{U_J}$ is a simple $\bbF L_J$-module associated with the admissible pair $(\th , J\cap I)$ of $L_J$.
\end{enumerate}
\end{thm}

\begin{thm}\label{StM}
Keep $M$ a simple $\bbF G$-module associated with the admissible pair $(\th ,I)$. Then $M$ is projective if and only if $I=S$.

If moreover $G$ is perfect and $\ov G\deq G/\zent G$ is simple, then $\ov G$ has only two $p$-blocks\begin{itemize}
	\item the principal block
	\item the block whose unique simple module is the $\bbF \ov G$-module corresponding to the simple $\bbF  G$-module associated with the admissible pair $(1,S)$.
\end{itemize}
\end{thm}

The projective simple module associated with the admissible pair $(1,S)$ is called the \sing{Steinberg module}.

\noindent{\it Proof of the Theorems.}
	One applies Theorem~\ref{Green} to $A=\bbF G$, $Y=\Ind_U^G\bbF$. The condition (1) of the theorem is satisfied by Proposition~\ref{lform}. The condition (2) comes from the fact that if $M$ is simple $\Hom_{\bbF G}(Y,M)\cong M^U\not= 0$ since $U$ is a $p$-group and $Y$ is isomorphic with its $\bbF$-dual. The first statement of Theorem~\ref{MU} along with Theorem~\ref{MthI}.(i) and (ii) then come from Theorem~\ref{Green} and the explicit description we made of the functor $H_Y$ in our case, see (\ref{aact}).  
	
	For the equality $M=\bbF U_-.M^U$, it is enough to show that $\bbF U_-.M^U$ is stable under $G$. The latter is generated by $U_-$, $T$ and $S$, so we just check stability under $T$ and $S$. First $\bbF U_-.M^U$ is $T$-stable since $T$ normalizes $U$ and $U_-$. Let now $s\in S$ corresponding to $\dd\in\Delta$. Then $$sU_-M^U=s(U_-\cap U_-^s)X_{-\dd}M^U\inn U_-sX_{-\dd}M^U$$ so it suffices to show that $sxM^U\inn \bbF U_-M^U$ for any $x\in X_{-\dd}$. Using the Bruhat decomposition in $L_s=TX_\dd \cup X_\dd sTX_\dd$, one gets $x\in X_\dd sTX_\dd$ or  $x=1$. In the first case $$sxM^U\inn X_{-\dd}TX_\dd M^U=X_{-\dd}M^U\inn U_-M^U.$$ On the other hand $ (\sum_{x\in X_{-\dd}}x)sM^U\inn M^U$ by (ii) of Theorem~\ref{MthI}, so the case $x=1$ can be deduced from the first case just treated.
	
	(iii) The weaker statement that $\bbF L_J.M^U$ is a simple $\bbF L_J$-module is enough for our purpose. Indeed if $M'\inn \bbF L_J.M^U$ is a simple $\bbF L_J$-submodule, it has fixed points under the $p$-subgroup $U\cap L_J$, but $M'\inn M^{U_J}$ since $U_J$ is normalized by $L_J$, so $M'{}^{U\cap L_J}\inn M^{U_J(U\cap L_J)} = M^U$ by the Levi decomposition. Since $M^U$ is a line, $M'$ must contain it and therefore $M'\nni \bbF L_J.M^U$. 
	
On \Th{StM}. $M$ is projective if and only if its restriction to the Sylow
	 $p$-subgroup $U_-$ is projective (see for instance [Benson, 3.6.9] or (\ref{Vertexp}) below), i.e. free as an $\bbF U_-$-module. By
	  Theorem~\ref{MU}, the restriction to $U_-$ is $\bbF U_-/V$ where $V=\{v\in \bbF U_-\mid v.M^U=0 \}$, and $\bbF U_-/V$ 
	  is free if and only if $V=0$. Since the only simple submodule of $\bbF U_-$ is the line $\bbF \si$ where $\si =\sum_{u\in 
	  	U_-}u$, one gets that $M$ is projective if and only if $(\sum_{u\in U_-}u)M^U\not= 0$. Let $n_0\deq \dot s_1\dot s_2\dots \dot s_l$ where $s_1s_2\dots s_l$ is a decomposition of 
  	the $l_S$-longest element of $W$. By the Theorem~\ref{MthI}.(ii) giving the action of $\HGU$ on the line $M^U$, one has  \begin{equation}\label{an0}
  n_0^\mm (\sum_{u\in U_-}u)M^U=(\sum_{u\in U_-}u) n_0^\mm 
  	M^U=\psi_{\th ,I} (a_{n_0})M^U=\big(\prod_{i=1}^{l}\psi_{\th ,I}
  	(a_{\dot s})\big)M^U. 	\end{equation}	By the definition of $\psi_{\th ,I}$ and since $\{s_1,\dots , s_l\}=S$, one has 
  $\prod_{i=1}^{l}\psi_{\th ,I}(a_{\dot s})\not= 0$ if and only if $I=S$.

Now concerning the $p$-blocks. We know from Proposition~\ref{blochp} and Corollary~\ref{charp} that $\ov G$ has only one block of defect $\not=\{1\}$. Let's see that there is only one block of defect $\{1\}$. This would correspond to a simple projective $\bbF \ov G$-module. On the other hand $\zent G$ is a $\pp$-group (since $U\cap U_-=\{1\}$), so we get an $\bbF G$-module simple and projective. By Theorem~\ref{MthI} it has to be associated with an admissible pair $(\th ,S)$. It is not too difficult to check that when $G$ is perfect the condition $S_\th =S$ implies $\th =1$ (show first that the associated $M$ is one-dimensional as a consequence of Theorem~\ref{MthI}.(iii)).
\qed

\subsection{Relation with weight modules}

Assume that our pair $(\GG,F)$ from Sect. 1.B is such that $F$ acts trivially on the Weyl group $\norm{\GG}{\TT}/\TT$. The simple $\bbF\GF$-modules have been classified in the following way by R. Steinberg in the 60's (see [St63], [Hum3, \S 2.11]). First the irreducible rational representations $$\GG\to\GL_n(\bbF)$$ are classified by the subset of so-called \sing{dominant  weights} $X(\TT)^+\inn X(\TT)$ where $\la\in X(\TT)$ is dominant if and only if $(\la, \dd^\vee)\geq 0$ for any fundamental coroot $\dd^\vee\in \Phi(\GG,\TT)^\vee$, $\dd\in\Delta$. Let us denote by $\MM(\la)$ the corresponding $\GG$-module.

Most of the features described in 3.A above are also present regarding the rational modules $\MM(\la)$. The link between the two situations is provided by the following.

\begin{thm}\label{Stein} The $\MM(\la)$ for $\la$ such that $0\leq (\la, \dd^\vee)\leq q-1$ have irreducible restrictions to $\GF$. This gives all simple $\bbF\GF$-modules only once when $\GG=[\GG,\GG]$.
\end{thm}

Among the properties of the $\MM(\la)$'s is the fact that $\MM(\la)$ has a line of fixed points under $\UU$. The torus $\TT$ acts by $\la$ on that line, and one proves easily the following relation with the description given before

\begin{prop}[{}]\label{lath} The admissible pair associated to $\Res^\GG_\GF\MM(\la)$ is $(\th, I)$ where $\th =\Res^\TT_{\TT^F} \la$ and $I\inn S$ is in bijection by (\ref{deltas}) with the fundamental roots $\delta$ such that $(\la, \delta^\vee)=q-1$.
\end{prop}

\subsection{Alperin weight conjecture}

\bigskip

For $\bbF$ an algebraic closure of $\bbF_p$ and $H$ a finite group, let us recall $\Alp (H)$ \index{$\Alp (H)$} \index{Alperin weight conjecture}the set of $H$-conjugacy classes of pairs $(Q,\pi)$ where $Q$ is a $p$-subgroup of $H$ and $\pi$ is a simple projective $\bbF (\norm HQ/Q)$-module.
Alperin conjectured that 
\begin{equation}\label{AWC}
\text{  $    |\Alp (H)|$ equals the number of simple $\bbF H$-modules }
\end{equation} for all finite groups $H$ and primes $p$ [Al87].

\smallskip

We take $G=\GG^F$ as before.

\begin{thm}\label{AlpG} $G$ satisfies Alperin's weight conjecture (\ref{AWC}) for the prime $p$. Moreover there is a map $M\mapsto (Q,\pi)$ inducing a bijection $\IBr (G)\to \Alp(G)$ and such that \begin{enumerate}[\rm(i)]
		\item the bijection is $\Aut (G)$-equivariant
	\item $\pi$, seen as an $\bbF\norm GQ$-module, is a submodule of $M^Q$.
\end{enumerate}
\end{thm}

\begin{proof} Let us first note that for any $(Q,\pi)$ in $\Alp (G)$, the subgroup $Q$ is $p$-radical, or equivalently that $\OP L=\{1\}$ for $L\deq \norm GQ/Q$. Indeed $L$ has an $\bbF L$-module $\pi$ that is simple and projective. Then $\pi^{\OP L}$ is a non trivial $\bbF L$-submodule, so $\OP L$ acts trivially on $\pi$. On the other hand $\pi$ remains projective when restricted to $\OP L$, so it is a free $\bbF \OP L$-module. This is possible only if $\OP L=\{1\}$.
	
	Now by Theorem~\ref{BTits}.(1) and (2), $\Alp (G)$ is in bijection with the set of pairs $(U_I,\pi)$ where $I\inn S$ and $\pi$ is an $\bbF L_I$-module that is simple and projective. From Theorem~\ref{MthI} we know that any such module is associated to an admissible pair $(\th ,I)$ of $L_I$, i.e., $\th\co T\to\Fm$ has to be such that $\th (T_\dd)=\{1\}$ for any $\dd\in \Delta $ corresponding to an element of $I$. 
	
	By Theorem~\ref{MthI} we then see that our sets $\IBr(G)$ and $\Alp(G)$ are both in bijection with the admissible pairs for $G$. 
	
	Moreover if $M\leftrightarrow (\th ,I)\leftrightarrow (Q,\pi)=(U_I, M_{L_I}(\th ,I))$, one has that $Q$ is the smallest $p$-radical subgroup of $G$ normal in $U$ such that $\bbF \norm G{Q}M^U$ is a simple projective module for $\norm GQ/Q$ (use Theorem~\ref{MthI}.(iii)). Then $\pi =\bbF \norm G{Q}M^U$. This intrinsic definition of the map shows that it is equivariant for automorphisms that preserve $U$. The latter being a Sylow $p$-subgroup and inner automorphisms acting trivially on both $\IBr(G)$ and $\Alp(G)$, we actually get equivariance for $\Aut(G)$.
	\end{proof}

\begin{rem}\label{remA} (a) Let us recall Green's notion of \sing{vertex} of an indecomposable $kH$-module $M$ (see [NagaoTsu, \S 4.3]). It is a subgroup $V$ of $H$ minimal for the property that $M$ is isomorphic to a direct summand of $\Ind_V^G\Res_G^VM$. It is easy to see that of \begin{equation}\label{Vertexp}
	\text{  $  M  $ is a direct summand of $\Ind_V^G\Res_G^VM$ for $V$ a Sylow $p$-subgroup of $H$. }
	\end{equation} Consequently the vertex of an indecomposable module is always a $p$-subgroup. 
	
On the other hand, Alperin's conjecture asks for associating to each simple $\bbF H$-module a conjugacy class of $p$-subgroups of $H$. When $H$ is $p$-solvable, the vertex of the given simple $\bbF H$-module provides such a correspondence (see [Oku81, 4.1], [IsNa95]). This can't be a solution for our groups $G=\GF$ since there, when $G/\Z(G)$ is simple non abelian, vertices of simple modules are Sylow $p$-subgroups except for simple projective modules (Dipper, see [Di80], [Di83]).
		
		(b) In this case of the defining characteristic a more suggestive definition for $Q$ associated with a simple $\bbF G$-module $M$ is as follows. The module $\Ind^G_U\bbF$ decomposes as a sum $Y_1\oplus \dots \oplus Y_v$ where the $Y_i$ are indecomposable. Then $M$ is the quotient of a single $Y_{i_0}$, and $Q$ is the vertex of that $Y_{i_0}$.
		
		(c) In our case, all $p$-radical subgroups of $G$ are present in $\Alp(G)$. Whether this is a general fact relates strongly with the question of quasi-simple groups having blocks of central defect (case of $Q=\OP{G}$). Quasi-simple groups $\GF$ have such blocks for all primes. For the primes 2 and 3, there are infinitely many alternating groups $\Alt_n$ without block of defect zero. For all that, see below \Th{10.1}.

\end{rem} 
\bigskip

{ }
 
\bigskip

{ }

\bigskip

{ }



\centerline{\bf II. NON-DEFINING CHARACTERISTIC ($\ell \not= p$)}



{}
\bigskip

{}

\section{Rational series and $\ell$-blocks}

From now on we will be looking at modular aspects of the representations of our groups $\GF$ (see Sect. 1.B) with regard to a prime $\ell$ different from the defining prime $p$. So we assume essentially that a prime $\ell\not= p$ has been chosen and also that we have a so-called \sing{$\ell$-modular system} $(\OO ,K,k)$ where 
$\OO$ is a complete valuation ring with $\ell\in J(\OO)$, $K$ its fraction field, $k=\OO/J(\OO)$. One assumes that $\OO$ contains $|H|$-th roots of 1 for any finite group $H$ we encounter (see [NagaoTsu, \S 3.6]). 

\subsection{Blocks and central functions}

Let us recall the notion of $\ell$-blocks of a finite group $H$ as decomposing the group algebra $\ov k H=B_0\times B_1\times \cdots \times B_\nu$ as in (\ref{Blo}) where $B_i=\ov k Hb_i$ with $b_i$ primitive idempotents of the center $\Z ({\ov k H})$. The field $k$ having enough roots of unity with regard to $H$, the $b_i$ belong in fact to $k H$ (one can also impose that $k=\ov k$, see [AschKeOl, Sect. IV.4.2]). On the other hand it is easy to see that reduction mod $J(\OO)$ induces an epimorphism
\begin{equation}\label{ilift}
\Z ({ \OO H})\rightarrow\Z ({ k H})\end{equation}
 which implies through the lifting of idempotents that the blocks of $\OO H$ and the $p$-blocks of $H$ identify by $\OO H e_i\mapsto kHb_i$ where $b_i$ is the reduction mod $J(\OO)$ of $e_i$. Now the $p$-blocks of $H$ also induce a partition of $\Irr (H)$, whose elements are seen as isomorphism types of simple $K H$-modules, namely $\Irr(H)=\sqcup_i\Irr (B_i)$ \index{$\Irr (B)$}corresponding to the decomposition $KH=\prod_i KHe_i$.
We will also need to look at character values, that is we see the elements of $\Irr (H)$ as central functions $H\to K$. Noting that the elements of $\Irr(H)$ take values in the subring of $K$ generated by the $|H|$-th roots of 1, we even see $\Irr (H)$ as a $\bbC$-basis of $\CF (H)$ \index{$\CF (H)$} the complex vector space of central functions $H\to \bbC$, hence with the decomposition \begin{equation}\label{blCF}
\CF(H)=\oplus_i\CF (H\mid B_i).\end{equation} Each element $\chi\in\Irr (H)$ defines some central idempotent $e_\chi\deq {\chi(1)\over |H|}\sum_{h\in H}\chi (h^\mm)h\in \Z(K H)$\index{$e_\chi$} and \begin{equation}\label{echi}
e_i=\sum_{\chi\in\Irr(B_i)}e_\chi .\end{equation}

We are later interested in ``union of blocks" in $\Irr(H)$. It is an easy exercice to prove the following.

\begin{pro}\label{regB} Let $A\inn\Irr(H)$. The three statements below are equivalent. \begin{enumerate}[\rm(i)]
		\item $A$  is a union of subsets $\Irr (B_i)$.
		\item $\sum_{\chi\in A}^{}e_\chi\in \OO H$.
		\item The projection $\pr_A\co\CF(H)\to\CF(H)$ associated to $A$, sends the regular character $\reg_H$\index{$\reg_H$} to a central function with values in $|H|\OO =|H|_\ell\OO$, namely \begin{equation}
		\pr_A(\reg_H)(h)\in |H|_\ell\OO \text{ for all }h\in H.
		\end{equation}
	\end{enumerate} 
\end{pro}

\subsection{Uniform and $p$-constant functions}

We return to our groups $G=\GF$ keeping the same notation except for the basic pair $\TT\leq\BB$ of $F$-stable maximal torus and Borel, that we rename $\TT_0\leq\BB_0$ since we will now allow our maximal tori (even when $F$-stable) {\bf not} to be included in $F$-stable Borel subgroups.  

We give some elements of Deligne-Lusztig's theory on $\Irr(\GF)$ in a very quick fashion. We refer to the contributions by Geck and Dudas for a more in-depth introduction (see [Ge17], [Du17]). 

There are basically two very important facts about ordinary characters of finite groups of Lie type. The {\bf functors} $\Lu\LL\GG$ \index{$\Lu\LL\GG$}and the existence of {\bf rational series}. A third basic feature -- a so-called {\bf Jordan decomposition of characters} -- will be seen later (see Sect. 4.D).

The functor $\Lu\LL\GG\co \bbZ\Irr(\LF)\to \bbZ\Irr(\GF)$ is defined as follows. 

One takes $\PP =\LL \RU(\PP)$ a Levi decomposition of a parabolic subgroup of $\GG$ and one assumes that $\LL$ (and not necessarily $\PP$) is $F$-stable. Then the variety 
\begin{equation}\label{YPP}
\text{  $    \YY_\PP = \{g\RU(\PP)\mid g^\mm F(g)\in \RU(\PP)F(\RU(\PP)) \}$\index{$    \YY_\PP$} }
\end{equation} is clearly acted on by $\LF$ on the right and $\GF$ on the left. Understandably any cohomology theory of that object would produce modules acted on by those finite groups on those sides. One denotes by $\rH^i_c(\YY_\PP)$\index{$\rH^i_c(\YY_\PP)$} the $i$-th cohomology group defined by $\ell$-adic cohomology with compact support of $\YY_\PP$ tensored by $\bbC$ (see more details in Sect. 9.A below), so as to give a $\bbC \GF\times \LF^\op$-module. 

\begin{defn}\label{RLG}
	One defines $\Lu{\LL}{\GG}\co \bbZ\Irr (\LF)\to \bbZ\Irr (\GF)$ by $$[M]\mapsto \sum_{i\in \bbZ}^{}(-1)^i[\rH^i_c(\YY_\PP)\otimes_{\bbC\LF} M]$$ where $[M]$ is the class of a $\bbC \LF$-module $M$ in $\bbN\Irr (\LF)$. One denotes by $$\slu{\LL}{\GG}\co \bbZ\Irr (\GF)\to \bbZ\Irr (\LF)$$ \index{$\slu{\LL}\GG$}the adjoint map for the usual scalar product of central functions.
\end{defn}

 \begin{rem}\label{4.4} (a) The maps defined above are independent of the choice of $\PP$ for a given $\LL$ in many cases in particular when $\LL$ is a torus (see [DigneMic, Ch. 11]). The independence in the general case relates with the validity of a reasonable Mackey formula similar to the one known for induction/restriction of characters of finite groups. The most complete result about such a formula is due to Bonnaf\'e-Michel [BoMi11] (see also [Tay17]).
 		
 		(b) When $\PP$ is $F$-stable, $\YY_\PP$ identifies with the finite set $\GF/\RU(\PP)^F$ and the functor $\Lu{\LL}{\GG}$ with the so-called parabolic or \sing{Harish-Chandra induction}, making any representation of $\LF$ into a representation of $\PP^F$ (through trivial action of $\RU(\PP)^F$) then applying ordinary induction $\Ind^\GF_{\PP^F}$. The formula 
 		\begin{equation}\label{D_G}
 		\text{  $  D_G=\sum_{I\inn S}(-1)^{|I|}\Lu{\LL_I}{\GG}\slu{\LL_I}{\GG}  $ }
 		\end{equation}
 involves only Harish-Chandra induction and is called \sing{Alvis-Curtis duality}. 
 
 (c) In order to see how many more functors Deligne-Lusztig theory constructs as opposed to usual functors defined from subgroups of $\GF$, let us focus on the case of $\LL$ a torus. In the finite group $G=\GF$ there is only one conjugacy class of subgroups one would call a maximal torus, the one denoted $T=\TT_0^F$ at the beginning of this section. Allowing any group $\TT^F$ for $\TT$ an $F$-stable maximal torus brings a lot more to the picture. Starting from our reference $\TT_0$, the $\GF$-conjugacy classes of $F$-stable maximal tori are in bijection \begin{equation}\label{FTori}
 g\TT_0g^\mm\mapsto g^\mm F(g)\TT_0\in W(\GG,\TT_0)
 \end{equation}with $W=W(\GG,\TT_0)$ mod the relation $w\sim_F vwF(v)^\mm$ for any $w,v\in W$. The element $g^\mm F(g)\TT_0$, or its $\sim_F$-class is called the {\bf type} of the $F$-stable maximal torus $g\TT_0g^\mm$\index{type of maximal torus}. This is a classical consequence of Lang's theorem (see [Ge17, 2.3]). For $\GL_n(\bbF)$ with $F$ the usual Frobenius on matrix entries, this gives as many classes of tori as the number of conjugacy classes of $\Sym n$.
  
 \end{rem}

 An important fixed point property of \'etale cohomology implies the following \sing{character formula} ([DL76, \S 3-4]).

\begin{pro}\label{chfla} If $(u,v)\in \GF\times \LF$ is assumed to be unipotent, let $$Q_\LL^\GG(u,v)=\sum_{i\in \bbZ}^{}(-1)^i \tr ((u,v), \rH^i_c(\YY_\PP)).$$
	
	If $su$ is the Jordan decomposition of an element of $\GF$ and $f\in\CF(\LF)$, then $$ \Lu{\LL}{\GG}(f)(su)=|\LF|^\mm |\ccent\GG s^F|^\mm \sum_{\{g\in\GF\mid s\in{}^g\LL\} } |\ccent{^g\LL} s^F|\sum_{v\in \ccent{^g\LL} s^F_u|}Q_{\ccent{^g\LL} s}^{\ccent{\GG} s}(u,v^\mm){}f(g^\mm svg).$$
\end{pro}

One lists below some consequences of the character formula that prove useful in the proof of Brou\'e-Michel's theorem on $\ell$-blocks.

\begin{defn}\label{unif}
	We call \sing{uniform functions} the elements of $\CF(\GF)$ that are $\bbC$-linear combinations of $\Lu{\TT}{\GG}\th$ for $\TT$ an $F$-stable maximal torus and $\th\in\Irr(\TT^F)$.
	
	Some $f\in\CF(\GF)$ is called \sing{$p$-constant} if and only if $f(su)=f(s )$ for any Jordan decomposition $su\in G$.
\end{defn}

\begin{lem}\label{omni} Let $f\in\CF(G)$ be $p$-constant. \begin{enumerate}[\rm(i)]
		\item $f$ is uniform.
		\item If $\LL$ is an $F$-stable Levi subgroup of $\GG$ and $f'\in\CF(\LF)$, then $\Lu{\LL}{\GG}(f)f'=\Lu{\LL}{\GG}(f\Res^\GF_\LF f')$.
		\item If $\chi\in\CF(G)$ then $D_G(f\chi)=fD_G(\chi)$.
		\item $|G|_\pp^\mm D_G(\reg_G)$ is the characteristic function of the set of unipotent elements (i.e. $p$-elements) of $G$.
	\end{enumerate}
\end{lem}

\subsection{Rational series and Brou\'e-Michel's theorem}
 
 An important case of the functor $\Lu{\LL}{\GG}$ is when $\LL$ is a maximal torus. It allows an important partition of $\Irr(\GF)$. 
 
 One defines first a group $\GG^*$ dual to $\GG$. This means that a choice has been made of maximal tori $\TT_0\leq\GG$, $\TT_0^*\leq\GG^*$ such that $\Hom (\TT_0,\Fm)\cong \Hom (\Fm ,\TT_0^*)$ and $\Hom (\TT_0^*,\Fm)\cong \Hom (\Fm ,\TT_0)$ in a way that is compatible with roots and coroots. One also assumes that all those are stable under compatible Frobenius endomorphisms $F\co\GG\to\GG$, $F^*\co\GG^*\to\GG^*$.
 The interest of groups in duality is in the parametrization of characters of finite tori $\TT^F$. Indeed for $w\in W(\GG,\TT_0)$ identified with $w^*\in W(\GG^*,\TT_0^*)$, one has isomorphisms $\Irr(\TT_0^{wF})\cong {\TT_0^*}^{w^*F^*}$ (where the notation $wF$ stands for $F$ followed by conjugation by $w$).
 
 One also gets a bijection  
 
 	 $\{ \GF$-conjugacy classes of pairs $(\TT ,\th)$ where $\TT$ is an $F$-stable maximal torus of $\GG$ and $\th\in\Irr(\TT^F)\}$ 
 	 \begin{equation}\label{dualT}
 	 \updownarrow
 	 \end{equation}

 $\{\GG^*{}^{F^*}$-classes of pairs $(\TT^*, s)$ where $\TT^*$ is an $F^*$-stable maximal torus of $\GG^*$ and $s\in\TT^*{} ^{F^*}\}$.

\begin{thm}[Deligne-Lusztig]\label{Lseries}
	For $s\in \GG^*{}^{F^*}$ a semi-simple element, one defines $\ser\GF s$ the set of irreducible components of generalized characters $\Lu{\TT}{\GG}\th$ for $(\TT ,\th)$ corresponding to some $(\TT^*,s)$ through the above correspondence.
		
		One gets a partition \begin{equation}\label{series}\Irr(\GF)=\sqcup_s\ser{\GF}{s}
		\end{equation} where $s$ ranges over semi-simple classes of $\GG^*{}^{F^*}$.
\end{thm}

The subsets $\ser{\GF}{s}$ for $s\in \GG^*_{\mathrm{ss}}{}^{F^*}$ are called the \sing{rational series} of $\Irr(\GF)$.
The proof of the Theorem, given in [DigneMic, Ch. 14] is quite indirect, going through the intermediate notion of geometric series and using a regular embedding $\GG\inn\w\GG$ (see [Ge17, \S 6], {[CaEn, \S 15.1]}) with connected $\zent{\w\GG}$.
It is easier to show that $\Lu{\LL}{\GG}$ sends $\ser{\LF}{s}$ into $\bbZ\ser{\GF}{s}$ via the correspondence between Levi subgroups of $\GG$ and $\GG^*$ [CaEn, 15.7]. This implies in particular that Alvis-Curtis duality (see Remark~\ref{4.4}.(b) above) satisfies \begin{equation}\label{Dser}
D_G(\ser Gs)\inn \bbZ\ser Gs \text{  for any semi-simple  } s\in \GG^*{}^{F^*}.
\end{equation}

\begin{thm}[Brou\'e-Michel]\label{BM90} 
	For $s\in \GG^*{}^{F^*}$ a semi-simple element of order prime to $\ell$, one defines $$\lser\GF s\deq \bigcup_{t\in\cent{\GG^*}{s}^{F^*}_\ell}\ser{\GF}{st}.$$ This is a union of $\ell$-blocks.\index{$\lser\GF s$}
\end{thm}

\begin{proof}
We abbreviate $\GG^F=G$. We show that the projection $\pr\co\CF(G)\to \CF(G)$ associated with the subset $\lser{G}{s}\inn\Irr(G)$, satisfies $\pr(\reg_G)(G)\inn|G|_\ell \OO$. This will give our claim by Proposition~\ref{regB}.
	\index{$\delta_\pi$}
	For $\pi$ a set of primes, one denotes by $\dd_\pi$ the characteristic function of $\pi$-elements of $G$. 
	
	Note that \begin{equation}\label{ddpi}
\dd_\pi\in \OO\Irr(G) \text{ as soon as  } \ell\in\pi
	\end{equation} thanks to a classical consequence of Brauer's characterization of generalized characters (see [NagaoTsu, 3.6.15.(iii)]). Note also that $\dd_\pi$ is $p$-constant as soon as $p\in\pi$.
	
	We also prove \begin{equation}\label{prdd}
\pr(	\dd_\lp f)=\dd_\lp\pr(f) \text{ for any uniform } f\in \CF(G). 
	\end{equation} 
	
	To show (\ref{prdd}) it suffices to show the equality with $f=\Lu{\TT}{\GG}\th$ for some $F$-stable maximal torus $\TT$ and $\th\in\Irr(\TT^F)$. The claim then reduces to showing that $\dd_\lp\Lu{\TT}{\GG}\th \in \bbC \lser G s$ when $(\TT ,\th)\leftrightarrow (\TT^*,s')$ (see (\ref{dualT})) for some $s'\in $ with $s'_\lp=s$. Note that this can be done for any semi-simple $\lp$-element, conjugate or not to the $s$ we are given. By Lemma~\ref{omni}.(ii), we have $\dd_\lp\Lu{\TT}{\GG}(\th) =\Lu{\TT}{\GG}(\dd_\lp^{(\TT^F)}\th)$. On the other hand $\dd_\lp^{(\TT^F)}\th =\sum_{\th '}^{}\th '$ where the sum is over $\th '\in \Irr(\TT^F)$ with $\th '_\lp =\th_\lp$ (we consider $\Irr(\TT^F)$ as a multiplicative group). But it is easy to check from the identifications of duality that if $s'_\lp =s$ then $(\TT ,\th ')\leftrightarrow (\TT^*, s '')$ with $s''_\lp =s$. This gives $\dd_\lp\Lu{\TT}{\GG}\th \in \bbC \lser G s$.
	
	Now the proof of the Theorem goes as follows. From Lemma~\ref{omni}.(i) we know that $\dd_{\{p,\ell \}}$ is uniform and (\ref{prdd}) gives now
	\begin{equation}\label{ddpr}
	\dd_\lp\pr (\dd_{\{p,\ell \}})=\pr (\dd_{\{p\}}).
	\end{equation}
	
	The image of the right hand side by $D_G$ is $D_G\circ\pr(\dd_p)=\pr\circ D_G(\dd_p) =|G|^\mm_\pp\pr(\reg_G)$ thanks to (\ref{Dser}) and Lemma~\ref{omni}.(iv). Using (iii) of the same lemma, the image by $D_G$ of (\ref{ddpr}) now gives
	\begin{equation}\label{dpr}
	|G|_\pp^\mm \pr (\reg_G)=\dd_\lp .\pr (D_G(\dd_{\{p,\ell \}})).
	\end{equation}
	We have seen (\ref{ddpi}) that $\dd_{\{p,\ell \}}$ hence also $D_G(\dd_{\{p,\ell \}})\in\OO\Irr(G)$. So the right hand side of (\ref{dpr}) takes values in $\OO$. Then indeed 
	$	\pr (\reg_G)$ takes values in $|G|_\pp\OO =|G|_\ell\OO$ as claimed.
\end{proof}

The sum of blocks of $\OO G$ corresponding to the theorem is denoted as follows.

\begin{defn} \label{BGs}
One denotes $e_\ell (\GF ,s)\deq\sum_{\chi\in\lser \GF s}^{}e_\chi\in \Z(\OO \GF )$, $\ov e_\ell (\GF ,s)$ its image in $\zent{k\GF }$, and $B_\ell (\GF ,s)\deq\OO \GF e_\ell (\GF ,s)$.  \index{$e_\ell (\GF ,s$}\index{$\ov e_\ell (\GF ,s$}\index{$B_\ell (\GF ,s$}
\end{defn}

Combining (\ref{prdd}) and the fact that multiplication by $\dd_\lp$ preserves $\ell$-blocks (see below Brauer's ``second Main Theorem") one easily gets
\begin{pro}[Hiss]\label{Hiss} 
For every $p$-block $B$ of $\GF $ such that $\Irr(B)\cap \lser \GF s\not=\emptyset$, one has $\Irr(B)\cap \ser \GF s\not=\emptyset$.
\end{pro}

\subsection{Jordan decomposition and $\ell$-blocks}

We keep $\GG ,F ,\GG^*, F^*$, etc... as before.

\begin{defn} \label{ublo}
The elements of $\ser{\GF}{1}$ are called \sing{unipotent characters}. Similarly $\ell$-blocks $B$ of $\GF$ such that $\Irr(B)\cap\ser{\GF}{1}\not=\emptyset$ are called \sing{unipotent blocks}.
\end{defn}

The set of unipotent characters tends to be sensitive only to the root system of $\GG$ and the action of $F$ on it. In particular one has bijections (see [DigneMic, 13.20])
\begin{equation}\label{uCh}
\text{  $ \ser{\GF}{1}\leftrightarrow \ser{[\GG ,\GG]^F}{1}\leftrightarrow \ser{(\GG/\Z(\GG))^F}{1}$ }
\end{equation} and $\ser{\GF}{1}\leftrightarrow\ser{\GD^\FD}{1}$ but also $\ser{\GF}{1}\leftrightarrow \ser{\GG^{F^m}}{1}$ ($m\geq 1$) when $F$ acts trivially on the root system of $\GG$. However the last bijection relates characters of different degree, though the degree is the same polynomial in various powers of $q$, see [Carter2, Sect. 13.8].

\begin{exm}\label{uGL}	We consider the case of $\GF=\GL_n(\bbF_q)$, see Example~\ref{GL}. For $w\in Sym n$ let $\TT_w$ denote an $F$-stable torus of type $w$ with regard to the diagonal torus in the sense of Remark~\ref{4.4}.c. The set of unipotent characters is in bijection with $\Irr(\Sym n)$ by the map $$\chi\mapsto R_\chi = n!^\mm \sum_{w\in\Sym n}{\chi(w)}\Lu{\TT_w}{\GG}1$$ which takes values in $\pm\ser{\GF}{1}$ and with adequate signs gives indeed a bijection $\Irr(\Sym n)\to\ser{\GF}{1}$ (see for instance [DigneMic, \S 15.4]).
\end{exm}

It is customary to call {\bf Jordan decomposition}\index{Jordan decomposition of characters} of $\Irr (\GF)$ any bijection $$\ser{\GF}{s}\leftrightarrow \ser{\cent{\GD}{s}^\FD}1$$ where $s\in (\GD)^\FD_\pp$. However the definition we gave of unipotent characters applies only to connected groups $\GG$, so the set $\ser{\cent{\GD}{s}^\FD}1$ above would be defined as the set of constituents of induced characters $\Ind^{\cent{\GD}{s}^\FD}_{\ccent{\GD}{s}^\FD}\zeta$ for $\zeta\in\ser{\ccent{\GD}{s}^\FD}1$. The existence of such a Jordan decomposition compatible with the $\Lu{\LL}{\GG}$ functors has been shown by Lusztig [Lu88], here again in a quite indirect way, the results being a lot more complete in the cases where $\Z (\GG)$ is connected, which in turn ensures that $\cent{\GD}{s}$ is then connected. A basic idea is that Jordan decomposition should behave like a $\Lu{\bf C}{\GG}$ functor for a suitable $\CC$.

For the following, see {[DigneMic, 13.25]}.

\begin{thm}[Lusztig] \label{LCGs} 
Assume $\LL^*$ is an $\FD$-stable Levi subgroup of $\GD$ such that $\cent{\GD}{s}\leq \LD$. Then $\LL \deq (\LD)^*$ can be seen as a Levi subgroup of $\GG$ and there is a sign $\eps_{\LL ,\GG}\in \{1,-1\}$ such that $\eps_{\LL,\GG}\Lu{\LL}{\GG}$ induces a bijection $$\eps_{\LL,\GG}\Lu{\LL}{\GG} \co\ser{\LF}{s}\to \ser{\GF}{s}.$$
\end{thm}

In this situation and with $s$ being an $\lp$-element it is not difficult to prove that $\eps_{\LL,\GG}\Lu{\LL}{\GG}$ also induces a bijection $$\eps_{\LL,\GG}\Lu{\LL}{\GG}\co\lser{\LF}{s}\to \lser{\GF}{s}.$$

\begin{thm} [Brou\'e] \label{broIso} 
The above bijection preserves the partitions induced by $\ell$-blocks. Moreover two $\ell$-blocks that thus correspond have defect groups of the same order.
\end{thm}

\noindent{\it About the proof.} We sketch the main ideas of the proof, based on Brou\'e's notion of perfect bi-characters {(see [Bro90a])}. For finite groups $H$, $L$, a bi-character $\mu\in\bbZ \Irr (H\times L)$ is called {\bf perfect}\index{perfect bi-character} if and only if for all $(h,l)\in H\times L$, $\mu (h,l)\in |\cent Hh|\OO\cap |\cent Ll|\OO$ and whenever $\mu(h,l)\not= 0$ then $h\in H_\lp$ if and only if $l\in L_\lp$.
	
This is a $\bbZ$-submodule of $\bbZ \Irr (H\times L)$ and the trace character of an $\OO H\times L^\op$-bimodule which is projective on each side is perfect. This last property is very important since it gives an arithmetic test for bicharacters that could come from a Morita equivalence of blocks over $\OO$ or even a derived equivalence since by a theorem of Rickard such equivalences are induced by complexes of bi-projective bimodules [Rick89]. 

Brou\'e shows that if \begin{enumerate}[a)] \item $I\inn\Irr(H)$ and $J\inn \Irr(L)$ are unions of $\ell$-blocks, with \item $\si\co J\to I$ a bijection and \item one has signs $(\eps_\chi)_{\chi\in J}$ such that $\sum_{\chi\in J} \eps_\chi\si(\chi)\otimes \ov\chi$ is perfect, \end{enumerate} 

 then \begin{enumerate}[\rm(i)] \item $\si$ preserves the partition of $I$ and $J$ into $\ell$-blocks and \item  corresponding blocks have defect groups of the same order and same number of simple modules over $k=\OO/J(\OO)$. \end{enumerate}
	
In order to apply this to our situation $H=\GF$, $L=\LF$, $I=\lser{\LF}{s}$ and the bijection of Theorem~\ref{LCGs}, it just remains to show that our bi-character is perfect. This is a consequence of what has been said about bi-projective modules producing perfect trace characters and the fact the action of $\GF\times \LF$ on the variety $\YY_\PP$ is free on each side. This last fact translates into a related property of $\ell$-adic cohomology groups as $\OO\GF\times \LF^\op$-modules, thanks to a result of Deligne-Lusztig ([DL76, 3.5]) which is also key to the proof of Proposition~\ref{chfla} above.
	 \qed
	 
It is important to notice that in the above the isometry of characters is without signs as would happen in the case of a Morita equivalence. Indeed Brou\'e made the conjecture that this corresponds to a Morita equivalence $$B_\ell (\LF ,s)\text{-}\mo \xrightarrow{\sim} B_\ell (\GF ,s)\text{-}\mo$$ between the module categories of $B_\ell (\GF ,s)$ and $B_\ell (\LF ,s)$.
	 
This was proved by Bonnaf\'e-Rouquier [BoRo03]. In Sect. 9 below, we try to give an idea of their proof which goes very deep into the definition of $\Lu{\LL}{\GG}$ functors. The result was completed recently by Bonnaf\'e-Dat-Rouquier [BoDaRo17] into a statement showing isomorphism of defect groups and local structure.
	 
In order to get a Morita equivalence from an isometry of characters, one needs essentially to have the latter induced by a bi-projective module thanks to the following 
	 
\begin{lem}[{Brou\'e [Bro90b]}]\label{K2O} 
Assume $H$, $L$ are finite groups, $B$, $C$ are sums of blocks of $\OO H$, $\OO  L$ respectively. Assume $M$ is a $B\otimes C^\op$-bimodule that is bi-projective (i.e., projective on restricting to the subalgebras $B\otimes \OO$ and $\OO\otimes C$). Then $$M\otimes_{\OO L}-\co C\text{-}\mo \to B\text{-}\mo$$ is an equivalence of categories if and only if $M\otimes K$ induces a bijection of ordinary characters $\Irr (C)\to\Irr(B)$.
\end{lem}

\begin{proof} Let $N:= \Hom_\OO (M,\OO)  $. This is a $C\otimes B^\op$-bimodule, projective on each side. 
	
	We have \begin{equation}\label{NadM}
	 C\text{-}\mo\xrightarrow{M\otimes_C-} B\text{-}\mo \text{ and }   B\text{-}\mo\xrightarrow{N\otimes_B-} C\text{-}\mo  \text{ are left and right adjoint. }
	\end{equation}
	
Indeed the classical (left) adjoint for the tensor product functor $M\otimes_C-$ is $\Hom_B(M,-)$. But the $B$-projectivity of $M$ allows to identify $\Hom_B(M,-)$ with $\Hom_B(M,B) \otimes_B-$. On the other hand, the algebra $B$ is symmetric over $\OO$, namely the restriction to $B$ of the evaluation of the coordinate at 1 in the group algebra yields a linear map $\la\co B\to \OO$ inducing an isomorphism between $B$ and its $\OO$-dual and such that $\la (bb')=\la (b'b)
$ for all $b,b'\in B$ (compare with assumption (1) in Theorem~\ref{Green} above). A basic property is then that \begin{equation}\label{la} \Hom_B(M,B)\cong N\text{ by the map }f\mapsto \la\circ f.
\end{equation}. This and exchanging the roles of $B$ and $C$ gives (\ref{NadM}).

Using the subscript $K$ to denote tensoring by $K$ for $B,C, M,N$, we have the same as (\ref{NadM}) for the semi-simple algebras $B_K$ and $C_K$. The assumption on $M_K$ implies that $M_K\otimes_{C_K}- $ and $N_K\otimes_{B_K}-$ are inverse functors and therefore \begin{equation}\label{MK}
M_K\otimes_{C_K}N_K\cong B_K \text{ and }N_K\otimes_{B_K}M_K\cong C_K \text{ as bi-modules.}
\end{equation}

On the other hand $B_B$ and $(M\otimes_CN)_B$ are projective as right $B$-modules thanks to the bi-projectivity of $M$ and $N$ for the second. But (\ref{MK}) above tells us that they are isomorphic once tensored with $K$ as $B_K$-modules. It is well-known that two projective $\OO H$ modules are isomorphic if and only if they are so when tensored with $K$, see for instance [Du17, \S 4.4]. So we get \begin{equation}\label{MN}
(M\otimes_{C}N)_B\cong B_B,\ {}_B(M\otimes_{C}N)\cong {}_BB,\ (N\otimes_{B}M)_C\cong C_C \text{ and }{}_C(N\otimes_{B}M)\cong {}_CC 
\end{equation} by the symmetry of the situation.

The adjunction between the functors $M\otimes_C-$ and $N\otimes_B-$ mentioned above provide natural transformations of the composites into identity functors. In the case of tensor products functors, this means we have bimodule maps $$\eps\co  C\to N\otimes_BM\ \ \ \text{   and   }\ \ \ \eta\co  M\otimes_CN \to B.$$ Note that they can be made explicit
by following the steps used above, for instance $\eta (m\otimes n)=\la^*(n)(m)$ where $\la^*$ is the inverse of the map (\ref{la}). The basic property of adjunctions (see [McLane, IV.1]) implies that the composite 
\begin{equation}\label{NMN}
N\xrightarrow{\eps\otimes\id_N}N\otimes_{B}M\otimes_{C}N\xrightarrow{\id_N\otimes \eta}N
\end{equation} is the identity. Keeping only the action of $B$ on the right, the three modules are all isomorphic thanks to the first statement in (\ref{MN}) and the maps are inverse isomorphisms. So the maps in (\ref{NMN}) are indeed isomorphisms. But then, tensoring by $M$ on the right gives an isomorphism $$N\otimes_BM\xrightarrow{\eps\otimes\id_{N\otimes M}}N\otimes_BM\otimes_CN\otimes_BM.$$ By the last statement of (\ref{MN}) this means that $\eps$ was an isomorphism in the first place. We also get the same for $\eta$ and this is enough to conclude that our functors $M\otimes_C-$ and $N\otimes_B-$ induce inverse (Morita) equivalences.
\end{proof}

A first application of the lemma is to show that the map of Theorem~\ref{LCGs} is induced by a Morita equivalence in a special case.

\begin{cor} Assume the hypotheses of Theorem~\ref{LCGs} with moreover that $\LL$ is a Levi subgroup of an $F$-stable parabolic subgroup $\PP$. Then the functor $\OO \GF/\RU(\PP)^F\otimes_{\LL^F}
-$ induces a Morita equivalence $B_\ell (\LF ,s)\text{-}\mo \xrightarrow{} B_\ell (\GF ,s)\text{-}\mo$	\end{cor}

The proof simply consists in noting that the functor given induces on characters the Harish-Chandra induction which coincides with $\Lu{\LL}{\GG}$ in our case (see Remark~\ref{4.4}.b), hence a bijection by Theorem~\ref{LCGs}, and on the other hand this bimodule is clearly bi-projective since one may write it $\OO \GF e$ where $e$ is the idempotent $|\RU(\PP)^F|^\mm\sum_{u\in \RU(\PP)^F}u$.



{}
\bigskip

{}

\section{Local methods for blocks of finite quasi-simple groups}

We give more material on general methods for blocks of finite groups. We then illustrate them with the case of symmetric groups. We conclude with a brief discussion of Chuang-Rouquier theorems [ChRo08].

\subsection{Subpairs and local structure of an $\ell$-block}

We go back to $H$ some abstract finite group, and $\ell$ a prime, with $(K,\OO ,k)$ an associated $\ell$-modular system.

An \sing{$\ell$-subpair} in $H$ is any pair $(Q,b_Q)$ where $Q$ is an $\ell$-subgroup of $H$ and $b_Q$ is a primitive idempotent of $\Z(\OO\cent HQ)$. Recall (see Sect. 1.E above) that for $C$ a finite group we have bijections (where prid stands for primitive idempotents)

$$\text{blocks of }\OO C\leftrightarrow \text{prid}(\Z(\OO C))\leftrightarrow \text{prid}(\Z(k C))\leftrightarrow \text{blocks of }k C $$ 

where the middle map is $i\mapsto \ov i$ (reduction mod $J(\OO)$) whose inverse is given by idempotent lifting.

We identify all four kinds of objects above and thus extend the notations $\Irr (B)$, $\CF (H\mid B)$ already seen.

We already introduced the Brauer morphism $\Br_Q\co \Z(kH)\to k\cent HQ$ in Sect. 1.E, but in fact it can be defined on a bigger algebra. Denoting by $(kH)^Q$ the fixed point subalgebra for the conjugacy action of $Q$, one has an algebra morphism

\begin{align*}
	\Br_Q\co(kH)^Q&\to k\cent HQ\\
	\sum_{h\in H}\la_hh&\mapsto \sum_{h\in \cent HQ}\la_hh
\end{align*}

One defines an order relation $\leq$ on $\ell$-subpairs of $H$ by transitive closure of the following

\begin{defn}[Alperin-Brou\'e]\label{AlBr} 
$(Q',b')\lhd (Q,b)$ if and only if \begin{itemize}
	\item $Q$ normalizes $Q'$ and $b'$ (so that $\ov b'\in (kH)^Q$) and
	\item $\Br_Q(\ov b')\ov b=\ov b$.
\end{itemize}
\end{defn}

The $\ell$-blocks of $H$ itself can be seen as $\ell$-subpairs of type $(\{1\},b_1)$. An inclusion $(\{1\},b_1)\lhd (Q,b)$ would exist if and only if $\Br_Q(\ov b_1)\not= 0$ which is the criterion we have seen to define defect groups (see 1.E).

\begin{thm} [Alperin-Brou\'e]\label{subp}
\begin{enumerate}[\rm(i)]
	\item If $(Q,b)$ is an $\ell$-subpair in $H$ and $Q'$ is some subgroup of $Q$, then there is a single subpair with $(Q', b')\leq (Q,b)$.
	\item If $(\{1\}, b_1)$ is an $\ell$-subpair of $H$, the $\leq$-maximal subpairs of $H$ containing it are all $H$-conjugate and are of type $(D,b)$ where $D$ is a defect group of the $\ell$-block $kH\ov b_1$.\index{maximal subpairs}
\end{enumerate}	 
\end{thm}

To an $\ell$-block $B$ of $H$ with defect group $D\leq H$, one can associate a finite category similar to the fusion system of Definition~\ref{FusCat}, see [AschKeOl, IV.2.21].

\begin{defn}[{}]\label{FDbD} Let $(D,b_D)$ be a maximal $\ell$-subpair containing $(\{1\},B)$. Then $\cF_{(D,b_D)}(B)$\index{$\cF_{(D,b_D)}(B)$} is the category whose objects are the subgroups of $D$ and if $D_1,D_2\leq D$, one defines $$\Hom_{\cF_{(D,b_D)}(B)}(D_1,D_2)$$ as the set of maps $D_1\to D_2$ of the form $x\mapsto c_h(x)=hxh^\mm$ where $h\in H$ is such that one has $\ell$-subpair inclusions $$^h(D_1,b_1)\leq (D_2,b_2)\leq (D,b_D)\geq (D_1,b_1).$$
\end{defn}

Like $\cF_Q(H)$ from Definition~\ref{FusCat} on $Q$, the above defines a fusion system in the sense of [AschKeOl] on the $\ell$-group $D$. The ``local structure" of the $\ell$-block $B$ \index{local structure of a block}usually means the knowledge of $\cF_{(D,b_D)}(B)$, which of course does not depend on the choice of the maximal subpair $(D,b_D)$.

\subsection{Brauer's second Main Theorem}

We need first to define the generalized decomposition map $d^x$ ($x$ an $\ell$-element) on central functions. We already had a glimpse of the ordinary decomposition map (when $x=1$) in the form of multiplication by the function denoted by $\dd_\lp$ in the proof of \Th{BM90}.

\begin{defn}
	For $x\in H_\ell$ let $$d^x\co\CF (G)\to \CF(\cent Hx)$$ defined by $d^x(f)(y)=f(xy)$ if $y\in \cent Hx_\lp$, $d^x(f)(y)=0$ otherwise.\index{$d^x$}
\end{defn}

\begin{thm}[Brauer 1959]\label{Br2nd}  Let $x\in H_\ell$.
	Let $(\{1\},b_1)$, $(\spann <x>,b_x)$ be $\ell$-subpairs of $H$. Let $\chi\in\Irr (b_1)$ and assume $d^x(\chi)\in \CF (\cent Hx)$ has non-zero projection on $\CF (\cent Hx\mid b_x)$. Then $$(\{1\},b_1)\leq (\spann <x>,b_x).$$
\end{thm}\index{Brauer's second Main Theorem}

\subsection{Centric or self-centralizing subpairs}

\begin{defn}\label{ctric}
	Let $(Q,b_Q)$ be an $\ell$-subpair of $H$. Then it is called {\bf centric} \index{centric subpair}if and only if $b_Q$ has defect group $\Z(Q)$ in $\cent HQ$. Then there is a single $\zeta\in\Irr(b_Q)$ with $\Z(Q)$ in its kernel, this is called the {\bf canonical character} \index{canonical character} of the centric subpair. 
\end{defn}

It is easy to show the uniqueness of $\zeta$ above, using that $k\cent HQ\ov b_Q$ has a single simple module, hence a single projective indecomposable module. One can recover $b_Q$ from $\zeta$ by the formula \begin{equation}\label{canon}
b_Q={\zeta(1)\over|\cent HQ|}\sum_{h\in\cent HQ_\lp}\zeta (h)h^\mm .
\end{equation}

\begin{thm}[Brauer]\label{inctric} 
	Let $(Q,b)$, $(Q',b')$ some centric $\ell$-subpairs of $H$, with $Q'\lhd Q$. Let $\zeta\in\Irr(b)$,  $\zeta'\in\Irr(b')$ the canonical characters. Then $(Q',b')\leq  (Q,b)$ if and only if $\zeta '$ is $Q$-stable and the multiplicity of $\zeta$ in $\Res_{\cent H{Q}} ^{\cent H{Q'}}\zeta '$ is in $\bbN\setminus \ell\bbN$. 
\end{thm}

In practice, centric subpairs lead easily to maximal subpairs.
\begin{pro}\label{maxic} A subpair inclusion $(Q_1,b_1)\leq (Q_2,b_2)$ with centric $(Q_1,b_1)$ implies that $(Q_2,b_2)$ is also centric and $\zent{Q_2}\leq \zent{Q_1}$. A subpair $(Q,b_Q)$ is maximal if and only if it is centric and $\norm H{Q,b_Q}/Q\cent HQ$ is an $\ell '$-group.
\end{pro}

\subsection{Two main theorems of Brauer and blocks of quasi-simple groups}

Assume we are given a finite group $H$ and a prime $\ell$. We assume we have some information on $\Irr(H)$ and some character values, especially in the form of algorithms reducing to the related questions for smaller groups of the same type. 

Using local methods we want to determine the splitting of $\Irr(H)$ into sets $\Irr(B)$ for $B$ the $\ell$-blocks of $H$, along with the $\ell$-subpairs of $H$ (which includes determining the defect groups of $\ell$-blocks).

For $\chi\in\Irr(H)$, let us denote by $b_H(\chi)$ the $\ell$-block such that $\chi\in\Irr(b_H(\chi))$.

Starting with $  \chi\in \Irr (H)$, two possiblities occur. Either there is some $1\not= x\in H_\ell $ such that $d^x\chi \not= 0$ or $\chi (H\setminus H_\lp)=0$. 

In the second case it is classical that $b_H(\chi)$ has defect $\{1\}$ and is the only character of that block. In the first case it's often the case that such an $x$ can be found non-central. Then Brauer's second Main Theorem allows to get an inclusion $$(\{1\},b_H(\chi))\leq (\spann <x>, b').$$ If now $H'\deq\cent Hx$ has a similar structure as $H$, or say we know $\Irr(H')$ just as well, and if the maps $d^{x'}\co \Irr(H')\to \CF(\cent{H'}{x'})$ are also not too difficult to compute, we can do for $H'$ the same as above. 

This subpair enlargement process will provide us with an inclusion $$(\{1\},b_H(\chi))\leq (A,b_A)$$ where $A$ is an abelian $\ell$-subgroup and $b_A$ has central defect group in $\cent HA$. So we can assume that $(A,b_A)$ is centric.

Using now Theorem~\ref{inctric}, including $(A,b_A)$ into other centric subpairs is a relatively classical problem of character restrictions. One then gets to a maximal subpair $(D,b_D)\geq (\{1\},b_H(\chi))$. By conjugacy of maximal subpairs, this solves the problem of saying when two characters $\chi$, $\chi'$ of $H$ belong to the same block. One has $b_H(\chi)=b_H(\chi ')$ if and only if the corresponding pairs $(D,b_D)$ and $(D',b_{D'})$ are conjugate.

This is not precisely the pattern followed by Brauer-Robinson to determine the blocks of symmetric groups first conjectured by Nakayama (see [Naka41b],[Br47]) but it was used by others (see [MeTa76] and Sect. 5.F below) and by Fong-Srinivasan for the blocks of finite classical groups ([FoSr82] and [FoSr89]).

\begin{rem}
Note that we have avoided the question of characters that would vanish on $H\setminus \Z(H)H_\lp$ but are not in an $\ell$-block of central defect. This can happen only if $\ell\mid |\Z(H)|$. If we have started with a quasi-simple group $H$, this means that $\ell$ divides the order of the Schur multiplier of a simple group. Indeed for $H$ the double cover of alternating or symmetric groups, the 2-blocks of faithful characters had to be determined by other methods (Bessenrodt-Olsson [BeOl97]). But such a phenomenon seems a bit isolated and not present in finite groups of Lie type.
\end{rem}

\subsection{The symmetric group: characters}

Let us recall the parametrization of $\Irr(\Sym{n})$ and the formula of Murnaghan-Nakayama giving the character values. We refer for instance to [JamesKer] for the classical theory while [Klesh] gives a very direct approach to a more general setting.

For $n\geq 0$, one defines $\cP(n)=\{\la\mid\la\vdash n \}$ the set of integer partitions of $n$, $\la =(\la_1\geq \la_2\geq \dots \geq\la_k)$ with $\la_k > 0$ and $\la_1+ \la_2+ \dots +\la_k=n$. One also denotes $|\la|=n$. This includes $\emptyset \vdash 0$.

One has a bijection \begin{align*}
	\cP(n)&\to\Irr(\Sym n),\\
	\la&\mapsto \zeta_\la\ .
\end{align*}

The trivial character corresponds to the partition $(n)$ through this bijection.
We don't give its definition but go to the Murnaghan-Nakayama rule that allows to compute inductively the character values.

Let $d\geq 1$. To each $\la\vdash n$ is associated the set $\hoo_d(\la)$ of its \sing{hooks} of length $d$ and for each $\tau\in\hoo_d(\la)$ there is a removal operation producing $\la -\tau\vdash n-d$. Each hook has a {\bf height} $h(\tau)\in\bbN$. The \sing{Murnaghan-Nakayama rule} is as follows. 

Assume $1\leq d\leq n$ and $x\in \Sym{n}$ writes as $x=x'c$ where $x'\in\Sym{n-d}$ and $c$ is a cycle of order $d$ on $\{n-d+1,\dots,n\}$. Let $\la\vdash n$. Then [JamesKer, 2.4.7] \begin{equation}\label{MNrule}
\zeta_\la (x)=\sum_{\tau\in\hoo_d(\la)}^{}(-1)^{h(\tau)}\zeta_{\la-\tau}(x').\end{equation}
  
Let's be more explicit on hooks and the removal process. Partitions are often represented by \sing{Young diagrams}, where $\la=(\la_1\geq \la_2\geq\cdots)$ is represented by rows of boxes of sizes $\la_1$, $\la_2$, etc.. The rows are aligned on the left and all boxes are identical. Below is the diagram for the partition $(4,3,1,1)\vdash 9$. The \sing{rim} of the diagram consists of the boxes such that no box is at the same time under and on the right of them. On the first diagram below the rim of 7 boxes is dotted. A hook is an interval in this rim starting and finishing at some box with no box under or right of it. Its length is the number of boxes it comprises. Its {\bf height} is the number of rows affected minus 1. Below are six hooks with length $d$ and height $h$ indicated. (Exercise: find the four hooks missing.)

$$\begin{array}{c}
{\begin{array}{ccccc}\cline{1-4}
	\multicolumn{1}{|c}{\phantom{\bf .}}	&\multicolumn{1}{|c}{\phantom{\bf .}}&\multicolumn{1}{|c}{ {\bf .}}&\multicolumn{1}{|c}{ {\bf .}} &\multicolumn{1}{|c}{\phantom{\bf .}} \\ \cline{1-4}
\multicolumn{1}{|c}{ {\bf .}}	&\multicolumn{1}{|c}{ {\bf .}}	&\multicolumn{1}{|c}{ {\bf .}}&\multicolumn{1}{|c}{\phantom{\bf .}}&  \\ \cline{1-3}
 \multicolumn{1}{|c}{ {\bf .}} &\multicolumn{1}{|c}{\phantom{\bf .}}	& &  & \\ \cline{1-1}
 \multicolumn{1}{|c}{ {\bf .}} &\multicolumn{1}{|c}{\phantom{\bf .}}	& &  & \\ \cline{1-1}	\end{array}}
\\
\\
(d,h)=(7,3)\ \  \end{array} \ \ 
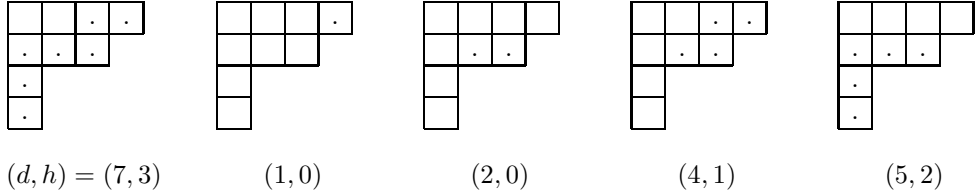
\begin{array}{c}
{\begin{array}{ccccc}\cline{1-4}
	\multicolumn{1}{|c}{\phantom{\bf .}}	&\multicolumn{1}{|c}{\phantom{\bf .}}&\multicolumn{1}{|c}{\phantom{\bf .}}&\multicolumn{1}{|c}{{\bf .}} &\multicolumn{1}{|c}{\phantom{\bf .}} \\ \cline{1-4}
	\multicolumn{1}{|c}{\phantom{\bf .}}	&\multicolumn{1}{|c}{\phantom{\bf .}}	&\multicolumn{1}{|c}{\phantom{\bf .}}&\multicolumn{1}{|c}{\phantom{\bf .}}&  \\ \cline{1-3}
	\multicolumn{1}{|c}{\phantom{\bf .}} &\multicolumn{1}{|c}{\phantom{\bf .}}	& &  & \\ \cline{1-1}
	\multicolumn{1}{|c}{\phantom{\bf .}} &\multicolumn{1}{|c}{\phantom{\bf .}}	& &  & \\ \cline{1-1}	\end{array}}
\\
\\
 (1,0)\ \  \end{array} \ \ \begin{array}{c}
{\begin{array}{ccccc}\cline{1-4}
	\multicolumn{1}{|c}{\phantom{\bf .}}	&\multicolumn{1}{|c}{\phantom{\bf .}}&\multicolumn{1}{|c}{\phantom{\bf .}}&\multicolumn{1}{|c}{\phantom{\bf .}} &\multicolumn{1}{|c}{\phantom{\bf .}} \\ \cline{1-4}
	\multicolumn{1}{|c}{\phantom{\bf .}}	&\multicolumn{1}{|c}{{\bf .}}	&\multicolumn{1}{|c}{{\bf .}}&\multicolumn{1}{|c}{\phantom{\bf .}}&  \\ \cline{1-3}
	\multicolumn{1}{|c}{\phantom{\bf .}} &\multicolumn{1}{|c}{\phantom{\bf .}}	& &  & \\ \cline{1-1}
	\multicolumn{1}{|c}{\phantom{\bf .}} &\multicolumn{1}{|c}{\phantom{\bf .}}	& &  & \\ \cline{1-1}	\end{array}}
\\
\\
 (2,0)\ \  \end{array} \ \ 
\begin{array}{c}
{\begin{array}{ccccc}\cline{1-4}
	\multicolumn{1}{|c}{\phantom{.}}	&\multicolumn{1}{|c}{\phantom{.}}&\multicolumn{1}{|c}{{.}}&\multicolumn{1}{|c}{{.}} &\multicolumn{1}{|c}{\phantom{.}} \\ \cline{1-4}
	\multicolumn{1}{|c}{\phantom{.}}	&\multicolumn{1}{|c}{{.}}	&\multicolumn{1}{|c}{{.}}&\multicolumn{1}{|c}{\phantom{.}}&  \\ \cline{1-3}
	\multicolumn{1}{|c}{\phantom{.}} &\multicolumn{1}{|c}{\phantom{.}}	& &  & \\ \cline{1-1}
	\multicolumn{1}{|c}{\phantom{.}} &\multicolumn{1}{|c}{\phantom{.}}	& &  & \\ \cline{1-1}	\end{array}}
\\
\\
 (4,1)\ \  \end{array} \ \ \begin{array}{c}
{\begin{array}{ccccc}\cline{1-4}
	\multicolumn{1}{|c}{\phantom{.}}	&\multicolumn{1}{|c}{\phantom{.}}&\multicolumn{1}{|c}{\phantom{.}}&\multicolumn{1}{|c}{\phantom{.}} &\multicolumn{1}{|c}{\phantom{.}} \\ \cline{1-4}
	\multicolumn{1}{|c}{{.}}	&\multicolumn{1}{|c}{{.}}	&\multicolumn{1}{|c}{{.}}&\multicolumn{1}{|c}{\phantom{.}}&  \\ \cline{1-3}
	\multicolumn{1}{|c}{{.}} &\multicolumn{1}{|c}{\phantom{.}}	& &  & \\ \cline{1-1}
	\multicolumn{1}{|c}{{.}} &\multicolumn{1}{|c}{\phantom{.}}	& &  & \\ \cline{1-1}	\end{array}}
\\
\\
(5,2)\ \  \end{array} \ \ $$

It is clear that removing a hook $\tau$ of length $d$ gives a Young diagram with $n-d$ boxes, hence the meaning of $\la -\tau\vdash n-d$ above. Note that in (\ref{MNrule}) above $d$ can be equal to 1. This case of the Murnagan-Nakayama rule gives the restriction of $\chi\in\Irr(\Sym n)$ to $\Sym{n-1}$ and is called the \sing{branching rule}.

 Note that when $\la$ has no $d$-hook, then (\ref{MNrule}) gives $\zeta_\la(x'c)=0$. A partition $\la$ is said to be a $d${\bf -core}\index{$d$-core} if and only if $\hoo_d(\la)=\emptyset$. For instance the partition above is a 6-core. 

For a given $d$, starting with some partition, the hook removal can be iterated $\la\mapsto \la-\tau_1\mapsto (\la-\tau_1)-\tau_2\mapsto \cdots$ where $\tau_1\in\hoo_d(\la)$, $\tau_2\in\hoo_d(\la-\tau_1)$, etc.. until we get a $d$-core. This is done below with $d=2$, the hook removed next being dotted. 

$${\begin{array}{ccccc}\cline{1-4}
	\multicolumn{1}{|c}{\phantom{.}}	&\multicolumn{1}{|c}{\phantom{.}}&\multicolumn{1}{|c}{\phantom{.}}&\multicolumn{1}{|c}{\phantom{.}} &\multicolumn{1}{|c}{\phantom{.}} \\ \cline{1-4}
	\multicolumn{1}{|c}{\phantom{.}}	&\multicolumn{1}{|c}{\phantom{.}}	&\multicolumn{1}{|c}{\phantom{.}}&\multicolumn{1}{|c}{\phantom{.}}&  \\ \cline{1-3}
	\multicolumn{1}{|c}{{.}} &\multicolumn{1}{|c}{\phantom{.}}	& &  & \\ \cline{1-1}
	\multicolumn{1}{|c}{{.}} &\multicolumn{1}{|c}{\phantom{.}}	& &  & \\ \cline{1-1}	\end{array}} 
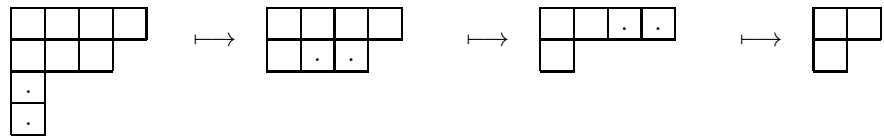
\begin{array}{c}  \longmapsto \\ \\ \\
\end{array}
 \ \ 
{\begin{array}{ccccc}\cline{1-4}
	\multicolumn{1}{|c}{\phantom{.}}	&\multicolumn{1}{|c}{\phantom{.}}&\multicolumn{1}{|c}{\phantom{.}}&\multicolumn{1}{|c}{\phantom{.}} &\multicolumn{1}{|c}{\phantom{.}} \\ \cline{1-4}
	\multicolumn{1}{|c}{\phantom{.}}	&\multicolumn{1}{|c}{{.}}	&\multicolumn{1}{|c}{{.}}&\multicolumn{1}{|c}{\phantom{.}}&  \\ \cline{1-3} \\
	&&&& 
		\end{array}} \ \ \begin{array}{c}  \longmapsto \\ \\ \\
	\end{array} \ \ 
	{\begin{array}{ccccc}\cline{1-4}
	\multicolumn{1}{|c}{\phantom{.}}	&\multicolumn{1}{|c}{\phantom{.}}&\multicolumn{1}{|c}{{.}}&\multicolumn{1}{|c}{{.}} &\multicolumn{1}{|c}{\phantom{.}} \\ \cline{1-4}
	\multicolumn{1}{|c}{\phantom{.}}	&\multicolumn{1}{|c}{\phantom{.}}	& & &  \\ \cline{1-1} \\
	&&&& 
	\end{array}}  \ \ \begin{array}{c}  \longmapsto \\ \\ \\
\end{array}\ \ 
{\begin{array}{ccc}\cline{1-2}
\multicolumn{1}{|c}{\phantom{.}}	&\multicolumn{1}{|c}{\phantom{.}}&\multicolumn{1}{|c}{\phantom{.}} \\ \cline{1-2}
\multicolumn{1}{|c}{\phantom{.}}	&\multicolumn{1}{|c}{\phantom{.}}	&   \\ \cline{1-1} \\
&& 
\end{array}}  $$

It can be proved that given $d$ and $\la$, this process of hook removal always ends in the same $d$-core $\la_{(d)}$ and that the sign $\eps_{\la ,d}=(-1)^{h(\tau_1)+h(\tau_2)+\cdots}$ also does not depend on the path followed. One then gets the following {\bf iterated Murnaghan-Nakayama rule} [JamesKer, 2.7.27]
\begin{equation}\label{iMN}
\zeta_\la (x'c_1c_2\dots c_w)=\eps_{\la ,d}N_{\la ,d}\zeta_{\la_{(d)}}(x')
\end{equation}
where $x'\in\Sym{n-wd}$, $c_1, c_2,\dots ,c_w$ are disjoint cycles of order $d$ on $\{n-wd+1,\dots ,n\}$, and $N_{\la ,d}$ is the number of ways to go from $\la$ to $\la_{(d)}$ by successive $d$-hook removals.

\begin{rem}
Young diagrams were essentially created to fill the boxes with additional information. Working with hooks and cores is made easier by using $\beta${\bf -numbers}\index{$\beta$-numbers} instead of partitions. One replaces the partition $\la =(\la_1\geq \la_2\geq \dots \geq\la_k)$ by the set $\beta:=\{ \la_1+k-1, \la_2+k-2,\dots ,\la_k\}$. A hook of length $d$ is then replaced by a pair $\{a,a-d\}$ such that $a\in\beta$ and $0\leq a-d\not\in\beta$. The removal $\la\mapsto \la-\tau$ becomes $\beta\mapsto \beta\setminus\{a\}\cup\{a-d\}$. The iteration with a fixed $d$ is then easy to control and uniqueness of the outcome is quite clear. The height of hook is the number of elements of $\beta$ between $a$ and $a-d$, so that the sign $(-1)^{h(\tau)}$ can be interpreted as the signature of a cycle and the product of signs at the end of the process is clearly independent of the path followed. This makes clear how to get (\ref{iMN}) from (\ref{MNrule}).

The following fact is also made trivial by working with $\beta$-numbers. \begin{equation}\label{d-dd'}
\text{If  } \la\vdash n \text{  and  } \hoo_d(\la)=\emptyset \text{  then  } \hoo_{dd'}(\la)=\emptyset \text{  for any  } d'\geq 1.
\end{equation}
\end{rem}
\subsection{The symmetric group: blocks} We give here the classification of blocks of symmetric groups by use of local methods. The approach to $\Irr(\Sym{n})$ described in [Klesh] gives more generally the blocks of all Iwahori-Hecke algebras of type $\tA$ (see [Klesh, 9.6.2]).

Let $n\geq 1$ and $\ell$ be a prime.

\begin{thm}[Brauer-Robinson]\label{BrRo} 
The $\ell$-blocks of $\Sym n$ are parametrized  $$\kappa\mapsto B_\kappa$$ by the $\ell$-cores $\kappa$ such that $\kappa\vdash n-w\ell$ for some $w\geq 0$.
One has $\zeta_\la\in\Irr(B_\kappa)$ if and only if $\la_{(\ell)}=\kappa$.

The Sylow $\ell$-subgroups of $\Sym{n-|\kappa|}$ are defect groups of $B_\kappa$.
\end{thm}

\begin{lem}\label{zetaka} If $\la\vdash n$ is an $\ell$-core, then $\zeta_\la$ vanishes outside $\lp$-elements of $\Sym{n}$.
\end{lem}

To prove the above, let $x\in\Sym n$ with $x_\ell\not= 1$. Then it has in its cycle decomposition a cycle of order $t$ a multiple of $\ell$, $x=x'c$ with $c$ a cycle of order $t$ and $x'$ fixing any element in the support of $c$. Then $\hoo_t(\la)=\emptyset$ by (\ref{d-dd'}) above and the Murnaghan-Nakayama rule (\ref{MNrule}) gives $\zeta_\la(x'c)=0$ as claimed.

We now prove Theorem~\ref{BrRo}. Let $\la\vdash n$ with $\la_{(\ell)}\vdash n-w\ell$, $w\geq 0$. Let $c$ a product of $w$ disjoint cycles of order $\ell$ on $\{n-w\ell +1,\dots , n\}$. Then $\cent{\Sym n}{c}=\Sym {n-w\ell}\times W$ where $W$ is isomorphic to the centralizer of a product of $w$ disjoint cycles of order $\ell$ in $\Sym{w\ell}$.  We have
\begin{equation}\label{CW} \cent{W}{\OP{W}}\leq {\OP{W}}.
\end{equation}

If $c=c_1\dots c_w$ is the product of our disjoint cycles of order $\ell$, it is clear that the $c_i$'s are permuted by any element of $W$, so $C\deq\spann <c_1,\dots ,c_w>\cong (\bbZ/\ell\bbZ)^w$ is normal in $W$. On the other hand $\cent WC=C$ since a permutation centralizing $C$ has to stabilize the support of each $c_i$ and the centralizer of a cycle of order $\ell$ in $\Sym\ell$ is clearly the cyclic subgroup this cycle generates. Thus (\ref{CW}).

Note that (\ref{CW}) implies that \begin{equation}\label{BlCW}
W \text{ has a single }\ell\text{-block  } B_0(W)
\end{equation}
Indeed the defect group of an $\ell$-block of $W$ must contain $\OP W$ (see for instance [NagaoTsu, 5.2.8.(i)]) and then we can argue as in the proof of Proposition~\ref{blochp}.

Let $B\deq b_{\Sym {n-w\ell}}(\zeta_{\la_{(\ell)}}).B_0(W)\in\Bl(\Sym {n-w\ell}\times W)$. We prove
\begin{equation}\label{Blincl}
(\{1\},b_{\Sym {n}}(\zeta_\la))\lhd (\spann <c>,B).
\end{equation}

By Brauer's second Main Theorem (Theorem~\ref{Br2nd}), it suffices to prove that $d^c\zeta_\la$ has non zero projection on $\CF(\Sym {n-w\ell}\times W\mid B)$. Since $W$ has only one block $\CF(\Sym {n-w\ell}\times W\mid B)=\CF(\Sym {n-w\ell}\mid b_{\Sym {n-w\ell}}(\zeta_{\la_{(\ell)}}))\times \CF(W)$. On the other hand, Lemma~\ref{zetaka} implies that $\zeta_{\la_{(\ell)}}$ is the only irreducible character in its $\ell$-block (see for instance [NagaoTsu, 3.6.29]) so $\CF(\Sym {n-w\ell}\times W\mid B)=\bbC\zeta_{\la_{(\ell)}} \times \CF(W)$. If the projection of $d^c\zeta_\la$ were 0 on it, we would have $\Res^{\Sym {n-w\ell}\times W}_{\Sym {n-w\ell}}d^c\zeta_\la\in \bbC(\Irr(\Sym {n-w\ell})\setminus\{\zeta_{\la_{(\ell)}}\})$.
Using the usual inner product on central functions, this would give \begin{equation}\label{sca}
\sum_{x\in(\Sym{n-w\ell})_\lp}^{}\zeta_\la(xc)\zeta_{\la_{(\ell)}}(x^\mm)=0.
\end{equation}

But we have \begin{eqnarray*}
\sum_{x\in(\Sym{n-w\ell})_\lp}^{}\zeta_\la(xc)\zeta_{\la_{(\ell)}}(x^\mm)&=&\sum_{x\in\Sym{n-w\ell}}^{}\zeta_\la(xc)\zeta_{\la_{(\ell)}}(x^\mm) \text{ by Lemma~\ref{zetaka}} \\
&=&\eps_{\la ,\ell}N_{\la,\ell}\sum_{x\in\Sym{n-w\ell}}^{}\zeta_{\la_{(\ell)}}(x)\zeta_{\la_{(\ell)}}(x^\mm) \text{ by (\ref{iMN})} \\
&=&\eps_{\la ,\ell}N_{\la,\ell}|\Sym{n-w\ell}|\not=0, \text{ a contradiction.} \end{eqnarray*}{}

Note that having (\ref{Blincl}) proves at once that $b_{\Sym {n}}(\zeta_\la)=b_{\Sym {n}}(\zeta_\mu)$ as soon as $\la,\mu\vdash n$ have same $\ell$-core $\kappa\vdash n-w\ell$. This gives the map $\kappa\mapsto B_\kappa$ announced. It is also easy to include 
the second subpair of (\ref{Blincl}) into a maximal one. Let $D$ be a Sylow $\ell$-subgroup of $S'$, the symmetric group on $\{n-w\ell+1,\dots ,n\}$. Assume $D$  contains the cycles of which $c$ is a product. Then the centralizer of $D$ in $\Sym n$ is $\Sym{n-w\ell}\times \cent{S'}D$ and we define $B_D=b_{\Sym {n-w\ell}}(\zeta_\kappa).B_0(\cent{S'}D)$ where $B_0(\cent{S'}D)$ is the principal block of $\cent{S'}D$. Using Theorem~\ref{inctric} and Proposition~\ref{maxic}, one gets 
inclusions
\begin{equation}\label{Maxinc}
(\{1\},B_\kappa)\leq (\spann <c>,B)  \text{  and  }  (\spann <c>,B)\leq (D,B_D)
\end{equation}
the latter being maximal.
\qed

\begin{rem}\label{FusSn} 
	It is easy to check that (\ref{CW}) above is true for any centralizer of an $\ell$-subgroup of $\Sym {m}$ having no fixed point on $\{1,\dots , m \}$.
	Let $P$ be an $\ell$-subgroup of $\Sym{n}$. We assume that its fixed points in $\{1,\dots ,n \}$ are $\{1,\dots ,n_P \}$, so that $\cent{\Sym{n}}{P}=\Sym{n_P}\times W_P$ where $W_P$ has only one $\ell$-block by (\ref{BlCW}). Let $B_\kappa$ an $\ell$-block of $\Sym{n}$ as in \Th{BrRo}, let $b_\kappa^{(n)}\in \zent{k\Sym{n}} $ the corresponding central idempotent in the group algebra of characteristic $\ell$. From the above, one computes easily the Brauer morphism 
	 \begin{equation}
	\Br_P(b_\kappa^{(n)})=
	\begin{cases}
	b_\kappa^{(n_P)}\otimes 1_{kW_P}, & \text{if}\ n_P\geq |\kappa| , \\
	0, & \text{otherwise.}
	\end{cases}
	\end{equation}

This shows that the fusion system of $\ell$-subpairs of $B_\kappa$ (see Definition~\ref{FDbD}) is isomorphic with the fusion system of $\ell$-subgroups of  $\Sym{n-|\kappa|}$ (Brou\'e-Puig, see [Bro86, 2.B.4]).

\end{rem}

\subsection{The symmetric group:  Chuang-Rouquier's theorems}

We keep $\ell$ a prime number. 

\begin{thm}[Chuang-Rouquier, 2008]\label{ChRo} If two $\ell$-blocks $A\inn\OO\Sym{n}$ and $A'\inn\OO\Sym{n'}$ have isomorphic defect groups then $$D^b(A\text{-}\mo)\cong D^b(A'\text{-}\mo).$$
\end{thm}

The proof introducing to representations of finite groups the notion of \sing{categorifications} certainly opens a new chapter of representation theory. See [Mazor] for a very complete introduction. [ChRo08] appeared on arXiv in 2004 and categorification was then a very recent notion. At that time MathSciNet had only 10 papers with the word in their title, the first in 1998; there are now on average 20 per year. 

The common feature of the various instances of categorification is for $A$ a ring the existence of exact endofunctors of an abelian category $\cC$ such that their images in the endomorphism ring of the Grothendieck group give a representation $$A\to \End(K_0(\cC)).$$ This is generally verified by using a presentation of $A$ and checking that our endofunctors induce endomorphisms of $K_0(\cC)$ satisfying the relations presenting $A$. The phrase ``categorical action" of $A$ is also used (see [DuVV15]).

Concerning blocks of symmetric groups, it was known for some time that the size of the defect group of a block of $\Sym{n}$ determined all its numerical invariants (see [En90]). The sum $$\cG\deq\oplus_{n\geq 0}K_0(k\Sym{n})$$ of the Grothendieck groups of all symmetric groups had been considered by various authors (see [Ze81]) in connection with the branching rule. More recently, Lascoux-Leclerc-Thibon had shown an action of affine Lie algebras on it related with \sing{Jucys-Murphy elements} $ L_i=(1,i)+(2,i)+\cdots +(i-1,i) $. See also [Klesh, \S 9]. The $L_i$'s pairwise commute and the algebra they generate compares with the Cartan subalgebra of a Lie algebra, thus bringing to symmetric groups a key feature of Lie theory. 

Recall the Lie algebra $\fsl_n$\index{$\fsl_n$} over $\bbZ$ of $n\times n$ matrices with trace 0. It can be presented by generators and relations satisfying the Chevalley-Serre relations. In the case of $n=2$, we get a Lie algebra $\fsl_2=\bbZ E\oplus \bbZ F\oplus \bbZ H$ defined by 
the relations $[E,F]=H$, $[H,E]=2E$, $[H,F]=-2F$.

The affine Lie algebra $\asl_n$ is generated by the elements $E_0,\dots ,E_{n-1}$, $F_0,\dots ,F_{n-1}$, $H_0,\dots ,H_{n-1}$ subject to the relations 
\begin{equation}\label{ChS1}
\text{  $ [E_i,F_j]=\delta_{i,j}H_i   $, $[H_i,E_j]=C_{i,j}E_j $ , $[H_i,F_j]=-C_{i,j}F_j $ }
\end{equation}
\begin{equation}\label{ChS2}
\text{  $ (\ad_{E_i})^{1-C_{i,j}}(E_j)= (\ad_{F_i})^{1-C_{i,j}}(F_j)=0 $ for $i\ne j $. }
\end{equation} where $C_{i,j}$ is the Cartan matrix of the affine root system of type $\wh\tA_{n-1}$.

Let $a\in \bbF_\ell$. For $M$ a $k\Sym{n}$-module one denotes $$F_{a,n}(M)=\{v\in M\mid av=\big((1,n)+(2,n)+\cdots+(n-1,n)\big).v\}$$ the eigenspace of the $n$-th Jucys-Murphy element. This is $\Sym{n-1}$-stable. This gives a decomposition of the additive restriction functor $$\Res_{\Sym{n-1}}^{\Sym{n}}=\oplus_{a\in\bbF_\ell}F_{a,n}\colon k\Sym{n}\text{-}\mo\to k\Sym{n-1}\text{-}\mo .$$ Analogously one gets a decomposition $ \Ind_{\Sym{n-1}}^{\Sym{n}}=\oplus_{a\in\bbF_\ell}E_{a,n}$ with corresponding adjunctions. One defines \begin{equation}\label{EaFa}
\text{  $   E_a=\oplus _{n\geq 1} E_{a,n},\   F_a=\oplus _{n\geq 1} F_{a,n}\ \colon \oplus_{n\geq 1}k\Sym{n}\text{-}\mo \to\oplus_{n\geq 1}k\Sym{n}\text{-}\mo $. }
\end{equation}

\begin{thm}[Lascoux-Leclerc-Thibon]\label{LLT}\begin{enumerate}[\rm(i)]
		\item The action of the above $E_0, \dots ,$ $E_{\ell-1}, F_0, \dots , F_{\ell-1}$ induces an action of $\asl_\ell$ on the Grothendieck group $\cG $.
		\item The decomposition of $\cG$ induced by $\ell$-blocks corresponds to a decomposition into weight spaces (for the subalgebra generated by the $H_a=[E_a,F_a]$'s).
		\item Two $\ell$-blocks have same defect group if and only if they are in the same orbit under the action of the Weyl group of $\asl_\ell$.
	\end{enumerate}	
\end{thm}

For each pair $a\in \bbF_\ell$, the above situation restricts to actions of $\fsl_2$. This is called more generally by Chuang-Rouquier a weak $\fsl_2$-categorification. A \sing{strong $\fsl_2$-categorification } is defined as follows\index{categorification}. We give the version actually used for blocks of symmetric groups, the one in [ChRo08] uses a parameter $q$ which is 1 here.

\begin{defn}[Chuang-Rouquier]\label{StrCat} Let $\cA$ be a $k$-linear abelian category with finiteness properties (satisfied in the application given). A strong $\fsl_2$-categorification is the data of $a\in k$, exact functors $$E,F\colon \cA\to \cA$$ and natural transformations $$X\colon E\to E\ ,\  T\colon E^2\to E^2$$ such that  
	\begin{enumerate}[]
		\item $(E,F)$ is an adjoint pair and $F$ is isomorphic to a left adjoint of $E$,
		\item $E$ and $F$ induce on the Grothendieck group $K_0(\cA)$ a locally finite representation of $\fsl_2$,
		\item the simple objects of $\cA$ are weight vectors for the above in $K_0(\cA)$,
		\item $(\id_ET)\circ (T\id_E)\circ (\id_ET)= (T\id_E)\circ (\id_ET)\circ (T\id_E)$ as natural transformations  $E^3\to E^3$,
		\item $T^2=\id_{E^2}$ and $T\circ (\id_EX)\circ T=X\id_E-T$ as natural transformations  $E^2\to E^2$, 
		\item $X-a\id_E$ is locally nilpotent.
	\end{enumerate}
\end{defn}

A very important feature is of course the role of the endomorphisms of functors $X$ and $T$ and the relations they satisfy. Categorification techniques lead to consider functors as objects and natural transformations as morphisms. Note that in equations like (4) and (5) above, $\circ$ denotes the classical composition of natural transformations of functors. Meanwhile, an expression like $\id_ET$ means the endomorphism of $EE^2$ obtained functorially from endomorphisms of $E$ and $E^2$. Note that in the case of module categories (or categories closely related with due adaptations), functors are mostly tensor products by bi-modules which in turn are easy to consider as objects of an abelian category as in the proof of Lemma~\ref{K2O}.

In practice, one will define several $\fsl_2$-categorifications that come from a structure involving a whole $\asl_\ell$ where $\ell$ is the characteristic of $k$. On top of the ``representations" of $\fsl_2$ one gets, the various $X$'s and $T$'s will contribute to controlling the modules produced through the action of (affine) Hecke algebras. 

In the setting of Definition~\ref{StrCat}, Chuang-Rouquier prove the following fundamental theorem (see [Du17, Sect. 1.2-3] for the categories $\HO$ and $\tD^b$).

\begin{thm}[{[ChRo08, 6.4]}]\label{Theta} There is an equivalence of categories $$\Theta\colon\HO(\cA)\to \HO(\cA)$$ inducing the action of the reflection of the Weyl group of $\fsl_2$ on the Grothendieck group $K_0(\cA)$.
\end{thm}

We have skipped the (difficult) definition of $\Theta$ (originally due to Rickard for blocks of symmetric groups). A key point is to find an equivalent of the divided powers $e^m/m!$, $f^m/m!$ ($e$ and $f$ being the images of $E,F\in \fsl_2$ in a representation over $\bbQ$) that are necessary to define the action $\exp(-f)\exp(e)\exp(-f)$ of the Weyl group on a representation. See [ChRo, 5.13, 6.1] for the model proposed. Then the proof of invertibility of $\Theta$ is another challenge, where a key step is to show invertibility of the induced functor $\tD^b(\cA)\to \tD^b(\cA)$, see [ChRo, 6.4, 6.6] and proofs.

This theorem, with an additional parameter $q\in k^\times$, has several applications in [ChRo08] beyond symmetric groups, namely blocks of Ariki-Koike algebras, of finite general linear groups, or the so-called category $\OO$ (see [ChRo08, \S 7]).

With \Th{Theta} in hand, the proof of \Th{ChRo} consists then in constructing a strong $\fsl_2$-categorification for each $a\in \bbF_\ell$ with $\cA=\oplus_{n\geq 1}k\Sym{n}\text{-}\mo$. The functors $E_a$, $F_a$ have been seen above. The natural transformations $$X_a\colon E_a\to E_a, \ T_a\colon E_a^2\to E_a^2$$ are defined as follows. The functor $E_a=\oplus_{n\geq 1}E_{a,n}$ is such that $$E_{a,n}\colon k\Sym{n-1}\text{-}\mo\to  k\Sym{n}\text{-}\mo$$ is a direct summand of induction, hence induced by a direct summand of the $k\Sym{n-1}\otimes  k\Sym{n}^\op$-bimodule $  k\Sym{n}$. One defines $X_a$ there as the right multiplication by $(1,n)+(2,n)+\cdots+(n-1,n)$ on the bimodule. On the other hand $E_a^2=\oplus_{n\geq 2} E_{a,n}E_{a,n-1}$ where the $n$-th term is a direct summand of the functor $k\Sym{n-2}\text{-}\mo\to  k\Sym{n}\text{-}\mo$ induced by the $k\Sym{n-2}\otimes  k\Sym{n}^\op$-bimodule $  k\Sym{n}$. The natural transformation $T_a$ is then defined by right multiplication by $(n-1,n)$.
One proves that this provides a strong $\fsl_2$-categorification. Working with $\cA=\oplus_{n\geq 1}k\Sym{n}\text{-}\mo$, \Th{Theta} then allows to deduce an equivalence 
\begin{equation}\label{A2A'}
\text{  $\HO(\cA)\to \HO(\cA)$ restricting to $\HO(A\text{-}\mo)\to \HO(A'\text{-}\mo)$}
\end{equation} for each pair $(A,A')$ of $\ell$-blocks of symmetric groups such that $A'$ is the image of $A$ by a fundamental reflection in \Th{LLT}.(iii). Using the integral nature of all functors involved one can lift that to algebras over $\OO$ or even $\bbZ_\ell$. Using \Th{LLT}.(iii) one can iterate this strengthened version of (\ref{A2A'}) to get equivalences $\HO(A\text{-}\mo)\to \HO(A'\text{-}\mo)$ for each pair of $\ell$-blocks $A\inn\OO\Sym{n}$ and $A'\inn\OO\Sym{n'}$ with isomorphic defect groups.

Using the knowledge of the Brauer morphism (see Remark~\ref{FusSn}), Chuang-Rouquier prove an even stronger equivalence of the blocks $A$, $A'$ concerned, namely a Rickard equivalence {(see Definition~\ref{RickEqu} below)} that basically preserves the local structures of the blocks and whose image by the Brauer morphism induces similar equivalences at the local level [ChRo08, 7.2].




{}
\bigskip

{}

\section{Local methods for unipotent blocks: the strategy}

In view of a possible Jordan decomposition of characters inducing a strong equivalence of $\ell$-blocks (see Sect. 9 below) it may make sense to give details about local methods only for unipotent blocks (see Definition~\ref{ublo}). This is what we do in Sections 6 to 8.

\subsection{Generalized $d$-Harish-Chandra theory}

We keep $\GG$ and $F\co\GG\to \GG$ as before (Ch. 4). 

We have until now defined Levi subgroups as complements in a decomposition of a parabolic subgroup $\PP =\RU(\PP)\rtimes\LL$. A more intrinsic definition is by saying that they are centralizers of tori (not necessarily maximal), see for instance [DigneMic, 1.22].

We will need to speak of \sing{cyclotomic polynomials}\index{$\phi_d$}. So, if $d\geq 1$, recall that $\phi_d(x)\in \bbZ[x]$ denotes the $d$-th cyclotomic polynomial, whose complex roots are the roots of unity of order $d$.

\begin{defn}\label{PSF} Any $F$-stable torus $\bS$ of $\GG$ has a so-called \sing{polynomial order} $P_{\bS ,F}\in \bbZ [x]$ defined by $$|\bS^{F^{m}}|=P_{\bS ,F} (q^m)$$ for some $a\geq 1$ and any $m\in 1+a\bbN$. Moreover $P_{\bS ,F}$ is a product of cyclotomic polynomials $P_{\bS,F}=\Pi_{d\geq 1}\phi_d^{n_d}$, ($n_d\geq 0$).\index{$P_{\bS ,F}$}
\end{defn}

We give below a fundamental example.

\begin{exm}\label{Fmax} Let $\TT_0$ be an $F$-stable maximal torus of a group $(\GG ,F)$ such that $F$ acts on $Y(\TT_0)$ by $q$. Note that this is the case as soon as the coroots $\Phi(\GG,\TT_0)^\vee$ generate a lattice of finite index of $Y(\TT_0)$ and $F$ fixes them. Let $w\in \W(\GG,\TT_0)$ and assume $\TT$ is an $F$-stable maximal torus of type $w$ in the sense of Remark~\ref{4.4}.(c). Then the pair $(\TT ,F)$ is made isomorphic to 
	 $(\TT_0,wF)$ through conjugation by $g\in \GG$ such that $g^\mm F(g)\TT_0=w$. Therefore $|\TT^{F^m}|=|\TT_0^{(wF)^m}|$ for any $m\geq 1$. It is an elementary fact that $\TT_0^{(wF)^m}\cong Y(\TT_0)/(1-{(wF)^m}) Y(\TT_0)$ see [DigneMic, 13.7]. Such a quotient is finite if and only if the endomorphism $1-{(wF)^m}$ of the lattice $Y(\TT_0)$ has non zero determinant and the cardinality of $Y(\TT_0)/(1-{(wF)^m}) Y(\TT_0)$ is $|\det (1-{(wF)^m}) |$. In our case we get the characteristic polynomial of $w^\mm$ at $q^m$. So if $a$ is the order of $w$ and $m\in 1+a\bbN$, we actually get $|\TT^{F^m}|=P_{\TT ,F}(q^m)$ for $P_{\TT ,F}$ the characteristic polynomial of $w$ on $Y(\TT_0)$. It is a product of cyclotomic polynomials since it is a monic polynomial whose zeroes are roots of unity, $w$ having finite order.
	 
	 In all cases of interest, 
	 $F$ induces a map of the form $q\phi$ on $Y(\TT_0)$ where $\phi$ is an automorphism of finite order. Then the above applies almost unchanged. 
\end{exm}

The polynomial orders of tori have many properties of orders of abelian groups, only cyclotomic polynomials play now the role of prime divisors.

\begin{pro}\label{Tphid} Let $\bS$ be an $F$-stable torus of $\GG$. If $P_{\bS,F}=\Pi_{d\geq 1}\phi_d^{n_d}$, then for any $d\geq 1$, there is a unique subtorus $\bS_d\leq \bS$ such that $P_{\bS_d,F}=\phi_d^{n_d}$.
\end{pro}

Indeed an $F$-stable torus $\bS$ of $\GG$ is essentially characterized as a subtorus of a maximal one $\TT$ by the $F$-stable pure sublattice $Y(\bS)$ of $Y(\TT)$.

\begin{defn}\label{dsplit} A \sing{$\phi_d$-torus} of $\GG$ is an $F$-stable torus 
	whose polynomial order is a power of $\phi_d$.
	
	A  \sing{$d$-split Levi subgroup} is any $\cent{\GG}{\bS}$ where $\bS$ is a $\phi_d$-torus of $\GG$.
\end{defn}

\begin{exm}\label{dLevi}
(a) For $d=1$, the $1$-split Levi subgroups are the one that are complements of $F$-stable parabolic subgroups, hence $\GF$-conjugate to the standard Levi subgroups $\LL_I$, for $I\inn S$, $F(I)=I$.

(b)  Let $\GG =\GL_n(\bbF)$ with $F$ the raising of matrix entries to the $q$-th power. Let $\TT_1$ the diagonal torus. By (\ref{FTori}) the $\GF$-classes of maximal tori are indexed by conjugacy classes of $\Sym n$, or equivalently partitions of $n$. For $\la\vdash n$, denote by $\TT_\la$ an $F$-stable maximal torus in the corresponding class. If $\la =(\la_1\geq \la_2\dots )$ then the polynomial order of $\TT_\la$ is $(x^{\la_1}-1)(x^{\la_2}-1)\cdots$ by Example~\ref{Fmax}. One calls
$\TT_{(n)}$ a Coxeter torus.  This maximal torus $\TT_{(n)}$ is the only $n$-split proper
Levi subgroup of
$\GG$ up to $\GF$-conjugation. To see this, note that its polynomial order $x^n-1$ is the only
polynomial order of an $F$-stable maximal torus divisible by $\phi_n$. 

Let $d\geq 1$, $m\geq 0$ such that $md\leq n$. Let $\bS_{(d)}$ be a
Coxeter torus of GL$_d(\bbF )$. Let $\LL^{(m)}$ be ${\rm GL}_{n-md}(\bbF
)\times (\bS_{(d)})^m$ embedded in $\GG ={\rm GL}_n(\bbF )$ via the diagonal subgroup
${\rm GL}_{n-md}(\bbF
)\times ({\rm GL}_{d}(\bbF
))^m$. Then $\LL^{(m)}$ is
$d$-split thanks to the above. A maximal $d$-split proper Levi subgroup $\LL$ of $\GG$ is
isomorphic to $({\rm GL}_m)^d\times {\rm GL}_{n-md}$ with $\LF\cong {\rm
	GL}_m(q^d)\times {\rm GL}_{n-md}(q)$.
\end{exm}

For finite group theorists, \sing{Harish-Chandra theory} consists in relating two elements of $\Irr(\GF)$ whenever they are constituents of the same $\Lu{\LL}{\GG}\zeta$ for  $\zeta\in\Irr(\LL^F)$ and $\LL$ an $F$-stable Levi subgroup of {\bf an $F$-stable parabolic subgroup}. This leads quickly to the notion of \sing{cuspidal characters}, i.e. characters that are in no $\Lu{\LL}{\GG}\zeta$ as above unless $\LL =\GG$. From the fact that $\Lu{\LL}{\GG}\zeta\in \bbN\Irr(\GF)$ it is easy to see that each set in our partition of $\Irr(\GF )$ then coincides with the set of components of some $\Lu{\LL}{\GG}\zeta$ with $\zeta$ a cuspidal character of $\LL^F$.

The idea of Brou\'e-Malle-Michel [BrMaMi93] is to generalize that to $d$-split Levi subgroups in the place of the $1$-split ones considered in Harish-Chandra theory.

\begin{defn}\label{leqd} One writes 
$(\LL_1,\chi_1)\leq _d(\LL_2,\chi_2)$ when $\LL_i$ are $d$-split Levi subgroups in 	$\GG$, $\chi_i\in\ser{\LL_i^F}{1}$ and $\chi_2$ is a component of $\Lu{\LL_1}{\LL_2}\chi_1$.\index{$\leq_d$}

A character $\chi\in\ser{\GF}{1}$ is said to be \sing{$d$-cuspidal} if a relation $(\LL,\zeta) \leq_d (\GG,\chi)$ is possible only with $\LL=\GG$. A pair $(\LL,\zeta)$ with $\LL$ a $d$-split Levi subgroup and a $d$-cuspidal $\zeta\in\ser{\LL^F}{1}$ is called a \sing{unipotent $d$-cuspidal pair} of $\GF$.
\end{defn}

The following is due to Brou\'e-Malle-Michel, building on observations made by Fong-Srinivasan [FoSr86] in non-exceptional types.
One keeps $\GG$, $F$ as before, and $d\geq 1$.
\begin{thm}[{[BrMaMi93, 3.2]}]\label{BMM}  \begin{enumerate}[\rm(i)]
		\item $\leq_d$ is transitive among pairs of the type considered in Definition~\ref{leqd}
		\item If $(\LL ,\zeta)$ is a unipotent $d$-cuspidal pair of $\GF$, then for any component $\chi$ of $\Lu{\LL}{\GG}\zeta$ one has \begin{equation}
		\slu{\LL}{\GG}\chi= N\sum_{g\in \norm{\GG}{\LL}^F/\norm \GF{\LL ,\zeta}}{}^g\zeta
		\end{equation} where $N=\langle\chi,\Lu{\LL}{\GG}\zeta\rangle_\GF\not= 0$.
\item $\ser{\GF}{1}=\dot\cup_{(L,\zeta)}\Irr(\GF\mid\Lu{\LL}{\GG}\zeta)$ where $(\LL,\zeta)$ ranges over $\GF$-conjugacy classes of unipotent $d$-cuspidal pairs  and $\Irr(\GF\mid\Lu{\LL}{\GG}\zeta)$ denotes the set of irreducible components of the generalized character $\Lu{\LL}{\GG}\zeta$.
	\end{enumerate}
\end{thm}

It is fairly clear that (ii) and (iii) above are easy consequences of each other and both consequences of the first point. The proof of the theorem is by a case by case analysis and relies in fact on the explicit description of the sets of unipotent characters and the computation of each Lusztig functor $$\Lu{\LL}{\GG}\colon \ser{\LL^F}{1}\to \bbZ\ser{\GG^F}{1}.$$ Such a computation was done by Asai for classical types ([As84a] and [As84b]) and by Brou\'e-Malle-Michel for exceptional types (see [BrMaMi93, Tables 1,2]). 

Brou\'e-Malle-Michel also give a parametrization of $\Irr(\GF\mid\Lu{\LL}{\GG}\zeta)$ much in the spirit of McKay's conjecture (see [Sp17, 3.A]) on character degrees
\begin{thm}[Brou\'e-Malle-Michel]\label{BMM'}  If $(\LL ,\zeta)$ is a unipotent $d$-cuspidal pair of $\GF$, then one has a bijection $$\Irr(\norm{\GG}{\LL,\zz}^F/\LL^F)\to \Irr(\GF\mid\Lu{\LL}{\GG}\zeta)$$ with good equivariance properties.
\end{thm}

\begin{exm}\label{key}
Let us describe the partition of generalized $d$-Harish-Chandra theory in the case of $\GG =\GL_n(\bbF)$, $\GF=\GL_n(q)$. We have seen in Example~\ref{uGL} the parametrization $$\chi\mapsto R_\chi =\pm n!^\mm \sum_{w\in\Sym n}{\chi(w)}\Lu{\TT_w}{\GG}1$$ of $\ser{\GF}{1}$ by $\Irr(\Sym n)$.  

Let $\LL^{(m)}=\GL_{n-dm}(\bbF)\times (\bS_{(d)})^m$ as in Example~\ref{dLevi}.

Let us compose the above parametrization of $\ser{\GF}{1}$ with the parametrization of $\Irr(\Sym n)$ by partitions of $n$ (see \S~5.E above), thus giving $\la\mapsto\chi_\la^\GG\in \ser{\GF}{1}$. Let $\la\vdash n$, then 
\begin{equation}\label{RLGd}
\slu{\LL^{(1)}}{\GG}(\chi_\la^\GG)=\sum_{\tau\in\hoo_d(\la)}(-1)^{h(\tau)}\chi_{\la-\tau}^{\LL^{(1)}}
\end{equation} where the notations about partitions and hooks is the one of \S~5.E and where we identify $\ser{\LL^{(1)}{}^F}1=\ser{\GL_{n-d}(q)}{1}$.

The proof of (\ref{RLGd}) relies on computing each $\slu{\LL^{(1)}}{\GG}(\Lu{\TT_w}\GG 1)$ by means of a Mackey type formula (see Remark~\ref{4.4} above) and the observation that $\LL^{(1)}$ can contain a $\GF$-conjugate of $\TT_w$ only if $w$ is conjugate in $\Sym n$ to $w'c$ where $w'\in \Sym{n-d}$ and $c=(n-d+1, n-d+2,\dots ,n)$. Then (\ref{RLGd}) is just a consequence of the Murnaghan-Nakayama rule (\ref{MNrule}).

Like in the case of the symmetric group, (\ref{RLGd}) can be iterated as long as $d$-hooks can be removed and one gets the equivalent of (\ref{iMN}), namely
\begin{equation}\label{iRLGd}
\slu{\LL^{(m)}}{\GG}(\chi_\la^\GG)=N\chi_{\kappa}^{\LL^{(m)}}
\end{equation}
where $\kappa\vdash n-md$ is the $d$-core of $\la$ and $N$ is a non-zero integer.

It is not too difficult to show that $\chi_{\kappa}^{\LL^{(m)}}$ is $d$-cuspidal.

So indeed we get enough unipotent $d$-cuspidal pairs $(\LL,\zeta)$ with any $\chi\in\ser{\GF}{1}$ being in one of the disjoint sets $\Irr (\GF\mid \Lu{\LL^{(m)}}{\GG}(\zeta))$.

More work with (\ref{MNrule}) as main ingredient would tell us that the above are all the unipotent $d$-cuspidal pairs and that Theorem~\ref{BMM} holds.
\end{exm}

\subsection{The
	 theorem}

The relation between $\ell$-blocks and $d$-Harish-Chandra theory is given by the following kind of theorem.

\begin{thm}[Cabanes-Enguehard] \label{CuspBl} Let $\GG$ be a reductive group defined over the finite field $\bbF_q$, and let $F\colon \GG\to\GG$ be the associated Frobenius map. Assume $\ell$ is a prime $\geq 7$, not dividing $q$. Let $d$ the (multiplicative) order of $q$ mod $\ell$. Then there is a bijection$$(\LL,\zeta)\mapsto B_\GF(\LL,\zeta)$$\index{$B_\GF(\LL,\zeta)$} between $\GF$-classes of unipotent $d$-cuspidal pairs and unipotent blocks (see Definition~\ref{ublo}). One has 
\begin{enumerate}[\rm(i)]
	\item $\Irr(B_\GF(\LL,\zeta))\cap\ser{\GF}{1}=\Irr(\GF\mid\Lu{\LL}{\GG}\zeta)$,
	\item the Sylow $\ell$-subgroups of $\ccent{\GG}{[\LL ,\LL]}^F$ are defect groups of $B_\GF(\LL,\zz)$.
\end{enumerate}
\end{thm}

The theorem has many precursors, first of all by Fong-Srinivasan ([FoSr82] and [FoSr89]) who treat all blocks (not just unipotent) for classical groups. Note that it is possible to show that just like unipotent characters are insensitive to the center of the group, unipotent blocks are basically the same for all groups of same type and rank (see [CaEn, \S 17]), so the above could be deduced from Fong-Srinivasan's work in many cases. We have chosen the statement for its simplicity and its relatively straightforward proof sketched in the next section.

The theorem essentially relates the splitting of $\ser{\GF}{1}$ into $\ell$-blocks with the Lusztig functor. More theorems of the same kinds were given by Cabanes-Enguehard (all $\ell$-blocks, $\ell\geq 5$ [CaEn99]), Enguehard (unipotent blocks for all primes [En00]) and recently by Kessar-Malle (all blocks and primes [KeMa15]). Note that given Bonnaf\'e-Dat-Rouquier's theorem showing equivalence of blocks in a very strong sense with blocks of generally smaller groups, the above is interesting only for blocks in $\lser{\GF}{s}$ (see Definition~\ref{BGs}) where $\ccent{\GD}{s}$ can't be embedded in a proper Levi subgroup of $\GD$ ("isolated" series), which brings us close to unipotent blocks.

The statement by Kessar-Malle is as follows. Here $\GG$ is assumed to be an $F$-stable Levi subgroup of a simple simply connected group. One keeps $\ell$ a prime not dividing $q$ and $d$ the order of $q$ mod $\ell$ when $\ell$ is odd, while $d$ is the order of $q$ mod 4 when $\ell=2$. One denotes $\cE(\GG^F,\ell')$ the union of rational series $\cE(\GG^F, s)$ with $s\in \GG^*{}^{F^*}$ semi-simple of order prime to $\ell$.

\begin{thm}[{[KeMa15, Th. A]}]
 \label{thmA}
	\begin{enumerate}[\rm(i)]
		\item For any $d$-Jordan-cuspidal pair $(\LL,\la)$ of $\GG$ such that
		$\la\in\cE(\LL^F,\ell')$, there exists a unique $\ell$-block
		$b_{\GG^F}(\LL,\la)$ of $\GG^F$ such that all irreducible constituents
		of $\Lu\LL\GG(\la)$ lie in $b_{\GG^F}(\LL,\la)$.
		\item The map \begin{equation}\label{Xi}
		\text{  $  (\LL,\la)\mapsto b_{\GG^F}(\LL,\la)  $ }
		\end{equation} is a
		surjection from the set of $\GG^F$-conjugacy classes of $d$-Jordan-cuspidal
		pairs $(\LL,\la)$ of $\GG$ such that $\la \in \cE(\LL^F,\ell')$ to the set
		of $\ell$-blocks of $\GG^F$.
		\item The map (\ref{Xi}) restricts to a surjection from the set of
		$\GG^F$-conjugacy classes of $d$-Jordan quasi-central cuspidal pairs
		$(\LL,\la)$ of $\GG$ such that $\la \in \cE(\LL^F,\ell')$ to the set
		of $\ell$-blocks of $\GG^F$.
		\item For $\ell\ge3$ the map (\ref{Xi}) restricts to a bijection between the set
		of $\GG^F$-conjugacy classes of $d$-Jordan quasi-central cuspidal pairs
		$(\LL,\la)$ of $\GG$ with $\la \in \cE(\LL^F,\ell')$ and the set of
		$\ell$-blocks of $\GG^F$.
		\item The map (\ref{Xi}) is bijective if $\ell\geq 3$ is a good prime for $\GG$,
		and $\ell\ne 3$ if $\GG^F$ has a factor $\tw3D_4(q)$.
	\end{enumerate}
\end{thm}

Here the notion of $d$-Jordan cuspidal character (or pair) is adapted from the unipotent case through Jordan decomposition. Quasi-central means belonging to a block of $\LF$ covering a block of $[\LL ,\LL]^F$ of central defect (see [KeMa15, \S 2]).




{}
\bigskip

{}

\section{Local methods: unipotent blocks and $d$-Harish-Chandra theory}

The proofs of Theorems~\ref{CuspBl}, \ref{thmA} follow the pattern described in \S~5.D above through subpair enlargement and use of Brauer's second main Theorem.

\subsection{The main subpair inclusion}

\begin{lem}[{see [DigneMic, 13.15.(i)]}] \label{NonC}
	If $x\in \GG^F_\pp$ then $\cent{\GG}{x}/\ccent{\GG}{x}$ has exponent dividing the order of $x$ and injects into $\Z(\GD)/\Z^\circ(\GD)$.
\end{lem}

The following compatibility between generalized decomposition maps $d^x$ and $\slu{\LL}{\GG}$ functors is of crucial importance. It was first spotted by Fong-Srinivasan [FoSr82, (2C)].

\begin{pro}\label{dx*R}
	Let $\PP=\RU(\PP)\LL$ be a Levi decomposition in $\GG$ with $F$-stable $\LL$. Let $\ell$ a prime $\ne p$, let $x\in \LL^F_\ell$. Then $\ccent \GG x$ is an $F$-stable reductive group and $\ccent{\PP}{x}=\ccent{\RU(\PP)}{x}\ccent{\LL}{x}$ the Levi decomposition of a parabolic subgroup. Moreover $$d^{x,\LF}\circ\slu{\LL\inn\PP}{\GG}=\slu{\ccent{\LL}{x}\inn\ccent{\PP}{x}}{\ccent{\GG}{x}}\circ d^{x,\GF}$$ on $\CF(\GF,K)$.
\end{pro}

\begin{proof}
  The group theoretic part of the proposition is classic and was already used in our statement of the character formula (Proposition~\ref{chfla}). The composition $\slu{\ccent{\LL}{x}\inn\ccent{\PP}{x}}{\ccent{\GG}{x}}\circ d^{x,\GF}$ makes sense thanks to the inclusion $\cent{\GG}{x}^F_\lp\inn {\ccent{\GG}{x}}^F$ ensured by Lemma~\ref{NonC}. The formula itself is an easy consequence of the character formula.
\end{proof}
Though we will apply this property mainly to unipotent blocks, it is fundamental to the proof of a theorem of Brou\'e-Michel on general sums of blocks $e_\ell (\GF,s)$ (see \S~4.C, Definition~\ref{BGs}).

We keep $\GG$, $F$ as before and $\ell$ some prime $\not= p$.

\begin{thm}[Brou\'e-Michel {[BrMi89]}]\label{Brell} Let $s\in (\GD)^{F^*}_\lp$ a semi-simple element and $e_\ell(\GF,s)$ the central idempotent of $\OO\GF$ associated (see Definition~\ref{BGs}). Denote $\ov e_\ell(\GF,s)$ its image in $k\GF$. Let $x\in\GG^F_\ell$. Then $$\Br_x(\ov e_\ell(\GF,s))=\sum_{t\ \mid\ i_x(t)=s}\ov e_\ell(\ccent{\GG}{x}^F,t),$$ where $i_x$ is a map associating conjugacy classes of semi-simple elements of $\GG^*{}^{F^*}$ to conjugacy classes  of semi-simple elements of $\ccent{\GD}{x}^*{}^{F^*}$ through pairs $(\TT^* ,t)\to (\TT ,\th)\to (\TT_1^*,s)$ using (\ref{dualT}) above.
\end{thm}

\begin{proof} Through Brauer's second Main Theorem it is easy to see that the main statement is equivalent to checking that 
	\begin{equation}\label{dxP}
	\text{  $  (d^x\circ P_s^{(\GF)})(\gamma_{\GF.x})= \sum_{t\ \mid\ i_x(t)=s}P_t^{(\ccent{\GG}{x}^F)}(\gamma_{1}) $ }
	\end{equation}
	where $ P_s^{(\GF)}\co \CF(\GF)\to \CF(\GF,B_\ell(\GF,s))$ is the projection and $\gamma_{\GF.x}$ is the function being 1 on the conjugacy class of $x$ and 0 elsewhere. One has $\gamma_{1}=d^x(\gamma_{\GF.x})$ and $\gamma_{\GF.x}$ is uniform (apply Lemma~\ref{omni}.(i)), so it is easy to reduce (\ref{dxP}) to the following 	\begin{equation}\label{dxRTG}
	\text{  $  d^{x,\GF}\circ\Lu{\TT}{\GG}=|\ccent{\GG}{x}^F|^\mm \sum_{h\in\GF\ \mid\ x\in {}^h\TT}\Lu{^h\TT}{\ccent{\GG}{x}}\circ d^{x,^h\TT^F} \circ \ad_h$ }
	\end{equation} where $\ad_h$ is the conjugation by $h$ of central functions. This in turn can be deduced from Proposition~\ref{dx*R} by taking adjoints. Note that, for $x$ an $\ell$-element of a finite group $H$, the adjoint of $d^x\colon\CF(H,K)\to \CF(\cent Hx,K)$ is the map sending the central function $f\colon \cent Hx\to K$ to $f'\colon H\to K$ defined by $$f'(h)=|\cent Hx|^\mm\sum_{v\in H\ \mid\ h_\ell = vxv^\mm}f(x^\mm h^v).$$
\end{proof}

Now for unipotent blocks and with the aim of proving \Th{CuspBl}, the main step is achieved by the following

\begin{thm}\label{GenInc} Let $\LL$ be an $F$-stable Levi subgroup of $\GG$ and $\zeta\in\ser{\LF}{1}$. Assume \begin{enumerate}[\rm(a)]
		\item $\LF=\cent{\GG}{\Z(\LL)_\ell^F}^F$, and
		\item for all $A\leq \Z(\LL)_\ell^F$ and $\chi_A$ an irreducible component of $\Lu{\LL}{\ccent{\GG}{A}}\zeta$, denoting $\HH={\ccent{\GG}{A}}$, one has $$\slu{\LL}{\HH}\chi_A =\langle\slu{\LL}{\HH}\chi_A ,\zeta\rangle_\LF\sum_{g\in\norm\HH\LL^F/\norm{\HH}{\LL,\zz}^F}{}^g\zeta.$$
	\end{enumerate}Then for any irreducible component $\chi$ of $\Lu{\LL}{\GG}\zeta$ one has an inclusion of $\ell$-subpairs $$(\{1\},b_\GF(\chi))\leq (\Z(\LL)_\ell^F,b_\LF(\zz)).$$
\end{thm}

\begin{proof} 
	One uses an induction on $|\GF/\LF|$. Everything is clear when $\GF =\LF$.
	
	Assume $\LF\ne\GF$, so that, thanks to (a), one can pick $z\in \Z(\LL)^F_\ell$ non central in $\GF$. Let $\HH\deq\ccent{\GG}{z}\nni\LL$. Let us show\begin{equation}\label{non0}
	\text{  $  \langle d^{z,\GF}\chi , \Lu{\LL}{\HH}\zz\rangle_{\HH^F}\ne 0.  $ }
	\end{equation}
	One has indeed $d^{z,\GF}\chi\in\CF(\HH^F,K)$ since $\cent{\GG}{z}^F_\lp\inn \HH^F$ thanks to Lemma~\ref{NonC}. We have 
	\begin{eqnarray*}
	 \langle d^{z,\GF}\chi , \Lu{\LL}{\HH}\zz\rangle_{\HH^F}&=& \langle  d^{z,\LF} \slu{\LL}{\GG}\chi ,\zz\rangle_{\LL^F} \text{ by Proposition~\ref{dx*R}} \\
		&=&\langle \Lu{\LL}{\GG}\zz,\chi\rangle_\GF\sum_{\zz'\in\norm{\GG}{\LL}^F.\zz}\langle  d^{z,\LF}\zz',\zz\rangle_\LF \text{  \  by (b) with }A=\{1\}  \\
		&=&\langle \Lu{\LL}{\GG}\zz,\chi\rangle_\GF\sum_{\zz'\in\norm{\GG}{\LL}^F.\zz}\langle d^{1,\LF}\zz',\zz\rangle_\LF  \text{ since  
			 $z\in\Z(\LL)^F\leq\ker(\zeta)$} \\
		&=& \langle \Lu{\LL}{\GG}\zz,\chi\rangle_\GF\sum_{\zz'\in\norm{\GG}{\LL}^F.\zz}\langle  d^{1,\LF}\zz',d^{1,\LF}\zz\rangle_\LF \\
		&=&\langle \Lu{\LL}{\GG}\zz,\chi\rangle_\GF|\norm{\GG}{\LL}^F.\zz|^\mm \langle  f,f\rangle_\LF    \end{eqnarray*} $ \text{ for  } f:=\sum_{\zz'\in\norm{\GG}{\LL}^F.\zz}d^{1,\LF}\zz'$. But 
$f\in\CF(\LF,K)$ is clearly a central function such that $f(x^\mm)$ is the complex conjugate of $f(x)$ for any $x\in\LF$ and $f\not=0$ by the value at 1. So $\langle  f,f\rangle_\LF\ne0$ and we get (\ref{non0}) from the above. 

Now (\ref{non0}) implies that there is an irreducible component $\chi _\HH$ of $\Lu{\LL}{\HH}\zz$ such that $d^{z,\GF}\chi$ has a non-zero projection on $\CF(\HH^F\mid b_{\HH^F}(\chi_\HH))$.

One may apply the induction hypothesis to $\HH, \LL,\zz$ replacing $\GG,\LL,\zz$ since (a) and (b) are clearly satisfied there. The fact that $\chi _\HH$ is a component of $\Lu{\LL}{\HH}\zz$ implies the subpair inclusion $$(\{1\},b_{\HH^F}(\chi_\HH))\leq (Z,b_\LF(\zz)) \text{ in }  \HH^F$$ where we abbreviate $Z=\Z(\LL)^F_\ell$. Assume $\HH^F=\cent{\GF}{z}$. Then it is easy to deduce from the above the subpair inclusion 
\begin{equation}\label{zZ}
\text{  $ (\spann <z>,b_{\HH^F}(\chi_\HH))\leq (Z,b_\LF(\zz))   $\ \ \ \ \ in\ \ \ $\GF$.}
\end{equation}
	On the other hand the fact that $d^{z,\GF}\chi$ has a non-zero projection on $\CF(\HH^F\mid b_{\HH^F}(\chi_\HH))$ implies that we have \begin{equation}\label{1z}
	\text{  $(\{1\},b_\GF(\chi))\leq (\spann <z>,b_{\HH^F}(\chi_\HH))  $\ \ \ \ \ in\ \ \ $\GF$}
	\end{equation} thanks to Brauer's second Main Theorem. We then get our claim from (\ref{zZ}) and (\ref{1z}) by transitivity of subpair inclusion.
	
	We have assumed for simplification that $\HH^F=\cent{\GF}{z}$. In general we only have $\HH^F\lhd\cent{\GF}{z}$ with index a power of $\ell$ thanks to Lemma~\ref{NonC}. Then it is easy to define the unique block $b'$ of $\cent{\GF}{z}$ covering $b_{\HH^F}(\chi_\HH)$ and prove the analogues of (\ref{zZ}) and (\ref{1z}) with it.
\end{proof}

\subsection{$\phi_d$-tori and $\ell$-subgroups}

We keep $\GG$, $F$ as before over $\bbF_q$, and $\ell$ a prime $\nmid q$. We also assume now that $\ell\geq 7$.

Note that $\ell$ divides $\phi_m(q)$ if and only if $m_\lp =d$ (see for instance [Serre, \S~II.3.2]).
\begin{pro}\label{ZellF}
Assume $\ell$ divides $\phi_m(q)$ but neither $|\Z(\GG)^F/\Z^\circ (\GG)^F|$ nor\break $|\Z(\GG^*)^F/\Z^\circ (\GG^*)^F|$. Let $\bS$ be a $\phi_m$-torus (see Definition~\ref{dsplit}), $\LL\deq\cent{\GG}{\bS}$, $Z\deq \Z(\LL)$. Then \begin{enumerate}[\rm(i)]
	\item $\LL=\ccent{\GG}{\bS_\ell^F}=\ccent{\GG}{Z_\ell^F}$ and
	\item $\LF=\ccent{\GG}{Z_\ell^F}^F=\cent{\GF}{Z_\ell^F}.$
\end{enumerate}
\end{pro}

\begin{proof} (i) It suffices to check the first equality. We show it by induction on the dimension of $\GG$. Let $\pi :\GG\to\GG_\ad \deq \GG/\zent{\GG}$ the reduction mod $\zent{\GG}$. 
	
 By a classical argument we have an exact sequence
$$ 1\to \pi(\bS^F)\to \pi (\bS)^F\to [\bS,F]\cap {\rm Z}(\GG )/[{\rm Z}(\GG ) ,F]\to 1 .$$ By
Lang's theorem, $[{\rm Z}(\GG ),F]\supseteq {\rm Z}^\circ (\GG )$, so $[\bS,F]\cap {\rm
	Z}(\GG )/[{\rm Z}(\GG ) ,F]$ is a section of 
$ {\rm Z}(\GG ) /{\rm Z}^\circ(\GG )$ on which the action of $F$ is trivial.
But $\ell$ does not divide $|\zent{\GG}^F/\czent{\GG}^F|$ so
$([\bS  ,F]\cap {\rm Z}(\GG ) )/[{\rm Z}(\GG )  ,F]$ is
$\ell'$~; thus
$\pi(\bS )_\ell^F\inn\pi(\bS^F)$. Moreover, if $s$ is of finite order, then
$\pi(s_\ell)=\pi(s)_\ell$. This implies 
\begin{equation}\label{piZFell}
\text{  $ \pi(\bS )_\ell^F=\pi(\bS^F_\ell)   $ .}
\end{equation} 

Now denote
$\CC\deq \ccent\GG {\bS ^F_\ell }$.  The fact that $\ell\geq 7$ eliminates some exceptional behaviour (``bad" primes\index{good primes}, see [GeHi91, 2.1] or [CaEn, \S 13.2]) and ensures that $\CC$ is a Levi subgroup of $\GG$.
One has clearly
$\LL\inn\CC$. If $\CC\not= \GG$, then the induction hypothesis gives $\LL =\CC$,
that is our claim.

Assume $\CC = \GG $, that is $\pi (\bS ^F_\ell )=\{1\}$. By (\ref{piZFell}), this implies $\pi
(\bS )^F_\ell =\{1\}$. But $\pi (\bS )$ is a $\phi_m$-torus of
$\GG_\ad$ whose number of fixed points under $F$ is a power of $\phi_m(q)$.
This is prime to $\ell$ only if this exponent is 0, that is $\bS_{}
\inn\czent\GG$. This implies $\LL =\GG$ and our claim is trivial.

(ii) The first equality comes from (i). For the second we have an inclusion $\ccent{\GG}{Z_\ell^F}^F\lhd \cent{\GF}{Z_\ell^F}$. But the factor group is trivial thanks to Lemma~\ref{NonC} and the hypothesis on $\ell$ with regard to $\GD$.
\end{proof}

\begin{cor}\label{dpair} Let $\ell$ be a prime $\geq 7$ and $\ne p$. Let $d$ be the order of $q$ mod $\ell$.
Let $(\LL,\zz)$ be a unipotent $d$-cuspidal pair of $\GG$. Then \begin{enumerate}[\rm(i)]
	\item $\LF=\cent{\GF}{Z_\ell^F}$ and 
	\item $\zz (1)_\ell =|\LF/\zent \LL^F|_\ell=|\LF/\czent \LL^F|_\ell$.
\end{enumerate}
\end{cor}

\begin{proof}
(i) To deduce this from Proposition~\ref{ZellF}.(ii) we essentially have to show that the condition $\ell\nmid |\Z(\GG)^F/\Z^\circ (\GG)^F|.|\Z(\GG^*)^F/\Z^\circ (\GG^*)^F|$ can be assumed. Since $ \ell$ is large, this concerns chiefly groups of type $\tA_{n-1}$ with $\ell$ dividing $q-\eps$ with $\eps=1$ or $-1$ according to the action of $F$ on roots being trivial or not, respectively (see [CaEn, 13.11]). Let $\TT_1$ be the diagonal maximal torus. That is the one whose image in $\GG_\ad$ is such that $F$ acts trivially on the associated Weyl group. Then we have 
\begin{equation}\label{L=T1}
\text{$\LL=\TT_1$ and   $  \LF=\cent{\GF}{\LL^F_\ell}  $. }
\end{equation} 
Indeed one can then assume $d=1$ or $2$ according to $\eps =1$ or $-1$. On the other hand it is well-known that $\ser{\GF}{1}$ is the set of components of $\Lu{\TT_1}{\GG}1$, so $(\TT_1,1)$ is the only unipotent $d$-cuspidal pair. This forces $\LL=\TT_1$. The second statement is an easy verification in PGL.

(ii) The degrees of $d$-cuspidal characters are known from [BrMaMi93] and, up to integral scalars involving only bad primes, they are polynomials in $q$ where the power of each $\phi_{d\ell^a}$ is the same as in the order of the group. 
\end{proof}

\subsection{Defect groups}

We finish to review the proof of Theorem~\ref{CuspBl} whose hypotheses we keep.

We have a unipotent $d$-cuspidal pair $(\LL,\zz)$ and we have seen that if $ (\LL,\zz)\leq_d(\GG ,\chi)$ (see Definition~\ref{leqd}) then one has $\LF= \cent{\GF}{\zent{\LL}_\ell^F}$ and the inclusion of $\ell$-subpairs 
\begin{equation}\label{1stIncl}
\text{  $  (\{1\},b_\GF(\chi))\leq (\zent{\LL}_\ell^F,b_\LF(\zz))  $. }
\end{equation}
This is obtained by applying Corollary~\ref{dpair}.(i) and Theorem~\ref{GenInc} above.
Note that this already implies that we can define $B_\GF(\LL,\zz)$ as the $\ell$-block $B$ such that $\Irr(B)$ contains all irreducible components of $\Lu{\LL}{\GG}\zz$. Concerning defect groups we prove
\begin{pro}\label{DefUn} Let $D$ be a Sylow $\ell$-subgroup of $\cent{\GG}{[\LL,\LL]}^F$ containing $\zent{\LL}^F_\ell$. Then $\Res^\LF_{\cent{\GF}{D}} \zz$ is irreducible and
	$$(\Z(\LL)^F_\ell, b_\LF(\zz))\leq (D,b_{\cent \GG D^F}(\Res^\LF_{\cent{\GF}{D}} \zz)).$$ Both subpairs are centric and the second is a maximal subpair.
\end{pro}

\begin{proof} For the first statement  notice that $[\LL,\LL]^F\leq\cent{\GF}{D}\leq\cent{\GF}{\zent{\LL}^F_\ell}= \LF$. Unipotent characters of $\LF$ restrict irreducibly to $[\LL,\LL]^F$ thanks to (\ref{uCh}) above, hence irreducibly to $\cent{\GF}{D}$. From Corollary~\ref{dpair} we know that $(\Z(\LL)^F_\ell, b_\LF(\zz))$ is a centric subpair. The subpair inclusion of the proposition is then easily checked by applying 
Theorem~\ref{inctric}. The maximality of the subpair $(D,b_{\cent \GG D^F}(\Res^\LF_{\cent{\GF}{D}} \zz))$ is not too difficult, see also Remark~\ref{B0G} below.	\end{proof}

We have now proved almost all of Theorem~\ref{CuspBl}. There remains to show that if two unipotent $d$-cuspidal pairs $(\LL_1,\zz_1)$,  $(\LL_2,\zz_2)$ are such that  $B_\GF(\LL_1,\zz_1)=B_\GF(\LL_2,\zz_2)$ then they are $\GF$-conjugate. In such a case the maximal subpairs given by Proposition~\ref{DefUn} would be $\GF$-conjugate by \Th{subp}.(ii). But since in Proposition~\ref{DefUn}, one has $[\LL,\LL]\leq\cent{\GG}{D}\leq \LL$ and therefore $[\LL ,\LL]=[\cent{\GG}{D},\cent{\GG}{D}]$, one gets easily that the pairs $([\LL_i,\LL_i], \Res^{\LL_i}_{[\LL_i,\LL_i]}\zz_i)$ are $\GF$-conjugate. The lemma below shows that one may assume $(\LL_1, \Res^{\LL_1}_{[\LL_1,\LL_1]}\zz_1)=(\LL_2, \Res^{\LL_2}_{[\LL_2,\LL_2]}\zz_2)$. But then $\zz_1=\zz_2$ by (\ref{uCh}).

\begin{lem}\label{LL}
If two $d$-split Levi subgroups $\LL_1$, $\LL_2$ of $\GG$ have same derived subgroup, then they are $\ccent{\GG}{[\LL_1,\LL_1]}^F$-conjugate.
\end{lem}

To show this one notices first that $\CC\deq\ccent{\GG}{[\LL_i,\LL_i]}$ is a reductive group where $\czent{\LL_i}$ is a maximal torus. Moreover $\czent{\LL_i}_{\phi_d}$ is a maximal $\phi_d$-torus in $\CC$. Both properties are by computing the centralizers in $\CC$ and remembering that $\LL_i=\cent{\GG}{\czent{\LL_i}_{\phi_d}}$ by definition. But then the Sylow theorem for maximal $\phi_d$-tori (see [BrMa92], [CaEn, 13.18]) implies our claim.

\begin{rem}\label{B0G} Using \Th{CuspBl} the principal $\ell$-block of $\GF$ is described in the following fashion. It corresponds to $(\LL,\zz)$ with $\LL=\cent{\GG}{\bS}$ where $\bS$ is a maximal $\phi_d$-torus of $(\GG ,F)$ and $\zz=1$ is the trivial character of $\LF$. With the hypothesis we have on $\ell$ (which may be loosened to include $\ell =5$) one can prove that the defect group $D$ may be taken normalizing $Z\deq \zent{\LL}^F_\ell$ with the additional property that 
	\begin{equation}\label{ZinD}
	\text{  $ Z$ is the unique maximal abelian normal subgroup of $D   $ }
	\end{equation}
	 (see [Ca94]). This gives a quite handy property of Sylow $\ell$-subgroups of finite groups of Lie type for the transversal primes. The exceptions for the primes $2,3$ are given in [Ma07, 5.14, 5.19]. Indeed the subgroup $Z$ is therefore characteristic in $D$. This can help conclude about the maximality of the subpair proposed in Proposition~\ref{DefUn} (see {[KeMa15]}). 
	 
	 One also has $M\deq\norm{\GG}{\bS}^F=\norm{\GF}{Z}\geq\norm{\GF}{D}$ and any automorphism of $\GF$ preserving $D$ will obviously preserve $M$ thanks to (\ref{ZinD}). This explains why, paving the way for future checkings in finite groups of Lie type, the inductive conditions for McKay or Alperin-McKay conjectures consider a possible overgroup for the normalizer of the defect group involved in the original statement of the conjectures (see {[IsMaNa07, \S 10 (2)]}, [Sp13, 7.2]).
\end{rem}
\subsection{Non-unipotent characters of unipotent blocks}

Brauer's second Main Theorem can also be used to give a complete description of $\Irr(B_\GF(\LL,\zz))$ (see \Th{CuspBl} above) in terms of Lusztig's functor.

We keep the context of  \Th{CuspBl} where $\ell$ is a prime $\geq 7$ different from the defining characteristic of $\GF$. We assume $(\GG ,F)$ is in duality with some $(\GD, F^*)$. Let $t\in(\GD)^{F^*}_\ell$. Then $\ccent{\GD}{t}$ is a Levi subgroup of $\GD$, which by duality yields an $F$-stable Levi subgroup $\GG(t)$ of $\GG$ and a linear character $\wh t\colon\GG(t)^F\to \bbC^\times$. A quite elementary generalization of \Th{LCGs} (see [CaEn, 15.10]) shows that there is a sign $\eps_{\GG,t}\in\{1,-1\}$ such that we get a map 
\begin{eqnarray*}
\psi_t\ \colon	\ \ser{\GG(t)^F}{1}&\to& \bbN\ser{\GF}{t} \text{\ \ \ \ \   by}\\
\chi_t	&\mapsto&\eps_{\GG,t}\Lu{\GG(t)}{\GG}(\wh t\chi_t) . \end{eqnarray*}
\begin{thm}[See [CaEn, \S 23.1]\label{NoUn}
Keep the hypotheses of \Th{CuspBl}. Let $(\LL,\zz)$ be a unipotent $d$-cuspidal pair in $\GG$. Let $t\in(\GD)^{F^*}_\ell$, $\chi_t\in  \ser{\GG(t)^F}{1}$. Then $\psi_t(\chi_t)$ has components in $\Irr(B_\GF(\LL,\zz))$ if and only if there is a unipotent $d$-cuspidal pair $(\LL_t,\zz_t)$ in $\GG(t)$ such that \begin{enumerate}[\rm(i)]
	\item $(\LL_t,\zz_t)\leq_d (\GG(t),\chi_t)$ in $\GG(t)$, with
	\item $[\LL,\LL]=[\LL_t,\LL_t]$ and $\Res^{\LL_t}_{[\LL_t,\LL_t]}\zz_t=\Res^{\LL}_{[\LL,\LL]}\zz$.
\end{enumerate}

Then all components of $\psi_t(\chi_t)$ are in $\Irr(B_\GF(\LL,\zz))$.
\end{thm}

Note that condition (ii) above implies that $t$ must be in the centralizer of $[\LL^*,\LL^*]$ for $\LL^*$ an $F^*$-stable Levi subgroup of $\GD$ corresponding to $\LL$ by duality. This condition looks like dual to the condition for an element of $\GF$ of being in the defect group of $B_\GF(\LL,\zz)$, see \Th{CuspBl}.(ii).

The proof of \Th{NoUn} is by using Proposition~\ref{dx*R} with $x=1$. One gets $$d^1\psi_t(\chi_t)=\pm d^1 (\Lu{\GG(t)}{\GG}\chi_t)$$ which by Brauer's second Main Theorem must have a non-zero projection on $B_\GF(\LL,\zz)$. This reduces the theorem to a question about unipotent characters. It is solved by studying a bit more the relation between $d$-split Levi subgroups and centralizers of $\ell$-subgroups beyond what has been seen in \S 7.B above (see [CaEn, \S 23.1]).

\subsection{Unipotent blocks are non-exotic}
 
 One of the main questions about blocks of quasi-simple groups in relation with fusion systems is to relate their fusion systems (see Definition~\ref{FDbD}) with the ones of finite groups, or equivalently principal blocks (Open problems 1 and 3 in [AschKeOl, Sect. IV.7]). Fusion systems on a $p$-group that are not isomorphic to a $\cF_Q(H)$ (see Definition~\ref{FusCat}) for a Sylow $p$-subgroup $Q$ of a finite group $H$ are called exotic\index{exotic fusion system}. See Remark~\ref{FusSn} above for the case of symmetric groups; a similar result is also known for general linear and unitary groups [Bro86, 3.8].
 
 We show here the same in the context of \Th{CuspBl}. In other words unipotent $\ell$-blocks ($\ell\geq 7$) are non-exotic (sorry). This builds on an earlier theorem [CaEn99b] showing control of fusion in the sense of [Thev, \S 49], a slightly weaker statement.

 \begin{thm}[{}]\label{NonExo} Keep the assumptions and notations of \Th{CuspBl}. Let $(\LL,\zz)$ be a unipotent $d$-cuspidal pair and $B_\GF(\LL,\zz)$ the associated $\ell$-block of which a Sylow $\ell$-subgroup $D$ of $\ccent{\GG}{[\LL ,\LL]}^F$ is a defect group.  \begin{enumerate}[\rm(i)]
 		\item There exists a subgroup
 		$H\leq {\rm N}_\GF ([\LL ,\LL ],\Res^{\LL^F}_{[\LL ,\LL ]^F}\zeta)$ such that 
 	\begin{enumerate}[\rm(a)]
 		\item $D$ is a Sylow $\ell$-subgroup of $H$, and
 		\item 
 		  $H[\LL ,\LL
 		]^F={\rm N}_\GF ([\LL ,\LL ],\Res^{\LL^F}_{[\LL ,\LL ]^F}\zeta)$~.
 	\end{enumerate}
 		\item For any $H$ satisfying the above, the fusion system of  $B_\GF(\LL,\zz)$ is isomorphic to $\cF_D(H)$.
 	\end{enumerate}
 \end{thm}

\begin{proof} The first point is purely group theoretic. The proof uses basically considerations in $\norm \GG\TT$, where $\TT$ is a maximally split torus of ${[\LL ,\LL]}\ccent{\GG}{[\LL ,\LL]}$, see [CaEn, Ex 23.1], [CaEn99b, 6-7] for all details. We now prove (ii). 
	
	Denote  $Z:=\zent \LL ^F_\ell$, $\CC\deq \ccent\GG {[\LL,\LL] }$. Note that $\CC^F\lhd H$. Recall from (\ref{1stIncl}) and Proposition~\ref{DefUn} that we have subpair inclusions in $\GF$
\begin{equation}\label{1ZD}
\text{  $  (\{1\},B_\GF(\LL,\zz))\leq (Z,b_\LF(\zz))\leq (D,b_D)  $ }  
\end{equation}
	where the middle one is centric and $(D,b_D)$ is maximal with $$b_D=b_{\cent{\GF}{D}}(\Res_{\cent{\GF}{D}}^\LF\zz).$$ For $X\leq D$ denote $(X,b_X)$ the unique subpair of $\GF$ such that $$(X,b_X)\leq (D,b_D).$$ Our category isomorphism will be (see Definition~\ref{FDbD}) \begin{eqnarray*}
		\cF_{(D,b_D)}(b_\GF(\LL,\zz))&\to&\cF_D(H), \\
		(X,b_X)&\mapsto&X . \end{eqnarray*}

From the theory of fusion systems, essentially the fact that ``$\cF$-essential objects are $\cF$-centric" see [AschKeOl, \S I.3] and the identification of centric objects in those categories with what we have called so until now [AschKeOl, IV.3.20], it suffices to check for $X\leq D$ 
\begin{equation}\label{NorCic}
\text{  $  \norm{\GF}{(X,b_X)}=\norm{H}{X}\cent{\GF}{X}  $ if $X$ in $H$ or $(X,b_X)$ in $\GF$ is centric.}
\end{equation}

Note that $\CC^F\lhd H$ with $\lp$ index by assumption (i.a), so a subgroup of $D$ is centric in $H$ if and only if it is centric in $\CC^F$.

Let us recall the decomposition $\GG =\GA\GB$ associated to a pair $(\GG ,F)$ and a prime $\ell$ (see [CaEn, 22.4]). In the decomposition of $$[\GG ,\GG]=\GG_1\dots \GG_m$$ as a central product of $F$-stable closed subgroups, one defines $\GA=\czent{\GG}\GG_\aaa'$ where $\GG_\aaa'$ is the subgroup generated by the $\GG_i$'s such that $(\GG_i)^F_\ad\cong {\mathrm {PGL}}_{n_i}(\eps_i q^{m_i})$ with $\ell$ dividing $  q^{m_i}-\eps_i $. The other $\GG_i$'s generate by definition $\GB$. From the properties of $|\zent{\GG_\SC}^F|$ according to the type of $(\GG ,F)$ it is easy to see that 
\begin{equation}\label{ZGb}
\text{  $ \zent{\GB}^F$ and $\GF/\GG_\aaa^F\GG_\bbb^F$ are abelian $\lp$-groups. }
\end{equation} An important (but easy) consequence for centric $\ell$-subgroups is the following [CaEn, 22.5.(ii)], where $Y$ denotes an $\ell$-subgroup of $\GF$:
\begin{equation}\label{CGY}
\text{ If\ \  $ \zent{\cent{\GF}{Y}} _\ell \inn \GA $ \ then \ $Y\inn \GA$. }
\end{equation}

Let us define $\KK\deq \GA{\mathrm C}^\circ_\GG (\zent D)$. Arguing as in the proof of Proposition~\ref{ZellF}, one sees that $\KK$ is a Levi subgroup of $\GG$ such that $\KK\nni \LL=\bS\KK_\bbb$, where $\bS$ is a diagonal torus of $\KK_\aaa$ (therefore $[\LL,\LL] =\KK_\bbb$), and $D\inn\KK_\aaa$. 

By (\ref{uCh}), restriction maps induce bijections $$\ser{\KK^F}{1}\cong \ser{\KK_\aaa^F}{1} \times \ser{\KK_\bbb^F}{1}\cong \ser{\KK_\aaa^F}{1}\times \ser{\LL^F}{1}.$$ We then define $\w \zz\in\ser{\KK^F}{1}$ corresponding to $(1_{\KK_\aaa^F},\zz)$ in the last product.

Assume $X\leq D$ is either centric in $\CC^F$ or $(X,b_X)$ is centric in $\GF$. By Proposition~\ref{maxic}, $\zent D\inn Z\cap \zent X$ and therefore $\KK$ contains ${\mathrm C}_\GF (X)$, and ${\mathrm C}^\circ_\GG (X)$. Iterating the above (\ref{CGY}) it is easy to see that $X\inn\KK_\aaa$.

Let $\zz_X:=\Res^{\KK^F}_{{\mathrm C}_\GF (X)}\w\zz$ whose restriction to $[\LL,\LL]^F$ is of central defect by applying for instance \Th{CuspBl} to $[\LL ,\LL]$. Note that $b_D$ is the block of $\zz_D$. By a slight variant of \Th{inctric} (see [CaEn, 5.29]) one gets the subpair inclusion $(X,b_{{\mathrm C}_\GF (X)}(\zz_X))\leq (D,b_D)$ and therefore 
\begin{equation}\label{bX}
\text{  $ b_X=b_{{\mathrm C}_\GF (X)}(\zz_X).   $ }
\end{equation} If $X$ is assumed centric in $\CC^F$, then  $(X,b_X)$ is centric, or equivalently $\zz_X(1)_\ell\geq |{\mathrm C}_\GF (X)/\zent X|_\ell$ because
\begin{eqnarray*} |{\mathrm C}_\GF (X)/\zent X|_\ell &=&|({\mathrm C}_{\KK_\aaa} (X){\KK_\bbb} )^F/\zent X|_\ell   \\ &=& |{\mathrm C}_{\KK_\aaa} (X)^F/\zent X|_\ell  .|\KK_\bbb^F|_\ell  {\rm \ \ by\ \ (\ref{ZGb})}\\ &=& |\KK_\bbb^F|_\ell {\rm \ \ (}X{\ \rm centric\ in}\ {\KK_\aaa^F}\inn {\mathrm C}^\circ_\GG (\KK_\bbb)^F)\\ &=&\zz (1)_\ell {\rm \ ,\ \ see\ above.} \end{eqnarray*}

Now assume conversely the weaker assumption that $(X,b_X)$ is centric. First $\zz_X$ is the canonical character of $b_X$ because it has $\zent X\in\KK_\aaa^F$ in its kernel. Moreover ${\mathrm C}^\circ_\GG (X)={\mathrm C}^\circ_{\KK_\aaa} (X){\KK_\bbb}$ has its first term of $\aaa$-type (an easy check by induction on $|X|$ in groups of type $\tA$) so  
\begin{equation}\label{Kb}
\text{  ${\mathrm C}^\circ_\GG (X)_\bbb =[\LL,\LL]= \KK_\bbb$. }
\end{equation} The restriction of $\w \zz$ to ${\mathrm C}^\circ_\GG (X)^F$ (or any $(\MM{\KK_\bbb} )^F$ with $\MM$ an $F$-stable connected reductive subgroup of ${\KK_\aaa}$) is a unipotent character, it is the unique one whose restriction to $[\LL,\LL]^F={\KK_\bbb^F}$ is the restriction of $\zz$. 

So we get 

\noindent (iii) {\it $\Res^{{\mathrm C}_\GF (X)}_{{\mathrm C}^\circ_\GG (X)^F}\zz_X\in\Irr ({\mathrm C}^\circ_\GG (X)^F)$
 is the only unipotent character $\zz_X^\circ\in\ser{{\mathrm C}^\circ_\GG (X)^F}1$ such that $\Res^{{\mathrm C}^\circ_\GG (X)^F}_{[\LL,\LL]^F}\zz^\circ_X=\Res^\LF_{[\LL,\LL]^F}\zz$. }

\smallskip

Let's keep $(X,b_X)$ centric. If $g\in\GF$ normalizes it, the above implies that $g$ normalizes $[\LL,\LL]$ while the canonical character of $b_X$ restricts to $[\LL,\LL]^F$ as $\Res^\LF_{[\LL,\LL]^F}\zz$. Then $g$ normalizes $([\LL,\LL] ,\Res^\LF_{[\LL,\LL]^F}\zz )$ and therefore $g\in H[\LL,\LL]^F\inn H{\mathrm C}_\GF (X)$ by assumption (i.b).

Conversely, if $h$ normalizes $X$ and $([\LL,\LL] ,\Res^\LF_{[\LL,\LL]^F}\zz )$, it normalizes ${\mathrm C}^\circ_\GG (X)$ and sends $\Res_{{\mathrm C}^\circ_\GG (X)^F}^{{\mathrm C}_\GG (X)^F}(\zz_X)$ to a unipotent character whose restriction to $[\LL,\LL]^F$ is $^h\Res_{[\LL,\LL]^F}^{{\mathrm C}_\GG (X)^F}(\zz_X)={}^h\Res^\LF_{[\LL,\LL]^F}\zz = \Res^\LF_{[\LL,\LL]^F}\zz$ by (iii) above. So $h$ fixes $\Res_{{\mathrm C}^\circ_\GG (X)^F}^{{\mathrm C}_\GG (X)^F}(\zz_X)$. By [NagaoTsu, 5.5.6], $b_X$ is the unique block covering $b_{{\mathrm C}^\circ_\GG (X)^F}(\Res_{{\mathrm C}^\circ_\GG (X)^F}^{{\mathrm C}_\GG (X)^F}(\zz_X))$ since the index is a power of $\ell$ (use Lemma~\ref{NonC}). So $b_X$ is fixed by $h$. By assumption (i.b) this applies to any $h\in \norm HX$. This completes the proof of (\ref{NorCic}).
\end{proof}

\subsection{A theorem of Broto-M\o ller-Oliver}

Until now we have compared only fusion systems of $\ell$-blocks of groups $\GF$ in the same defining characteristic $p$. Broto-M\o ller-Oliver [BrMO12] have proved a very impressive theorem showing equivalence of fusion systems of $\ell$-subgroups for groups $\GF$ of various defining characteristics. 

We give the theorem in a simplified form (the original one is stronger, see [BMO, Th. A]).

\begin{thm}\label{BMO}
Let $\GG$ a reductive group over $\bbF$. Let $\HH$ a reductive group over a field of $\Bbb K$ of characteristic $r$. Assume that for some maximal tori $\TT\leq\GG$, $\bS\leq \HH$, the two groups have same quadruple $\Hom (\TT ,\Fm)\nni \Phi(\GG ,\TT)$, $\Hom (\Fm ,\TT)\nni \Phi (\GG,\TT)^\vee$. Let $F\co\GG\to\GG$ a Frobenius endomorphism acting trivially on the root system, let $q$ the power of $p$ associated (for instance $q=|\XX_\al^F|$ for all $\al\in \Phi(\GG ,\TT)$). Let $F'\co \HH\to\HH$ a similar endomorphism for $\HH$ and $q'$ the corresponding power of $r$. 

Assume $\ell$ is a prime $\not\in\{  2,p,r\} $, assume $q$, $q'$ have same multiplicative order $d$ mod $\ell$ and that $(q^d-1)_\ell =(q'{}^d-1)_\ell $. 

Then $\GG^F$ and $\HH^{F'}$ have isomorphic fusion systems of $\ell$-subgroups.
\end{thm}

The proof would be too long to sketch here, see also [AschKeOl, Sect. III.1.7]. Let's say just that it uses all the strength of the topological methods developed by Broto-Levi-Oliver along with an old theorem of Friedlander [Fr82, 12.2] on \'etale homotopy of the algebraic groups $\GG$ and a less old one by Martino-Priddy (see [Mis90], [MaPr96]).

\begin{rem}
With the elementary group theoretical methods used in the proof of \Th{CuspBl} (Sylow $\phi_d$-tori and their normalizers) and under the same assumptions about $\ell$, it is easy to describe the Sylow $\ell$-subgroups of $\GF$ as semi-direct products $$Z\rtimes N$$ (see [Ca94, 4.4]) where \begin{itemize}
	\item  $Z=\zent{\cent\GG{\bS}}^F$ for $\bS$ a Sylow $\phi_d$-torus of $\GG$ (see also Remark~\ref{B0G}),
	\item $N$ is a Sylow $\ell$-subgroup of $(W_{\cent\GG{\bS}}(\TT)^\perp)^F$, where $\TT$ is an $F$-stable maximal torus of $\cent\GG{\bS}$ and $W_{\cent\GG{\bS}}(\TT)^\perp$ is the subgroup of the Weyl group $W_\GG(\TT)$ generated by reflections through roots orthogonal to any $\al\in\PGT$ with $\al(\bS)=1$
	\item  the action of $N$ on $Z$ comes from the inclusion $Z\leq\TT^F$.
\end{itemize}  All the above can be read in the ``root datum" quadruple of the pair $(\GG ,F)$. This would imply that the two finite groups $\GF$ and $\HH^{F'}$ of the Theorem above have isomorphic Sylow $\ell$-subgroups. Comparing the fusion systems needs to find the {\it essential} subgroups of $ZN$ in the sense of [AschKeOl, \S I.3] and the action of their normalizers. This has been determined in many cases by Jianbei An as a by-product of his program to determine radical subgroups (see Definition~\ref{radp}) in finite groups of Lie type and check Alperin's weight conjecture (\ref{AWC}) for those groups. See [AnDi12, \S 3] for essential $\ell$-subgroups of finite classical groups. See [AnDiHu14] and the references given there for many exceptional types.
\end{rem}


{}
\bigskip

{}

\section{Some applications}

\subsection{Abelian defect}

When the defect group of some block $B$ defined by \Th{CuspBl} is assumed to be abelian, the description of $\Irr(B)$ simplifies a lot. One keeps the same hypotheses on $\GG$, $F$, $\ell$, $d$, $(\LL,\zz)$.

\begin{thm}\label{AbDef}
	Assume the defect $\ell$-groups of $B_\GF(\LL,\zz)$ are abelian. Then $$\Irr(B_\GF(\LL,\zz))=\bigcup_{t,\chi_t}\Irr(\Lu{\GG(t)}{\GG}(\wh t\chi_t))$$ where $t$ and $\chi_t$ are subject to the following conditions\begin{enumerate}[\rm(a)]
		\item $t\in(\GD)^{F^*}_\ell$ ,
		\item $\LL\inn\GG(t)$ where the latter is a Levi subgroup in duality with $\ccent{\GD}{t}$ ,
		\item $\chi_t$ is an irreducible component of $\Lu{\LL}{\GG(t)}\zz$.
	\end{enumerate}
	\end{thm}
\begin{proof} By Corollary~\ref{dpair} and Proposition~\ref{DefUn} we know that the defect group can be abelian only if the centric subpair $(\zent{\LL}_\ell^F,b_\LF(\zz))$ is maximal and $\zent{\LL}_\ell^F$ is a Sylow $\ell$-subgroup of $\ccent\GG{[\LL,\LL]}^F$. By Corollary~\ref{dpair}.(ii) this means that the polynomial order of $(\ccent\GG{[\LL,\LL]},F)$ has not more powers of cyclotomic polynomials $\phi_{m}$ with $m_\lp =d$ than its (maximal) torus $(\czent{\LL},F)$. This property can be written entirely in the groups $X(\TT_0)$ and $Y(\TT_0)$ of $\GG$, so they transfer to the same property in the dual, namely $\ccent\GD{[\LL^*,\LL^*]}^{F^*}$ has a Sylow $\ell$-subgroup in $\czent{\LD}^{F^*}$. Then when imposing the condition that $t$ commutes with $[\LL^*,\LL^*]$ from \Th{NoUn}, one may assume that $t\in\czent{\LL^*}$ and therefore $\LD\inn\ccent{\GD}{t}$. Then one may choose $\GG(t)\nni \LL$. The last point is then clear from  \Th{NoUn} by use of Lemma~\ref{LL}.
\end{proof}

\subsection{Brauer's height zero conjecture}

The description of \Th{AbDef}, along with the parametrization of \Th{BMM'}, leads quickly to check the degrees in $\Irr (B_\GF(\LL,\zz))$ when the unipotent block $B_\GF(\LL,\zz)$ has abelian defect (see [BrMaMi93, 5.15]), keeping the restrictions on $\ell$ of \Th{CuspBl}. In particular, $\chi(1)_\ell$ takes only one value for $\chi\in\Irr(B_\GF(\LL,\zz))$, thus confirming \sing{Brauer's height zero conjecture} (BHZC) \index{(BHZC)}
\begin{equation}\label{BHZC}
\text{  $ D   $ is abelian if and only if $|\{\chi(1)_\ell\mid \chi\in\Irr(B) \} |=1$  }
\end{equation}
where $B$ is an $\ell$-block of a finite group with defect group $D$.

Kessar-Malle have proven 

\begin{thm}[{see [KeMa13], [KeMa17]}]\label{KeMa17} The equivalence of (\ref{BHZC}) is true for all blocks of finite quasi-simple groups. 
\end{thm}

Given past knowledge about alternating groups, sporadic groups and blocks of finite reductive groups for ``good" primes recalled above, Kessar-Malle's proof concentrates mostly on $\ell$-blocks of groups of Lie type for $\ell\leq 5$ where the challenge is still remarkably difficult.

 This type of verification in groups of Lie type is important because of the reduction theorems of Berger-Kn\"orr [BeKn88] and Navarro-Sp\"ath [NaSp14].

\begin{thm}[Berger-Kn\"orr]\label{BeKn} Let $\ell$ be a prime number.
	If for any $\ell$-block $B$ of a quasi-simple group with abelian defect group, $(\chi(1)_\ell)_{\chi\in\Irr(B)}$ is constant, then it is the case for any $\ell$-block with abelian defect of any finite group, i.e. (BHZ1), the ``only if" part of (\ref{BHZC}), holds.
\end{thm}

\begin{cor} [{see [KeMa13]}]\label{KeMa13}  If an $\ell$-block $B$ of a finite group has abelian defect groups, then $(\chi(1)_\ell)_{\chi\in\Irr(B)}$ is constant.
\end{cor}

The converse should be checked through Navarro-Sp\"ath's reduction theorem.
\begin{thm}[Navarro-Sp\"ath]\label{NaSp} If all blocks of finite quasi-simple groups satisfy the inductive Alperin-McKay condition of [Sp13, 7.2], then (\ref{BHZC}) holds.
\end{thm}

The reduction theorems for the two directions of Brauer's height zero conjecture are proven with very different methods pointing possibly to problems of quite different nature. While the proof of \Th{BeKn} uses module theoretic methods and a theorem of Kn\"orr on Green vertices of simple modules (see [Kn79]), the proof of \Th{NaSp} uses mainly the techniques described in [Sp17].

\subsection{Nilpotent blocks}

A {\it nilpotent} $\ell$-block is one such that any of its defect groups controls the fusion of its subpairs (see [Thev, Sect. 49], {[AschKeOl, Sect. IV.5.6]}). Namely

\begin{defn}\label{NilBl} An $\ell$-block $B$ of a finite group $H$ is called a \sing{nilpotent block} if and only if for any $B$-subpair $(P,b_P)$ in $H$ the quotient $$\text{  $   \norm{H}{P,b_P}/\cent{H}{P} $ }$$ is an $\ell$-group.
\end{defn}

As in all statements about the fusion system of subpairs, the condition above can be loosened to be asked only for {\it centric} $B$-subpairs $(P,b_P)$.
Note that the above condition about $\ell$-subgroups instead of subpairs would give the well-known local characterization of $\ell$-nilpotent groups (i.e., $H/{\rm O}_\lp(H)$ is an $\ell$-group) due to Frobenius [Asch, 39.4].

The main structure theorem about nilpotent blocks is due to Puig (see [Thev, \S 49-51], see also [Kuls, 15.3] for the easier version over a finite field).

\begin{thm}\label{PuigNil} Let $B$ be a nilpotent $\ell$-block seen as a subalgebra of $\OO H$. Let $D$ one of its defect group. Then there is an integer $m$ such that $$B\cong \Mat_m(\OO D).$$
\end{thm}

This of course implies the same over the finite field $k=\OO /J(\OO)$ and therefore the very important property 
\begin{equation}\label{l=1}
\text{  $|\IBr (B)|=1$. }
\end{equation} 
Determining nilpotent blocks of quasi-simple groups $H$ was achieved by An-Eaton. Their result implies

\begin{thm}[{[AnEa11,~1.1, 1.2], [AnEa13, 1.1, 1.3]}]\label{AnEa} Let $B$ be an $\ell$-block of a finite quasi-simple group $H$. Then $B$ is nilpotent if and only if $|\IBr(b_P)|=1$ for any $B$-subpair $(P,b_P)$. Moreover $B$ has abelian defect groups.
\end{thm}

 We prove below a slightly stronger statement concentrating on the property (\ref{l=1}) again in the framework of unipotent $\ell$-blocks with $\ell$ not too bad.

We keep $\GF$, $\ell\geq 5$ a prime good for $\GG$ (see [GeHi91, 2.1] or [CaEn, \S 13.2]) not dividing $q$, and $B_\GF (\LL ,\zz )$ a unipotent $\ell$-block of $\GF$ as in \Th{CuspBl}.

\begin{prop}\label{Nilpo}  {\sl  
		Assume $B_\GF (\LL ,\zz )$ has just one Brauer character. Then $B_\GF (\LL ,\zz )$ is a nilpotent block and its defect groups are abelian.    }\end{prop} 

\noindent{\it Proof.}  By [CaEn,~14.6], the restrictions of the elements of $\ser\GF 1$ to the set $\GG^F_\lp$ of $\ell$-regular elements of $\GF$ are distinct and linearly independent central functions. Since Brauer characters are a basis for the central functions on $\GG^F_\lp$, our hypothesis implies that $\ser\GF 1\cap\Irr (B_\GF (\LL ,\zz ))$ has a single element. By Theorem~\ref{CuspBl}, this implies that $\Lu\LL\GG(\zz )$ is a multiple of a single irreducible character. By \Th{BMM'}, this implies in turn that $\Irr ({\rm N}_\GF (\LL ,\zz )/\LF)$ has a single element and therefore $${\rm N}_\GF (\LL ,\zz )=\LF.$$ Then the centric subpair $(\zent{\LL}^F_\ell ,b_\LF (\zz ))$ of Proposition~\ref{DefUn} above is maximal (and the only centric subpair up to conjugacy). This can be seen by applying Proposition~\ref{maxic} and noting that $(\zent{\LL}^F_\ell ,b_\LF (\zz ))$ is normal in no other subpair since ${\rm N}_\GF (\zent{\LL}^F_\ell ,b_\LF (\zz ))=\LF={\rm C}_\GF (\zent{\LL}^F_\ell )$. This proves at the same time that the defect groups are abelian and that the block is nilpotent.
\qed

\subsection{Brou\'e's abelian defect conjecture when $\ell$ divides $q-1$}

Brou\'e's abelian defect conjecture [Bro90a, 6.2] is as follows. 

Let $H$ be a finite group, $(\OO,K,k)$ an associated $\ell$-modular system, $B$ a block of $\OO H$, $D$ its defect group and $B_D$ its Brauer correspondent (see \Th{BrThs}.(i) above) viewed as a subalgebra of $\OO \norm HD$. When $D$ is abelian, Brou\'e's \sing{abelian defect conjecture} says that the derived categories of $B$ and $B_D$ should be equivalent
\begin{equation}\label{ADC}
\text{  $ \mathsf{ D}^b(B\text{-}\mo ) \cong \mathsf{ D}^b(B_D\text{-}\mo )  $ }
\end{equation}
later strengthened to the requirement that
\begin{equation}\label{ADCho}
\text{  $ \HO(B\text{-}\mo ) \cong \HO(B_D\text{-}\mo )  $ }
\end{equation}
by a Rickard equivalence (see Definition~\ref{RickEqu} below), that is an equivalence of the homotopy categories with a strong compatibility with fusion. Note that here one does not expect consequences on the fusion systems of the blocks involved since in this case it is very simply the one of $B_D$ as a classical consequence of abelian defect.

In the case of principal blocks, Craven and Rouquier have proved a reduction theorem to simple groups [CrRo13]. The conjecture for arbitrary blocks with abelian defect has been checked in many cases. For the defining prime and $\SL_2(q)$ it was proved by Okuyama in the influential preprint [Oku00]. Chuang-Kessar showed it for certain blocks of symmetric groups [ChKe02]. This combined with \Th{ChRo} allow Chuang-Rouquier to also check it for blocks of symmetric groups [ChRo08, 7.6]. The same paper shows it for $\GL_n(q)$ for $\ell\nmid q$ as a consequence of the Rickard equivalences they prove between blocks of $\GL_n(q)$'s and theorems of Turner [Tu02] supplying results similar to [ChKe02] for those groups.  Dudas-Varagnolo-Vasserot in [DuVV15] and [DuVV17] have also checked Brou\'e's conjecture (and Rickard equivalences similar to \Th{ChRo}) for certain unipotent blocks of finite reductive groups of types $^2\tA$, $\tB$ and $\tC$ through categorifications they build for certain affine Lie algebras. For the application to Brou\'e's conjecture, some work of Livesey is also used to spot nicer representatives among Rickard equivalent blocks (see [Li15]).

We just prove here a very elementary case yet substantial where the equivalence is in fact a quite explicit Morita equivalence. The following is a simplification of a more general statement by Puig with a different proof [Puig90].

We keep $(\GG ,F)$ defined over $\bbF_q$.
\begin{thm}\label{Pu90} Let $\ell\geq 7$ be a prime dividing $q-1$. Let $D$ be a Sylow $\ell$-subgroup of $\GF$. Assume $D$ is abelian. Then the principal $\ell$-blocks over $\OO$ of $\GF$ and $\norm\GF D$ are Morita equivalent: $$ B_0(\GF)\mmo\cong  B_0(\norm\GF D)\mmo .$$
\end{thm}

We will prove more concretely

\begin{pro}\label{YHell} Let $\TT\inn\BB$ both $F$-stable in $(\GG,F)$ as above. Let $\ell$ be a prime dividing $q-1$ and such that $\norm{\GG}{\TT}^F/\TT^F$ is an $\lp$-group. Let $U\deq \RU(\BB)^F$, $T'=\TT^F_\lp$. Let $M\deq\Ind_{UT'}^{\GF}\OO$. Then $$\End_{\OO\GF}M\cong \OO(\norm{\GG}{\TT}^F/\TT_\lp^F)$$ by an isomorphism mapping any $t\in \TT_\ell^F$ to the endomorphism of $M$ sending $1\otimes 1$ to $t^\mm\otimes 1$.
\end{pro}

Let's see first how this will give \Th{Pu90}. We will essentially apply Lemma~\ref{K2O} and \Th{AbDef}.

Let's note that 
\begin{equation}\label{B'G}
\text{  $ \Irr(B_0(\GF))=\Irr(\Ind^\GF_{UT'}1)   $ .}
\end{equation}

This is indeed easy to deduce from \Th{AbDef} and the fact that $(\TT, 1)$ is the $1$-cuspidal pair satisfying $(\GF,1)\geq (\TT, 1)$.

Let us denote $$A= \End_{\OO\GF}(\Ind_{UT'}^{\GF}\OO).$$ By Proposition~\ref{YHell}, $A$ and $B_0(\GF)$ are both blocks of finite groups. Moreover $M$ is a bi-projective $B_0(\GF)\otimes A^\op$-module. Indeed right projectivity is ensured by writing $M=\OO\GF e$ for $e=|\BB^F_\lp|^\mm\sum_{x\in \BB^F_\lp}x$. For the right projectivity it suffices to check the restriction to a Sylow $\ell$-subgroup of $\norm{\GG}{\TT}^F/T'$ through the isomorphism of Proposition~\ref{YHell}. By the assumption on $\norm{\GG}{\TT}^F/\TT^F$ this is $\TT_\ell^F$ whose action on the right is said to be through the left action of $\TT^F$, so has already been checked.

By Brou\'e's lemma (Lemma~\ref{K2O}), in order to get \Th{Pu90} it suffices to show that $K\otimes_\OO M$ induces a bijection between simple $K\otimes_\OO A$-modules and $\Irr(B_0(\GF))$. One has $K\otimes_\OO A=\End_{K\GF}(K\otimes_\OO M)$, so $K\otimes_\OO M$ bijects the simple $K\otimes_\OO A$-modules and the $\Irr(K\otimes_\OO M)$. Then (\ref{B'G}) gives our claim. 

Let us now prove Proposition~\ref{YHell}. We abbreviate $G=\GF$, $N=\norm{\GG}{\TT}^F$, $T=\TT^F$, $W=N/T$. Using again that $W$ is an $\lp$-group, one has 
\begin{equation}\label{TW}
\text{  $ N/T'\cong T_\ell\rtimes W   $ }
\end{equation} by a map leaving unchanged the elements of $T_\ell$. 

On the other hand by Example~\ref{Hec}, $A$ has a basis $(a_n)_{n\in UT'\backslash G/UT' } $ defined by (\ref{aact}). Note that one can take $n\in N/T'$ by Bruhat decomposition  (\ref{fBru}). This contains $T_\ell$ for which the action of $a_t$ ($t\in T_\ell$) is by multiplication by $t$. One has clearly 
\begin{equation}\label{anat}
\text{  $  a_na_t=a_{nt}=a_{^nt}a_n  $ for any $n\in N/T'$, $t\in T_\ell$.}
\end{equation}

Let us consider the map 
\begin{equation*}\label{}
\text{  $    M=\Ind^G_{UT'}\OO\to \Ind^G_{UT}\OO \ \text{ defined by }\   1\otimes_{UT'}1\mapsto 1\otimes_{UT}1.$ }
\end{equation*}
One sees easily that the kernel of this map is stable under $A$,
so any endomorphism of $M$ induces an endomorphism of $ \Ind^G_{UT}\OO$ seen as a quotient. The corresponding morphism between endomorphism rings is (notations of Example~\ref{Hec})
\begin{equation}\label{anaw}
\text{  $ a_n\mapsto a_w   $ }
\end{equation} for $n\in N/T'$ and $w=nT\in W$. The algebra on the right, $\End_{\OO G}(\Ind^G_{UT}\OO)$ is the well-known Iwahori-Hecke algebra whose generators satisfy $(a_s)^2=(q_s-1)a_s+q_s$ for $s\in S$, $q_s\deq |U/U\cap U^s|$ (a power of $q$) and $a_wa_{w'}=a_{ww'}$ when the $l_S$-lengths add (see for instance [CurtisRei, 67.3] or deduce it from the proof of \Th{Yoko}). By the assumption on 
$\ell$, by reduction mod. $J(\OO)$ the above relations become the defining relations of $W$. So composing (\ref{anaw}) with reduction mod $J(\OO).\End_{\OO G}(\Ind^G_{UT}\OO)$ gives a ring morphism 
\begin{equation}\label{rho}
\text{  $ \rho\colon A\to kW   $ such that $\rho(a_n)=nT\in W$.}
\end{equation}
The kernel of $\rho$ is clearly $J(\OO T_\ell)$ where we identify $\oplus_{t\in T_\ell}\OO a_t$ with $\OO T_\ell$ as said before. So we get an exact sequence of $\OO$-modules 
\begin{equation}\label{JAW}
\text{  $ 0\to J(\OO T_\ell)A\to A\xrightarrow{\rho}kW\to 0   .$ }
\end{equation} Note that $J(\OO T_\ell)A\inn J(A)$ (in fact an equality) thanks to (\ref{anat}), so that $a_n\in A^\times$ for any $n\in N/T'$. Let $\Gamma\leq A^\times$ the group generated by the $a_n$'s ($n\in N/T'$). So (\ref{JAW}) yields an exact sequence of groups \begin{equation}\label{1JA}
\text{  $ 1\to \Gamma\cap (1+J(\OO T_\ell)A)\to \Gamma\xrightarrow{\rho}W\to 1   $ }
\end{equation} where the second term acts trivially on $T_\ell$. If the above had been done with $\OO/J(\OO)^m$ ($m\geq 1$) instead of $\OO$ we would get some $\Gamma_m$ an extension of the $\lp$-group $W$ by a finite $\ell$-group, so (\ref{1JA}) would split. In the general case we consider the $J(A)$-adic topology on $A$ for which $\rho$ is continuous. We have an exact sequence of groups
\begin{equation}\label{C1J}
\text{  $ 1\to \cent{1+J(\OO T_\ell)A}{T_\ell}\to \cent{1+J(\OO T_\ell)A}{T_\ell}.\Gamma\xrightarrow{\rho}W\to 1   .$ }
\end{equation} The sequence splits by a classical lemma about lifting of $J(A)$-closed subgroups (see [CaEn, 23.18]), thus giving some $W'\leq \cent{1+J(\OO T_\ell)A}{T_\ell}.\Gamma$ isomorphic to $W$ by $\rho$ and acting the same on $T_\ell$. Then $A=\OO T_\ell W'$ by Nakayama's lemma and the equality $ \OO T_\ell W' +J(A)=A$ implied by $\rho ( \OO T_\ell W')=kW$. This shows that $A\cong \OO (T_\ell\rtimes W)$ as claimed. \qed

\begin{rem}\label{Div} 
	A typical example of Morita equivalence between algebras $A$, $B$ over $\OO$ that are sums of blocks over finite groups is when $$B\cong \Mat_n(A)$$ for some integer $n\geq 1$. This is equivalent to our Morita equivalence inducing a bijection of characters 
	$$\Irr (A\otimes_\OO K)\to \Irr (B\otimes_\OO K)\ ;\ \chi\mapsto \chi '\eqno(*)$$
	 where the ratio of degrees $\chi '(1)/\chi (1)$ is a constant integer $n$ (see [CaEn, Ex. 9.6]). Examples are \Th{PuigNil} and the equivalences of Bonnaf\'e-Dat-Rouquier, see below \Th{BDR1.1}. 
	
	In the case of \Th{Pu90} above it is generally not the case. For instance when $\GF$ is $\SL_2(q)$ for $q$ a power of 2 and $\ell\geq 7$ is a prime divisor of $q-1$, then $\norm{\GG}{\TT}^F$ is a dihedral group of order $2(q-1)$ whose principal $\ell$-block has $((q-1)_\ell -1)/2 \geq 2$ characters of degree 2. On the other hand the whole $\Irr(\GF)$ has only one character of even degree, the Steinberg character (see for instance [DiMi, \S 15.9]). This makes impossible any bijection as in ($*$) with the ratio $\chi '(1)/\chi(1)$ being always an integer, even depending on $\chi$.
\end{rem}


{}
\bigskip

{}

\section{Bonnaf\'e-Dat-Rouquier's theorems}

The main theorem of [BoDaRo17] is about the situation of \Th{LCGs} above
where $(\GG, F)$ is defined over $\bbF_q$ with dual $(\GD , F^*)$, $\ell$ is a prime not dividing $q$, $s\in (\GD)^{F^*}_\lp$ is a semi-simple element and $\LL^*$ is an $F^*$-stable Levi subgroup of $\GD$ such that 
\begin{equation}\label{CGsL}
\text{  $ \cent{\GD}{s} \inn \LD  $ }
\end{equation} a condition sometimes weakened to
\begin{equation}\label{CCGsL}
\text{  $ \ccent{\GD}{s} \inn \LD  .$ }
\end{equation}

Bonnaf\'e-Dat-Rouquier's main theorem [BoDaRo17] in that situation is the following.

\begin{thm}[{[BoDaRo17, Th. 1.1]}]\label{BDR1.1} Let $\LD$ minimal for the condition (\ref{CCGsL}), let $\LL$ an $F$-stable Levi subgroup of $\GG$ associated with $\LD$ by duality, so that $\ser{\LF}{s}$ and $e_\ell (\LF,s)$ (see Definition~\ref{BGs}) make sense. Let $N$ be the stabilizer of $ e_\ell (\LF,s)$ in $\norm{\GG}{\LL}^F$. Then one has a Morita equivalence $$\OO Ne_\ell (\LF,s)\longrightarrow \OO\GF e_\ell (\GF,s).$$ Moreover if two $\ell$-blocks are related by the above they have isomorphic defect groups and fusion systems in the sense of Definition~\ref{FDbD}.
\end{thm}

\subsection{Etale topology and sheaves} We refer to [CaEn, Ch. A2] for the basic notions about algebraic varieties.

\noindent {\bf Sheaves on topological spaces.} (See [KaSch].) If a topological space is given by the datum of the set $\Op_X$ of open subsets of a certain set $X$, $\Op_X$ \index{$\Op_X$}can be considered as a category with $\Hom(O,O')=\{\to \} $ (a single element) when $O'\inn O$,  $\Hom(O,O')=\emptyset$ otherwise. A \sing{presheaf} on this topological space is then any functor 
\begin{equation}\label{presh}
\text{$\cF\colon \Op_X\to \Se$ or  $\cF\colon \Op_X\to A\Mmo$      }
\end{equation} to the category of sets or the category of $A$-modules for $A$ a ring. An example is the constant presheaf. When $O'\inn O$ in $\Op_X$ and $s\in \cF(O)$ one denotes $s_{|O'}\deq \cF(\to)(s)\in \cF(O')$. One also generally denotes $$\Gamma (X,\cF)=\cF (X)$$
({\it global sections}). A \sing{sheaf} is a presheaf $\cF$ such that if $(O_i)_i$ is a family of elements of $\Op_X$ and $s^i\in \cF(O_i)$ is a family such that $s^i_{|O_i\cap O_j}=s^j_{|O_i\cap O_j}$ for any $i,j$, then there is a unique $s\in \cF(\cup_iO_i)$ such that $s_{|O_i}=s^i$ for any $i$. There is a canonical way, called ``sheafification", 
\begin{equation}\label{F+}
\text{  $ \cF\mapsto\cF^+   $ }
\end{equation}
 to associate a sheaf with a presheaf, adjoint to the forgetful functor $\cF\mapsto\cF$. The constant sheaf is the sheafification of the constant presheaf. For $M$ in $A\Mmo$, the associated constant sheaf $M_X$ on $X$ satisfies $M_X(U)=M^{\pi_0(U)}$ where $\pi_0(U)$ is the set of connected components of $U$. When $f\colon X\to X'$ is continuous and $\cF$, $\cF'$ are sheaves on $X$, $X'$ respectively the formulas $$f_*\cF (U')=\cF(f^\mm(U'))\ \ \ ,\ \ f^*\cF '(U)=\lim _{U'\nni f(U)}\cF '(U')  $$\index{$f_*$}\index{$f^*$}define \sing{direct and inverse images of sheaves} that are obvious presheaves. They are made into sheaves by (\ref{F+}) keeping the same notation. For the map $\si_X\colon X\to \{\bullet \}$, one gets
 \begin{equation}\label{si*}
 \text{   $(\si_X)_*\cF =\Gamma (X,\cF)$.  }
 \end{equation}
 
 When $j\colon X\to X'$ is an open immersion (i.e. a homeomorphism between $X$ and $j(X)\in \Op_{X'}$)
 then one defines a presheaf by
 \begin{equation}\label{j!}
 {   j_!\cF (U')   = }\begin{cases}
 \cF (U')& \text{ if } U'\inn j(X)\cr
 0&\text{ otherwise. }
 \end{cases}
 \end{equation}
 This is also made into a sheaf by (\ref{F+}).
 
 Most sheaves of interest are deduced from locally constant sheaves by those operations. Assume $X$ is pathwise connected and locally simply connected. Let $\pi_1(X,x_0)$ its fundamental group (homotopy classes of loops based at a given $x_0$). The topological relevance of sheaves is partly contained in the elementary fact that locally constant sheaves with values in sets and some additional finiteness condition (finite stalks) are in bijection with continuous finite $\pi_1(X,x_0)$-sets.
 
 The category $Sh_A(X)$ \index{$Sh_A(X)$}of sheaves $\cF\colon \Op_X\to A\Mmo$ has enough injectives. When $f\colon X\to X'$ is continuous, we can right-derive the left exact functor $f_*\colon Sh_A(X)\to Sh_A(X')$ into $$\R f_*\colon \tD^+(Sh_A(X))\longrightarrow \tD^+(Sh_A(X')).$$ In the case of (\ref{si*}) one writes \index{$\R\Gamma (X,\cF)$}
 \begin{equation}\label{RGam}
 \text{  $    \R\Gamma (X,\cF) \in \tD^+(A\Mmo)$ }
 \end{equation}
 since $Sh_A(\{\bullet \})=A\Mmo$. The \sing{$i$-th cohomology $A$-module of $\cF$ } 
 is by definition $\rH^i(X,\cF)\deq\rH^i(\R\Gamma(X,\cF))$.

 \noindent {\bf \'Etale cohomology.}\index{\'etale cohomology.} (See [Milne], [CaEn, A3], [Du17, \S 2].)
Let $\XX$ be a variety over $\bbF$. The sheaves for the \'etale topology on $\XX$ and their cohomology are roughly defined as follows from the topological model sketched above. The topology on $\XX$ is not the Zariski topology but a Grothendieck topology where $\Op_X$ is replaced by the category $\XX_{et}$\index{$\XX_{et}$} whose objects are \'etale maps  of varieties over $\bbF$ with codomain $\XX$ $$\UU\to \XX$$  and morphisms are given by commutative triangles. Presheaves are defined with values in $A\Mmo$ for $A$ a ring that is generally finite of characteristic prime to $p$. 

A lot of adaptations are needed to define substitutes to intersections (pullbacks), coverings, sheaves, etc... 
One defines a certain category of sheaves $Sh_A(\XX_{et})$  (finiteness and constructibility assumptions) to which the homological constructions of above can apply.
The map $\si_X\colon X\to\{\bullet \}$ used above is replaced by the structural map $\si_\XX\colon \XX\to \mathrm{Spec}(\bbF)$. This leads to $\R\Gamma (\XX,\cF)\in \tD^+(A\mmo)$ and the corresponding cohomology modules. One has also a notion of \sing{cohomology with compact support}.  Assume one has a compactification $j\colon\XX \hookrightarrow \ov \XX $ (an open embedding with $\ov\XX$ complete), then 
\begin{equation}\label{RGac}
\text{  $    \R\Gamma_c(\XX,\cF)\deq \R(\si_{\ov\XX})_*j_!\cF\in \tD^+(A\mmo)$ }
\end{equation} with corresponding homology groups $\rH^i_c(\XX ,\cF)\deq \rH^i(\R\Gamma_c(\XX,\cF))$\index{$\rH^i_c(\XX ,\cF)$}\index{$\R\Gamma_c(\XX,\cF)$}.

The notion of \sing{$\ell$-adic cohomology} (here with compact support) is defined as follows. Denote $\OO^{(n)}\deq \OO/J(\OO)^n$ (recall $\OO$ is a finite extension of $\bbZ_\ell$). An $\ell$-adic sheaf is a projective system $\cF=(\cF^{(n)})_{n\geq 1}$ of sheaves where $\cF^{(n)}\in Sh_{\OO^{(n)}}(\XX)$  and $\cF^{(n)}=\cF^{(n+1)}\otimes (\OO^{(n)})$ . Then $$\rH^i_c(\XX ,\cF)\deq \varprojlim _n \rH^i_c(\XX ,\cF^{(n)}) \in \OO\mmo .$$

For instance the module $\rH^i_c(\YY_\PP)$ defining the functor $\Lu{\LL}{\GG}$ of Deligne-Lusztig theory in Definition~\ref{RLG} is $\rH^i_c(\YY_\PP)\deq\bbC\otimes_\OO \varprojlim _n \rH^i_c(\YY_\PP ,\OO^{(n)}_{\YY_\PP})$.

Compactifications give rise to the notion of \sing{ramification}. The context is roughly as follows. Assume one has a compactification $j\colon \XX\to \ov\XX   $ with smooth $\ov\XX$ and complement $\ov\XX\setminus \XX = \DD_1\cup \DD_2\cup\dots$ a smooth divisor with normal crossings. For each irreducible component $\DD_m$ let $$\ov\XX_m =\ov\XX\setminus\cup_{i\ne m}\DD_i\text{\ \ \ \  and \ \ \ \ } \XX\xrightarrow{j_m} \ov\XX_m\xleftarrow{i_m} \DD_m\setminus \cup_{i\mid i\ne m}(\DD_m\cap \DD_i)$$ the associated open and closed immersions.
\begin{defn}[]\label{Fram} One says that $\cF$ ramifies along $\DD_m$ when $\cF$ is not of the  form $j_m^*\cF_m$ for $\cF_m$ a sheaf on $\ov\XX_m$.
\end{defn}

Then by results from [SGA4.5], [SGA5] (see also the survey in [CaEn, A3.19]) one gets that the above condition is equivalent to 
\begin{equation}\label{iRj}
\text{  $ i_m^*\R(j_m)_*\cF =0   $ }
\end{equation}
(otherwise $\cF = j_m^*\cF_m$ for $\cF_m\deq(j_m)_*\cF$).

\subsection{Brou\'e's reduction} In the context of the functor $\Lu{\LL}{\GG}$, one starts in general with an $F$-stable Levi subgroup $\LL$ complement in a parabolic subgroup $\PP$. From (\ref{YPP}) recall the varieties
 $$\GG/\PP\nni \XX_\PP \deq\{g\PP\mid g^\mm F(g)\in \PP F(\PP) \}$$ and $$ \YY_\PP\deq\{g\RU(\PP)\mid g^\mm F(g)\in \RU(\PP) F(\RU(\PP)) \}$$ and the actions of $\GF$ and $\LF$ on them. One has  ${\pi}\colon \YY_\PP\to\XX_\PP$ the $\LF$-quotient map sending $g\RU(\PP)$ to $g\PP$.

\begin{defn}[{}]\label{cFs} Let $\LL$ and $s$ be like in \Th{broIso}. Recall $\OO^{(n)}\deq \OO/J(\OO)^n$.  Then let $\cF_s= (\cF_s^{(n)})_{n\geq 1}$ where $$\cF_s^{(n)}\deq \pi_*(\OO^{(n)}_{\YY_\PP})e_\ell (\LF,s)$$ recalling that $\pi_*$ sends sheaves of $\OO^{(n)}$-modules to sheaves of $\OO^{(n)}\LF$-modules.\index{$\cF_s^{(n)}$}\index{$\cF_s^{}$}
\end{defn}

For simplicity we assume that $\XX_\PP$ is affine. This is conjectured in general and known in many cases (see [Du17, ]). However what follows can be proven knowing just that it is quasi-affine, which is the case (see [CaEn, 7.15]).
\begin{thm}[{[Br90b]}]\label{Br90} If there exists a compactification $\XX_\PP\xrightarrow{j}\ov\XX_\PP\xleftarrow{i}\ov\XX_\PP\setminus\XX_\PP$ such that $$i^*\R j^*\cF^{(n)}_s =0$$ for all $n\geq 1$, then $$\varprojlim_n \rH^{\dim\YY_\PP}_c(\YY_\PP,\OO^{(n)})e_\ell (\LF,s)$$ induces a Morita equivalence $$B_\ell (\LF ,s)\mmo\to B_\ell (\GF,s)\mmo .$$
\end{thm}

\begin{proof} A first consequence of (\ref{iRj}) for $\cF_s^{(n)}$ is that 
	\begin{equation}\label{j!R}
	\text{  $ j_!\cF_s^{(n)}\cong \R j_*\cF_s^{(n) }  $ .}
	\end{equation}
	This is seen by applying to $\cF_s^{(n) }$ the open-closed exact sequence $$0\to j_!\to j_*\to i_*i^*j_*\to 0$$ suitably right-derived (see [Du17, 2.6]) into a distinguished triangle (note that $j_!$, $i^*$ and $i_*$ are right exact).
	
We omit the subscripts $\PP$ from now on. We denote $\ov\si\colon \ov \XX\to \SF$ and $\si =\ov\si\circ j\colon \XX\to \SF$
the structure morphisms of $\ov\XX$ and $\XX$. Then (\ref{j!R}) and the definition of $\R\Gamma$ and $\R\Gamma_c$ allow to write 
\begin{equation}\label{RRc}
\R\Gamma(\XX,\cF_s^{(n) } )=\R\si_*\cF_s^{(n) } =\R\ov\si_*\circ\R j_*\cF_s^{(n) } =\R\ov\si_*\circ j_!\cF_s^{(n) }  =\R\Gamma_c(\XX,\cF_s^{(n) } ).  \end{equation}
Since $\XX$ is affine of dimension the same $d$ everywhere, $\R\Gamma(\XX,\cF_s^{(n) } )$ has cohomology in degrees only within the interval $[0,d]$ (see [Du17, \S 2.1]). But by Poincar\'e-Verdier duality (see [Du17, 2.4]) since $\XX$ is smooth, $\R\Gamma_c(\XX,\cF_s^{(n) } )$ has cohomology in degrees $\in [d,2d]$. So (\ref{RRc}) implies that $\R\Gamma_c(\XX,\cF_s^{(n) } )=\R\Gamma_c(\YY ,\OO^{(n)}).e_\ell(\GF ,s)$ has cohomology in degree $d$ only. Let's call $H^{(n)}$ this cohomology $\OO^{(n)}$-module. One can prove that it is $\OO^{(n)}$-free. Moreover the groups $\GF$ and $\LL^F$ act on $\YY$ with stabilizers that are finite unipotent groups of order invertible in $\OO^{(n)}$ (trivial in the case of $\LL^F$). So applying for instance [Du17, 2.4] one gets that both restrictions of $H^{(n)}$ to $\OO^{(n)}\GF$ and $\OO^{(n)}\LF$ are projective. So the same is true for $H^\infty$ the limit over $n$. By definition $\bbC\otimes_\OO \varprojlim_n \R\Gamma_c(\YY ,\OO^{(n)})$ is the bimodule inducing the functor $\Lu{\LL\inn\PP}{\GG}$, so $H^\infty$ is actually a bi-projective $\OO\GF\otimes\OO\LF^\op$-module such that $K\otimes_\OO H^\infty$ induces the bijection between ordinary characters  $$\cE_\ell(\LF ,s)=\Irr (K\LF e_\ell(\LF ,s))\to \cE_\ell(\GF,s)=\Irr (K\GF e_\ell(\GF ,s))$$ thanks to \Th{LCGs}. Now we have everything to apply Lemma~\ref{K2O} and get our claim.
\end{proof}

\begin{rem}\label{L=T} When $\LL$ is a torus, Deligne-Lusztig have shown the existence of an $\ov \XX$ such that (\ref{iRj}) is satisfied [DeLu76, 9.14]. So the Morita equivalence holds in that case [Br90b, 3.6]. Note however that in that case methods similar to Sect. 7 above allow to show that $B_\ell(\GF, s)$ (see Definition~\ref{BGs}) is a single block which is nilpotent of defect $\LL_\ell^F$. The Morita equivalence is then a consequence of \Th{PuigNil} which gives the structure of nilpotent blocks in general.
\end{rem}

\subsection{Bonnaf\'e-Rouquier (2003)}
In view of \Th{Br90}, the main objective of [BoRo03] is to prove

\begin{thm}[{[BoRo03, 11.7]}]\label{BR03}
 There exists a smooth compactification $\XX_\PP\xrightarrow{j}\ov\XX_\PP\xleftarrow{i}\ov\XX_\PP\setminus\XX_\PP$ such that $i^*Rj^*\cF_s =0$.\end{thm}

As a consequence the authors get 
\begin{thm}[{[BoRo03, Theorem B']}]\label{BRMorita} Assume $\cent{\GD}{s}^{F^*}\inn \LD$. One has a Morita equivalence $$ \OO\LF e_\ell(\LF,s)\text{-}\mo\to \OO\GF e_\ell(\GF,s)\text{-}\mo .$$
\end{thm}

The construction of the smooth compactification for varieties $\XX_\PP$ with $\PP$ a Borel subgroup extends the one of Bott-Samelson-Demazure-Hansen for Schubert varieties (which are obtained by removing the condition involving $F$ in what follows). Let $\BB_0$, $\TT_0$ a pair of $F$-stable Borel and torus as before. Let \index{$\Sigma$}$$\Sigma\deq\bigcup_{m\geq 0}(S\cup\{1\})^m$$ the set of finite sequences of elements of $S\cup\{1\}$. One recalls that a lifting $S\to\norm{\GG}{\TT_0}$, denoted $s\mapsto \dot s$ has been chosen satisfying the braid relations of the Weyl group (see (\ref{dotsi}) above). For $w=(s_1,\dots ,s_r)\in \Sigma$, let\index{$\XX(w)$}\index{$\YY(w)$}

 \leftline{$\XX(w)\deq \{(X_1,\dots ,X_r)\in (\GG/\BB_0)^r\mid$ }
 
 \rightline{$ X_1^\mm X_2\in \BB_0s_1\BB_0,\dots ,X_{r-1}^\mm X_r\in \BB_0s_{r-1}\BB_0, X_r^\mm F(X_1)\in \BB_0s_r\BB_0   \}$}

 \leftline{$\YY(w)\deq\{(Y_1,\dots ,Y_r)\in (\GG/\UU_0)^r\mid $ }
 
 \rightline{$Y_1^\mm Y_2\in \UU_0\dot s_1\UU_0,\dots ,Y_{r-1}^\mm Y_r\in \UU_0\dot s_{r-1}\UU_0, Y_r^\mm F(Y_1)\in \UU_0\dot s_r\UU_0   \}.$}

Both are acted on by $\GF$ on the left, the first is also acted on by $\TT_0^{wF}\deq \TT_0^{s_1\dots s_rF}$ on the right.
The reduction mod $\BB_0$ gives \index{$\pi_w$}a finite quotient $$\pi_w\colon \YY(w)\to \XX(w)\cong \YY(w)/\TT_0^{wF}.$$

Let 
\begin{equation}\label{ovX}\index{$\ov\XX(w)$}
\text{  $  \ov\XX(w)= \bigcup_{w'\leq w}\XX(w')  $ }
\end{equation} where $w'\leq w$ means that $w'\in \{1,s_1\}\times\cdots\times\{1,s_r\}$. This is smooth just like $\BB_0\cup \BB_0s_i\BB_0$ is smooth, being an algebraic group.

Bonnaf\'e-Rouquier define the set $\nabla $ \index{$\nabla$}of pairs $(w,\th)$ where $w=(s_1,\dots ,s_r)\in\Sigma$ and $$\th\colon \TT_0^{wF}=\TT_0^{s_1\dots s_rF}\to k^\times$$ is a group morphism ([BoRo03, \S 4.4]). For such a pair they define $w_\th =(s'_1,\dots ,s'_r)$ by \index{$w_\th$}
 \begin{equation}\label{wtheta}
{   s'_i = }\begin{cases}
1& \text{ if } s_i\ne 1 \text{ and } \th\circ N_{s_1\dots s_r}(s_1\dots s_{i-1}(\delta_i^\vee))=1 \cr
s_i&\text{ otherwise, }
\end{cases}
\end{equation} where $\delta_i^\vee\in \Phi(\GG,\TT_0)^\vee$ is the fundamental coroot corresponding to $s_i$ and $N_{v}\colon Y(\TT_0)\to \TT_0^{vF}$ for $v\in W(\GG,\TT_0)$ is the norm map used in the classical identification $Y(\TT_0)/(1-vF)Y(\TT_0)\cong \TT_0^{vF}$ (see for instance [DigneMic, 13.7]).

They also define $\cF_{(w,\th)}$ and $\cS_{(w,\th)}$ as follows.

\begin{defn}[{}]\label{cFwth} Let $b_\th\in k\TT_0^{wF}$ the primitive idempotent such that $\th(b_\th)\ne 0$. Let $$\cF_{(w,\th)}=(\pi_w)_*k_{\YY(w)}.b_\th$$ a sheaf \index{$\cF_{(w,\th)}$}\index{$\cS_{(w,\th)}$}on $\XX(w)$ with values in $k$-vector spaces. Since $\R\Gamma (\YY(w),k_{\YY(w)})$ is represented by a complex of $k\GF\otimes k\TT_0^{wF}{}^\op$-modules, we can define $$\cS_{(w,\th)}\deq \R\Gamma (\YY(w),k_{\YY(w)})b_\th\in \tD^b(k\GF\mmo).$$
\end{defn}

One proves

\begin{thm}[{[BoRo03, 7.7]}]\label{BRsh}
	For $w'\leq w$, let $j_{w'}^{\ov w}\colon \XX(w')\to \ov\XX(w)$ \index{$j_{w'}^{\ov w}$}the inclusion from (\ref{ovX}). Then $\R(j_{w}^{\ov w})_*\cF_{(w,\th)}$ is annihilated by $
(j_{w'}^{\ov w})^*$ unless $w_\th\leq w'$.\end{thm}

\begin{thm}[{[BoRo03, Th. A]}]\label{BRmod} The subcategory of $\tD^b(k\GF\text{-}\mo)$ generated (through shifts, direct sums, direct summands and mapping cones) by the $\cS_{(w,\th)}$ for $(w,\th)\in\nabla$ contains the regular module $k\GF[0]$.
\end{thm}

Those two theorems, of a quite different nature, both concern only varieties associated to Borel subgroups, not parabolic subgroups. 
The proof of \Th{BRsh} needs a particularly deep study of the sheaves and tori actions involved, see [BoRo03, \S 4]. See also [BoRo09] on a related question. For the proof of \Th{BRmod}, see [Du17, \S 3.5].
Note that [BoDaRo17, 1.2] gives a strengthened version of that theorem (see also [Du17, 3.12]).

Let's sketch briefly how \Th{BR03} is deduced from those two theorems (proof of [BoRo03, 10.7]). The pair $(\LL ,F)$ can be changed into $(\LL_I,\dot vF)$ for some $I\inn S$ and $\dot v\in\norm{\GG}{\TT_0}$ with $vF(I)v^\mm =I$ through conjugation by an element of $\GG$. Then the varieties $\XX$ and $\YY$ of \Th{BR03} become $\XX_{I,v}=\{g\PP_I\mid g^\mm F(g)\in \PP_I\dot vF(\PP_I)\}$ and $\YY_{I,v}=\{g\UU_I\mid g^\mm F(g)\in \UU_I\dot vF(\UU_I)\}$ with evident  \index{$\XX_{I,v}$}\index{$\YY_{I,v}$}$\LL_I^{\dot vF}$-quotient map $\pi\colon \YY_{I,v}\to \XX_{I,v}$. Abbreviating $L=\LL_I^{\dot vF}$, one has to prove 
\begin{equation}\label{Rjpi}
\text{  $  i^*\R j_*(\pi_*k\otimes_{kL}kL\ov e_\ell(L,s)) =0$ }
\end{equation}
where we have kept the notation $i,j$ for the immersions associated with $\XX_{I,v}\inn \ov\XX_{I,v}$ the later being the Zariski closure in the complete variety $\GG/\PP_I$. By the generation property of \Th{BRmod} (applied to $\LL_I$) it suffices to check
\begin{equation}\label{RjpiS}
\text{  $  i^*\R j_*(\pi_*k\otimes_{kL}\cS_{(w,\th)}^{\LL_I,\dot vF}) =0 $ }
\end{equation} for any $(w,\th)\in\nabla^{\LL_I,\dot vF}$ relating to $s$ by duality. 

Let $d_v\in S^{l_S(v)}$ be a reduced expression of $v$ and $w\cup d_v$ \index{$w\cup d_v$}be the concatenation in $\Sigma$. Let $\tau\colon \XX(w\cup d_v)\to \XX_{I,v}$ defined by $(g_1\BB_0,\dots  )\mapsto g_1\PP_I$. By basic properties of (derived) direct image functors and an isomorphism of varieties related to the transitivity of Deligne-Lusztig induction, one gets 
\begin{equation}\label{Rtau}
\text{  $ \pi_*k\otimes_{kL}\cS_{(w,\th)}^{\LL_I,\dot vF}=\R\tau_*\cF_{(w\cup d_v,\th)}   $ .}
\end{equation} 
On the other hand we have $j\tau =\ov\tau j_{w\cup d_v}^{\ov{w\cup d_v}}$ where $\ov\tau\colon \ov\XX(w\cup d_v)\to \ov\XX_{I,v}$ is $(g_1\BB_0,\dots )\mapsto g_1\PP_I$, a proper morphism. So now (\ref{RjpiS}) reduces to  
\begin{equation}\label{Rjtau}
\text{  $  i^*\R\ov\tau_* \R( j_{w\cup d_v}^{\ov{w\cup d_v}})_*\cF_{(w\cup d_v,\th)}  =0 $. }
\end{equation}
 One now applies base change (see for instance [CaEn, A3.5]) and gets \begin{equation}\label{iRtau}
\text{  $  i^*\R\ov\tau_* =\R\tau'_*\circ i^*_v $ }
\end{equation} where 
\begin{equation}\label{iv}
\text{  $    i_v\colon \bigcup_{w'\leq w, v'\lneq d_v}\XX(w'\cup v')=\ov\XX(w\cup d_v)\setminus \ov\tau^\mm (\XX_{I,v})\longrightarrow \ov\XX(w\cup d_v)$ }
\end{equation} is the open immersion and $\tau'$ is the restriction of $\ov\tau$. In view of (\ref{Rjtau}) and (\ref{iRtau}) it then suffices to prove that \begin{equation}\label{iRt}
\text{  $  i_v^*\R( j_{w\cup d_v}^{\ov{w\cup d_v}})_*\cF_{(w\cup d_v,\th)}=0$. }
\end{equation} The situation is now close to the one covered by \Th{BRsh} for each inclusion $\XX(w'\cup v')\to\ov\XX(w\cup d_v)$. One checks that $(w\cup d_v)_\th =w_\th\cup d_v\not\leq w'\cup v'$ for pairs $(w,\th)$ relating to $s$. \Th{BRsh} then tells us $i_{w'\cup v'}^*\R( j_{w\cup d_v}^{\ov{w\cup d_v}})_*\cF_{(w\cup d_v,\th)}=0$ for each $i_{w'\cup v'}\colon \XX({w'\cup v'})\to \ov\XX(w\cup d_v)$ involved in (\ref{iv}). This implies (\ref{iRt}) by checking stalks.

\subsection{Bonnaf\'e-Dat-Rouquier (2017)}

{}

Among many results (see also [Du17, 3.12]) the paper [BoDaRo17] shows that the situation of \Th{BRMorita} implies more than a Morita equivalence. The hypothesis is also slightly strengthened assuming just (\ref{CCGsL}).

One takes $\GG$, $\GD$ in duality, $s$ a semi-simple $\lp$-element of $\GD^{F^*}$. One lets $$\LL^*_s\deq \cent{\GD}{\czent{\ccent\GD s}}\rhd \ccent{\GD}{s}  \text{ and } \NN^*_s =\cent{\GD}{s}^{F^*}\LL^*_s$$ so $\LL^*_s$ is the smallest Levi subgroup of $\GD$ containing $\ccent\GD s$\index{$\LL^*_s$}\index{$\LL_s$}\index{$\NN^*_s$}. Let $\LL_s$ be an $F$-stable Levi subgroup of $\GG$ in duality with $\LL^*_s$. Note that $\ser{\LL_s^F}{s}$ makes sense. 

 Let $\NN_s\leq \norm\GG{\LL_s}$ \index{$\NN_s$}such that $\NN_s/\LL_s$ identifies with $\NN^*_s/\LL^*_s$ through duality. It is then $F$-stable and $$\NN_s^{F}=\norm{\GF}{\LL_s,\ser{\LL_s^F}{s}}$$ so that $e_\ell (\LL_s^F,s)$ is a central idempotent of $\OO\NN_s^F$.

 The following establishes a Morita equivalence for the blocks in characteristic $\ell$.
 
\begin{thm}[{[BoDaRo17, 7.5]}]\label{BDR1}  Let $\PP= \RU(\PP)\LL_s$ be a parabolic subgroup having $\LL_s$ as Levi subgroup. We have: \begin{enumerate}[\rm(i)]
		\item  The action of $\GF\times (\LL_s^F) ^\op$ on  $\rH^{\dim\YY_\PP}_c(\YY_\PP,k)$ extends to $\GF\times (\NN_s^F) ^\op$.
		\item  The resulting bimodule induces a Morita equivalence $$k\NN_s^F\ov e_\ell(\LL_s^F,s)\text{-}\mo \longrightarrow k\GF \ov e_\ell(\GF,s)\text{-}\mo .$$
	\end{enumerate}
\end{thm}

\medskip\noindent{\bf A. Independence of the parabolic $\PP$.} A first step in proving \Th{BDR1}.(i) is to show that the $k\GF\otimes k\LL_s^F{}^\op$-module $\rH^{\dim\YY_\PP}_c(\YY_\PP,k)$ is invariant under the action of $\GF\times \NN_s^F$ (through automorphisms of $\GF\times \LL_s^F$ induced by $\GF\times \NN_s^F$). Only the action of some $x\in\NN_s^F$ needs to be checked. The action of $\GF\times \LL_s^F{}^\op$ twisted by $(1,x)$ on $\YY_\PP =\{ g\RU(\PP)\mid g^\mm F(g)\in \RU(\PP)F(\RU(\PP))  \}$ is clearly the action of $\GF\times \LL_s^F{}^\op$ on $\YY_\PP\cdot x =\YY_{\PP^x}$. The equivariance of \'etale cohomology (here the automorphism is even a homeomorphism for Zariski topology) implies that $\rH^{\dim\YY_\PP}_c(\YY_\PP,k)^{(1,x)}\cong \rH^{\dim\YY_\PP}_c(\YY_{\PP^x},k)$ and therefore $$\big( \rH^{\dim\YY_\PP}_c(\YY_\PP,k)\ov e_\ell (\LL_s^F,s)\big)^{(1,x)}\cong \rH^{\dim\YY_\PP}_c(\YY_{\PP^x},k)\ov e_\ell (\LL_s^F,s).$$  The parabolic subgroups $\PP$ and $\PP^x$ have both $\LL_s$ as a Levi complement, so the invariance sought is a consequence of the following.

\begin{thm}[{[BoDaRo17, 6.5, 6.7]}]\label{BDRinv} If $\PP$ and $\PP'$ are parabolic subgroups of $\GG$ admitting $\LL_s$ as Levi complement then $$\rH^{\dim\YY_\PP}_c(\YY_{\PP},k)\ov e_\ell (\LL_s^F,s)\cong \rH^{\dim\YY_{\PP '}}_c(\YY_{\PP '},k)\ov e_\ell (\LL_s^F,s)$$ as $k\GF\otimes k\LL_s^F{}^\op$-modules.
\end{thm}

\begin{rem}\label{BDRem} 
Note that this answers the question of the dependence of the Morita equivalence from \Th{BRMorita} on the parabolic subgroup used. Note that the corresponding statement on characters was known for long [DigneMic, 13.28]. In the case of $F$-stable $\PP$, $\PP'$, one has $$\OO(\GF/\RU(\PP)^F)\cong \OO(\GF/\RU(\PP')^F)$$ as bimodules (Dipper-Du, Howlett-Lehrer [CaEn, 3.10]). In general one does not have \Th{BDRinv} without projecting on $\ov e_\ell (\LL_s^F,s)$. \Th{BDRinv} has also been used by Dat in a study of representations of $p$-adic groups [Dat16].\end{rem}

The proof is quite delicate ([BoDaRo17, Sect. 5-6], see additional explanations and perspective in [Dat15]). It involves ``intermediate" varieties of type $$\YY_{\PP ,\PP '} \deq \{(g\RU(\PP),g'\RU(\PP')) \mid g{}^\mm g'\in \RU(\PP)\RU(\PP')\ ,\ g'{}^\mm F(g)\in \RU(\PP')F(\RU(\PP)) \}$$ and maps from their cohomology complexes (up to a shift) to our $\R\Gamma_c(\YY_\PP ,k)$ or $\R\Gamma_c(\YY_{\PP'} ,k)$. Then getting quasi-isomorphism once those are projected on the sum of blocks $B_\ell (\LL_s^F,s)$ needs quite a lot of additional considerations in the spirit of the proof of \Th{BRsh}.

\medskip\noindent{\bf B. Extendibility.}
The next problem is to extend $\rH^{\dim\YY_\PP}_c(\YY_\PP,k)$ into a $k(\GF\times\NN_s^F{}^\op)$-module. 


 
 One actually shows that the obstruction to extending $\rH^{\dim\YY_\PP}_c(\YY_\PP,k)$ is the same in $\GG$ as it would be in an overgroup with connected center, where this obstruction does not exist.
One considers a regular embedding (see [CaEn, \S 15.1]), that is an inclusion of algebraic groups $$\GG\hookrightarrow \w\GG=\GG\zent{\w \GG}$$ with connected $\zent{\w \GG}$\index{$\w\GG$}. One may assume that $F$ extends to $\w\GG$. This induces a surjection $\si\colon \w\GG^*\to 
 \GD$ with connected central kernel. Denote $$ \w\PP\deq \zent{\w \GG}\PP \ ,\ \w\LL_s =\zent{\w \GG}\LL_s\ ,\ \w\NN_s\deq \zent{\w \GG}\NN_s .$$\index{$\w\PP$}\index{$\w\LL_s$}\index{$\w\NN_s$}
\begin{defn}\label{Je}
Let $J\inn \si^\mm (s)^{F^*}_\lp\inn \w\LL_s^*{}^{F^*}$ be a representative system for $\w\GG^*{}^{F^*}$-conjugacy in $\si^\mm (s)^{F^*}_\lp$. Let $$e\deq \sum_{\w t\in J}\ov e_\ell (\w\LL_s^F,\w t)\in \zent{k\w\LL_s^F}.$$ One also defines the following subgroups of  $\w\GG^F\times (\w\GG^F)^\op$ $$\w\cL\deq \w\GG^F\times (\w\LL_s^F{})^\op\lhd \w\cN \deq\w\GG^F\times (\w\NN_s^F{})^\op ,$$ \index{$\w\cL$}\index{$\w\cN$}\index{$\cL$}\index{$\cN$}  $$\cL\deq (\GG^F\times (\LL_s^F{})^\op)\Delta\w\LL_s^F\lhd \cN \deq(\GG^F\times (\NN_s^F{})^\op)\Delta\w\NN_s^F$$ (where $\Delta H=\{(h,h^\mm)\mid h\in H \}$ for a given subgroup $H\leq \w\GG^F$).
\end{defn} 

Note that $\w \cN/\w \cL\cong \NN_s^F/\LL_s^F\cong (\NN^*_s/\LL^*_s)^{F^*}\leq (\cent{\GD}{s}/\ccent{\GD}{s})^{F^*}$, an $\lp$-group by Lemma~\ref{NonC}. 

From the definition of $\YY_\PP$, it is clear that it is acted on by $\cL$, so we may consider $$M\deq \rH_c^d(\YY_\PP,k)\ov e_\ell (\LL_s^F,s)\ ,\ \w M=\Ind^{\w\cL}_\cL M.$$\index{$\w M$}  The variety $\YY_{\w\PP}\inn \w\GG/\RU(\PP)$ is defined with regard to the same unipotent subgroup as $\YY_\PP$ so it has same dimension as $\YY_\PP$ and one has (see for instance [CaEn, 12.15.(iii)])  
\begin{equation}\label{IndY}
\text{  $\w M\cong \rH_c^d (\YY_{\w\PP},k)\ov e_\ell(\LL_s^F,s).  $ }
\end{equation}
Some mild considerations in the dual groups show that 
\begin{equation}\label{eell}
\text{  $ \ov e_\ell (\GF,s)=\sum_{\w t\in J}^{} \ov e_\ell (\w\GG^F,\w t)  $ and $ \ov e_\ell (\LL_s^F,s)=\sum_{x\in \NN_s^F/\LL_s^F}e^x .$}
\end{equation}

In $\w\GG$, one has $\cent{\w\GG^*}{\w t}=\si^\mm (\ccent{\GG^*}{s})\inn \w\LL_s^*$, so Bonnaf\'e-Rouquier's theorem (\Th{BRMorita} above) implies 
\begin{equation}\label{wMorita}
\text{  $ \w Me\otimes_{\w\LL_s^F}-   $ induces a Morita equivalence $k\w\LL_s^Fe\text{-}\mo\longrightarrow k\w\GG^F\ov e_\ell(\GF ,s)\text{-}\mo$.}
\end{equation}

So $\w Me$ is a direct sum of pairwise non-isomorphic indecomposable $k\w\cL$-modules. The same applies to $\w M\cong \Res^{\w\cN}_{\w\cL}\Ind^{\w\cN}_{\w\cL}\w Me$.

The next step is to deduce from the above that $M$ extends to $\cN$. When the quotient $\cN/\cL$ is cyclic this is enough to extend the action of $\cL$ on $M$ into an action of $\cN$ (see for instance [Da84, 4.5]). For the general case, see Remark~\ref{BDRatum} below. One writes
\begin{equation}\label{M=Res}
\text{  $M=\Res^\cN_\cL M'$ for some $k\cN$-module $M'$. }
\end{equation} It is not too difficult to deduce from (\ref{wMorita}) that $\Ind^{\w\cN}_\cN M'$ induces a Morita equivalence 
\begin{equation}\label{wMo2}
\text{  $ k\w\LL_s^F\ov e_\ell (\LL_s^F,s)\text{-}\mo  \longrightarrow  k\w\GG^F\ov e_\ell (\GG^F,s)\text{-}\mo. $ }
\end{equation}
Then one shows that $M'$ induces the sought Morita equivalence $$k\NN_s^F \ov e_\ell (\LL_s^F,s)\longrightarrow k\GG^F\ov e_\ell (\GG^F,s).$$

For instance the canonical map $k\NN_s^F \ov e_\ell (\LL_s^F,s)\to \End_{k\GF}(M')$ is indeed an isomorphism since $k\w\NN_s^F \ov e_\ell (\LL_s^F,s)\to \End_{k\w\GG^F}(\Ind^{\w\cN}_\cN M')$ is one by (\ref{wMo2}) and one has $\End_{k\w\GG^F}(\Ind^{\w\cN}_\cN M')\cong \End_{k\GF}(M)\otimes_{\NN_s^F}\w\NN_s^F$.

This finishes the proof of \Th{BDR1}.

\medskip\noindent{\bf C. Rickard equivalence and local structure.}
Bonnaf\'e-Dat-Rouquier prove then that \Th{BDR1} can be strengthened to a Rickard equivalence preserving the local structure of the blocks of $\GF$ and $\NN_s^F$ that are related through this equivalence. 

Recall [BoDaRo17, 2.A] \begin{defn}[{}]\label{RickEqu}
 A \sing{Rickard equivalence} between sums of block algebras $A$, $A'$ over $\Lambda\in\{\OO ,k \}$ is an equivalence $$\HO(A)\to \HO(A')$$ induced by a complex $C$ of bi-projective $A'\otimes_\Lambda A^\op$-modules such that the canonical maps $A\to  {\rm{End}}^\bullet_{A'}(C)$ and $A'\to  {\rm{End}}^\bullet_{A^\op}(C)$ are isomorphisms in $\HO(A\otimes_\Lambda A{}^\op\text{-}\mo)$ and $\HO(A'\otimes_\Lambda  A'{}^\op\text{-}\mo)$ respectively (notations of [Du17, \S 1.2]).\end{defn}

 This coincides with the usual definition requiring additional properties of the Green vertices of summands of $C$ by results of Rouquier [Rou01].
On the other hand those $\ell$-subgroups of the product of the two finite groups involved serve as a bridge between the local structures of blocks so related. While Rickard's original paper [Rick96] had the assumption that blocks involved have same defect group, one can prove that a Rickard equivalence in the above sense {\it implies} a strong relation at the level of local subgroups. The following is due to Puig.

\begin{thm}[{[Puig99, 19.7]}]\label{Pu99}
If two $\ell$-block algebras $A$, $A'$ over $\OO$ are Rickard equivalent then the defect groups are isomorphic $D\cong D'$ and the associated fusion systems (see Definition~\ref{FDbD}) on $D$ and $D'$ are equivalent. 
\end{thm}

Recall a theorem of Rickard (see [Du17, 2.2]).

\begin{thm}[{[Rick95], [Rou02]}]\label{Rick95}
The element $\R\Gamma_c (\YY_\PP,\OO)$ of $\tD^b(\OO\GF\otimes\OO\LL_s^F{}^\op)$ is represented by a well-defined element $\mathrm{G}\Gamma_c  (\YY_\PP,\OO)$ of $\HO(\OO\GF\otimes\OO\LL_s^F{}^\op)$\index{$\mathrm{G}\Gamma_c  (\YY_\PP,\OO)$}  whose terms are direct summands of modules of type $\Ind^{\GF\times \LL_s^F{}^\op}_Q\OO$ where $Q$ is an $\ell$-subgroup of $\GF\times \LL_s^F{}^\op$ such that $(\YY_\PP)^Q\ne\emptyset$. \end{thm}

The main result of [BoDaRo17] can then be stated as follows.

\begin{thm}[{[BoDaRo17, 7.7]}]\label{BDR2}
In the framework of \Th{BDR1} for $\GG$, $s$, $\PP\geq \LL_s\lhd \NN_s$ the complex $\mathrm{G}\Gamma_c  (\YY_\PP,\OO)e_\ell(\LL_s^F,s)$ induces a Rickard equivalence between $\OO\GF e_\ell (\GF,s)$ and $\OO\NN_s^Fe_\ell(\LL_s^F,s)$.
\end{thm}

We sum up some of the main features of the proof ([BoDaRo17, \S 7.D]).

One works first over $k$. Denote $$C=\mathrm{G}\Gamma_c(\YY_\PP,\OO)e_\ell(\LL_s^F,s)\otimes k.$$ \index{$C$}The main step is to prove the following.

\begin{prop}[]\label{Endgr} ${\End}^\bullet_{k\GF}(C)\cong \End_{D^b(k\GF)}(C)[0]$ in $\HO (k\LL_s^F\times \LL_s^F{}^\op)$
\end{prop}

The proof of that Proposition leads to checking the following about the action of $\GF\times \LL_s^F{}^\op$ on $\YY_\PP$.
	
\begin{lem}[{[BoDaRo17, 3.5]}]\label{YQ}	Assume $\PP=\RU(\PP)\LL{}$ is a Levi decomposition with $F(\LL)=\LL$.	If $Q$ is an $\ell$-subgroup of $\GF\times \LL^F{}^\op$ with fixed points on $\YY_\PP$, then $Q$ is $\GF\times \LL^F{}^\op$-conjugate to a subgroup of $\Delta(\LF)\deq \{(x,x^\mm)\mid x\in\LF  \}\inn \GF\times \LF{}^\op$.		
	\end{lem}

Let us now recall that for $H$ a finite group, an {$\ell$-permutation $kH$-module}\index{$\ell$-permutation module} is by definition any direct summand of a permutation module. For $Q$ an $\ell$-subgroup of $H$ and $M$ an $\ell$-permutation $kH$-module one denotes 
\begin{equation}\label{BrQ}
\text{  $ \Br_Q(M)\deq M^Q/(M^Q\cap J(kQ)M )  $ in $k\cent{H}{Q}\text{-}\mo$ }\index{$ \Br_Q(M)$}
\end{equation} the image of the $Q$-fixed points of $M$ in the cofixed points (see also [Du17, \S 2.3]). This induces an additive functor from $\ell$-permutation $kH$-modules to $\ell$-permutation $k\cent HQ$-modules. Note that if $\Omega$ is a set acted upon by $H$, then  $\Br_Q(k\Omega)=k(\Omega^Q)$ which allows to identify our first definition (\ref{Brmo}) of the Brauer morphism with a special case of the above.

The following, chiefly due to Bouc, is very useful to check homotopic equivalence locally.
\begin{lem}[{[Bouc98, 6.4, 6.9]}]\label{BrE}
Let $E$ be a bounded complex of $\ell$-permutation $kH$-modules. Assume that for any $\ell$-subgroup $Q\leq H$, $\Br_Q(E)$ has homology in degree 0 only. Then $$E\cong \rH^0(E)[0] \ \ \text{in   }\HO(kH\text{-}\mo).$$
\end{lem}
	
	This will be applied to $H=\LL_s^F\times \LL_s^F{}^\op$ and $E\deq \End^\bullet_{k\GF}(C)$. 

Lemma~\ref{YQ} somehow shows that the relevant $\ell$-subgroups to check are of the form $\Delta Q$ for $Q$ an $\ell$-subgroup of $\LL_s^F$. By a theorem of Rickard (see [Du17, 2.11]) $\Br_{\Delta Q}(\R\Gamma_c(\YY_\PP,k))$ identifies with $\R\Gamma_c((\YY_\PP)^{\Delta Q},k)$. In [BoDaRo17, \S 3.A] it is shown that $(\YY_\PP)^{\Delta Q}$ is to be considered as a variety $\YY_{\cent\PP Q}$ in the (possibly non-connected) reductive group $\cent{\GG}{Q}$, which in turn gives sense to and establishes
\begin{equation}\label{BrC}
\text{  $  \Br_{\Delta Q}(C)=\mathrm{G}\Gamma_c(\YY_{\cent\PP Q}^{(\cent\GG Q)},k) \Br_Q(\ov e_\ell (\LL_s^F,s)). $ }
\end{equation}

Then the authors show for $\Br_Q(\ov e_\ell (\LL_s^F,s))$ a formula [BoDaRo17, 4.14] generalizing the one of Brou\'e-Michel seen before (\Th{Brell}) for cyclic subgroups $Q$. This allows to identify the right hand side of (\ref{BrC}) with a sum of complexes of the same type as $C$ itself in the local subgroup $\cent{\GG}{Q}^F$. One applies to them Bonnaf\'e-Rouquier's theorem (\Th{BRMorita}) thus getting that their homology is in one single degree. This essentially gives Proposition~\ref{Endgr} thanks to Lemma~\ref{BrE}. Let us comment that the above adaptations to the case of non-connected reductive groups needs indeed a lot of work [BoDaRo17, Sect. 3-4].

The next steps go through the following propositions and are less difficult. Remember that one is looking for a complex acted on by $\NN_s^F$ on the right.

\begin{prop}[]\label{Step2} One has\begin{eqnarray*}
		\End_{\HO(k(\GF\times\NN_s^F{}^\op))}(\Ind^{\GF\times\NN_s^F{}^\op}_{\GF\times\LL_s^F{}^\op}(C))&\cong\End_{\tD^b(k(\GF\times\NN_s^F{}^\op))}(\Ind^{\GF\times\NN_s^F{}^\op}_{\GF\times\LL_s^F{}^\op}(C)) \\
		\cong&\End_{k(\GF\times\NN_s^F{}^\op)}(\Ind^{\GF\times\NN_s^F{}^\op}_{\GF\times\LL_s^F{}^\op}(\rH^d_c(\YY_\PP,k))  \end{eqnarray*} 
\end{prop}

The proof of the following uses \Th{BDR1}.(i).

\begin{prop}[]\label{Step3} There is a direct summand $\w C$ of $\Ind^{\GF\times\NN_s^F{}^\op}_{\GF\times\LL_s^F{}^\op}(C)$ satisfying\begin{enumerate}[\rm(i)]
		\item $\Res^{\GF\times\NN_s^F{}^\op}_{\GF\times\LL_s^F{}^\op}(\w C)\cong C$ and
		\item $\End^\bullet_{k\GF}(\w C)\cong \End_{\tD^b(k\GF)}(\w C)[0]\cong k\NN_s^F\ov e_\ell(\LL_s^F,s)[0]$ in $\HO(k(\NN_s^F\times\NN_s^F{}^\op))$.
	\end{enumerate}
\end{prop}

Using relatively standard techniques allowing with extra information to check only one of the two isomorphisms of Definition~\ref{RickEqu}, one now gets a Rickard equivalence over $k$. Lifting all that to $\OO$ as claimed in \Th{BDR2} follows classical procedures (see [Rick96, 5.2]).

\begin{rem}\label{BDRatum} Recalling that we have assumed for (\ref{M=Res}) above that $\NN_s/\LL_s$ is cyclic, one gets at this point \Th{BDR2} in that case. We even get a similar statement for any $F$-stable Levi subgroup $\LL^*$ (replacing $\LL^*_s$) such that $\LL^*$ contains $\ccent{\GD}{s}$, is normalized by $\cent{\GD}{s}$ and the factor group $(\cent{\GD}{s}\LD)^{F^*}/\LD{}^{F^*}$ is cyclic. Then the two algebras shown to be Rickard equivalent are $\OO\GF e_\ell(\GF ,s)$ and $\OO N e_\ell(\LF ,s)$ where $N\leq \norm{\GG}{\LL}^F$ is such that $N/\LF\cong (\cent{\GD}{s}\LD)^{F^*}/\LD{}^{F^*}$. 
	
	When $\NN_s/\LL_s$ is not cyclic, the proof proposed by Bonnaf\'e-Dat-Rouquier in {[BoDaRo17b]} consists in several steps described to the present author as follows (October 2017). First, one reduces the problem to groups $\GG$ that are simple as algebraic group (finite center and irrreducible root system) through direct products and central extension. Once this is done, the cases to care about are when  $\NN_s/\LL_s$ or equivalently $\cent{\GD}{s}/\ccent{\GD}{s}$ is not cyclic, which  
by Lemma~\ref{NonC}	can occur only in type $\tD_{2n}$ ($n\geq 2$). There are three possibilities for $s$ up to conjugacy, but only one such that $(\cent{\GD}{s}\LL_s^*)^{F^*}/\LL_s^*{}^{F^*}$ is not cyclic. Then $\ccent{\GD}{s}$ has type $\tA_{2n-3}$ and one can choose an $F$-stable Levi subgroup $\LL^*$ of type $\tA_{2n-3}\times \tA_{1}\times \tA_{1}$ satisfying the above. One then gets the equivalence sought between $\OO\GF e_\ell(\GF ,s)$ and $\OO N e_\ell(\LF ,s)$ where $N/\LF$ corresponds to a subgroup of order 2 of $\cent{\GD}{s}/\ccent{\GD}{s}$. Going from $\OO N e_\ell(\LF ,s)$ to our goal $\OO \NN^F_s e_\ell(\LF ,s)$ can then be done by proving versions of \Th{BDRinv} and (\ref{wMo2}) in a non-connected group $\HH$ such that $\HH^\circ =\LL$ and $\HH/\LL$ covers the missing part of $\cent{\GD}{s}/\ccent{\GD}{s}$. We refer to {[BoDaRo17b]} for more details.
\end{rem}



{}
\bigskip

{}

\section{Recreation: Blocks of defect zero}

 Modular group algebras $kH$ where the characteristic of $k$ divides the order of the finite group $H$ are the typical examples of non semi-simple algebras but they of course may have blocks that are indeed simple. This may be seen as rather exceptional and the local structure or representation theory of such blocks is quite trivial. But on the other hand a statement like Alperin's weight conjecture (see 3.C above) crucially needs that enough of those situations exist. There are very few general theorems ensuring that a finite group algebra has such blocks and this is probably related with how difficult it is to say anything general about Alperin's conjecture. Using CFSG, one can see that non abelian simple groups have a lot of blocks of defect zero.

\begin{thm}\label{10.1}
	Let $\ell$ be a prime and $S$ a finite non-abelian simple group. Then it has an $\ell$-block of defect zero (see 1.D) except in the following cases \begin{enumerate}[(a)]
		\item $\ell=2$ and $S$ is an alternating group ${\Alt}_n$ for $n\geq7$ such that neither $n$ nor $n-2$ is a triangular number, or one of the sporadic groups $M_{12}$, $M_{22}$, $M_{24}$, $J_2$, $HS$, $Suz$, $Co_1$, $Co_3$, $BM$. 
		\item $\ell=3$ and $S$ is an alternating group ${ \Alt}_n$ for $n\geq 7$ such that $(3n+1)_p$ is non-square for at least one prime $p\equiv -1 (3)$, or $S$ is one of the two sporadic groups $Suz$ and $Co_3$.
	\end{enumerate}
\end{thm}

The checking of this theorem on the character table of a given simple group is easy since an $\ell$-block of defect zero is signaled by an ordinary character of degree divisible by the highest power of $\ell$ dividing the order of the group (see [NagaoTsu, 3.6.29]). This applies to the 26 sporadic groups and the 18 primes $\{2, 3, \dots , 43, 47, 59, 67, 71 \}$
that divide the order of one of them.

For groups of Lie type, note that the theorem asserts that all have blocks of defect zero for all primes. This was checked by Michler [Mi86, 5.1] for odd $\ell$ and Willems  [Wi88] for $\ell=2$. When $\ell$ is the defining prime, the Steinberg module gives such a block (see \Th{StM} above). Assume now that the defining prime is some $r\not= \ell$. Then the checking for a group $S=\GF/\Z(\GF)$ basically consists in finding regular semi-simple elements $s\in [\GG^*,\GG^*]^{F^*}$ whose centraliser is a maximal torus $\TT^*$ such that $\TT^F/\Z(\GF)$ is of order prime to $\ell$. Then the corresponding Deligne-Lusztig character $\pm\Lu{\TT}{\GG}\th$ is irreducible with degree $|\GF/\TT^F|_{r'}$ (see [DigneMic, 12.9]), has $\Z(\GF)$ in its kernel and therefore is in an $\ell$-block of defect zero of $\GF/\Z(\GF)$.

Strangely enough the answer for alternating groups was known after the case of simple groups of Lie type. The problem reduces to the case of the symmetric group $\Sym{n}$ except for the prime $2$. There a $2$-block of $\Sym{n}$ can restrict to a block of defect zero of $\Alt_n$ if it has defect group $1$ or $\Sym 2$. 
 Theorem~\ref{BrRo} gives the defect group of a 2-block in terms of $2$-cores. It
  is easy to see that a Young diagram has no $2$-hook if and only if its rim has the shape of a regular stair. This means that this is the partition $m,m-1,\dots ,1$ of the triangular number $m(m+1)/2$. Similarly the Young diagram can have only one $2$-hook if it is as above plus two more boxes at the first row, or two more boxes at the first column. Hence $n-2$ being a triangular number. 
  
 For primes $\ell\geq 3$, one gets from Theorem~\ref{BrRo} that $\Alt_n$ has an $\ell$-block with defect zero if and only if there exists an $\ell$-core $\kappa\vdash n$.

 Many things have been known for a long time about cores since their introduction by Nakayama [Naka41a]. If \sing{$c_d(n)$} denotes the number of $d$-cores $\kappa\vdash n$, one has 
 \begin{equation}\label{genf}
 \text{  $\sum_{n\geq 0}^{} c_d(n) t^n=\prod_{n\geq 1}{(1-t^{dn})^d\over (1-t^n)}  $ }
 \end{equation}
  as seen from the theory of $d$-quotients [JamesKer, \S~2.7.30]. The numbers are documented in https://oeis.org/A175595.
  
The general theorem about existence of $d$-cores was finally reached by Granville-Ono in 1996.
 
 \begin{thm}[{[GrOn96]}]
 Let $n\geq 1$, $d\geq 3$. Then $c_d(n)=0$ if and only if $d=3$ and $(3n+1)_\ell$ is a non-square for some prime $\ell\equiv -1(3)$.
 \end{thm}

For $d=3$, Granville-Ono show that \begin{equation}\label{genf}
\text{  $ c_3(n) =\sum_{m{\mid}3n+ 1,\ m\ge1}{\Big({m\over 3}\Big)}  $ }
\end{equation}
 where $({m\over 3})$ is the Legendre symbol. Using Gauss' quadratic reciprocity law [Serre, p.7], this gives the statement about $c_3(n)$. An elementary proof of (\ref{genf}) can be found in
[HiSe09].
	
	 For $d\geq 4$, one claims that there are $d$-core partitions of $n$ for any $n$. Note that (\ref{d-dd'}) implies that one can restrict the study to $d=4$, 6 and odd $d\geq 5$. 
	 
	Using a variant of $\beta$-numbers one also shows elementarily that $c_d(n)\not=0$  if and only if there is a $d $-tuple of integers $(x_1,\dots ,x_d )$ such that (see [GaKiSt90, Bij. 2]) \begin{equation}\label{EqCo}
	n=\sum_{i=1}^d  \Big({d \over 2}\cdot x_i^2+(i-1)\cdot x_i\Big)  \text{ \   and   \ }  \sum_{i=1}^d  x_i=0.
	\end{equation}
	
	  Granville-Ono [GrOn96] solve the above using modular forms but mainly elementary arguments for primes $d\geq 17$. We conclude by giving below an elementary argument taken from [Ki96]. 
The number theoretic flavor is quite apparent.

\begin{pro}
	Assume $d\geq 9$ is an odd integer and $n\geq {1\over 4}d^3+{3\over 4}d-1$. Then $c_d(n)\not=0$.
\end{pro}

Note that $n\leq{1\over 4}(d-1)^2$ leads to trivial solutions (use Young diagrams included in a square to get $\la\vdash n$ such that $\hoo_c(\la)=\emptyset$ for any $c\geq d$). 

\begin{proof}
	(Kiming) One solves the problem in the form of (\ref{EqCo}). One will need 8 integers $x_1,\dots ,x_8$ with sum 0 to represent $n$, whence the condition $d\geq 9$.
	
	The condition $n\geq {1\over 4}d^3+{3\over 4}d-1$ implies that the euclidean division of $n$ by $d$ gives $n=dq+r$ with $4q\geq {d^2-1}$ and $d-1\geq r\geq 0$. Let's change slightly the parity of these integers while keeping $r^2$ small. Let 
	$$ (q',r')\deq\begin{cases} (q,r)  &\text{ if }  q\equiv 1\  (2) \text{ and } r\not\equiv 0\ (4) \\ (q+1,r-d)  &\text{ if }  q\equiv r\equiv 0\ (2) \\  
	(q+2,r-2d)  &\text{ if }  q\equiv 1\ (2) \text{ and } r \equiv 0\ (4) \\
	(q-\eps,r+\eps d)  &\text{ if }  q\equiv 0\ (2) \text{ and } r \equiv \eps d\ (4) \text{ for } \eps=\pm 1 .
	\end{cases} $$
	
	We still have $n=dq'+r'$ but now $q'$ is odd and one of the following occurs \begin{enumerate}[(a)]
		\item $r'$ is odd and $4q'\geq r'{}^2$ (two first cases above).
		\item $r'\equiv 2\ (4)$ and $16q'\geq r'{}^2$.
	\end{enumerate}  
	
	Let's look at case (a). Then $0<4q'-r'{}^2\equiv 3\ (8)$ so $4q'-r'{}^2$ can be represented by a sum of three odd squares (see [Serre, p. 45]). Therefore \begin{equation}\label{rabc}
	4q'=r'{}^2+a^2+b^2+c^2
	\end{equation} with $r',a,b,c$ odd. If necessary, we may change $a$ into $-a$  to ensure that $r'+a+b+c$ is a multiple of 4. One then defines \begin{eqnarray*}
		\al =(r'+a+b+c)/4& &\beta =(r'-a-b+c)/4 \\ \gamma =(r'-a+b-c)/4& &\dd =(r'+a-b-c)/4
	\end{eqnarray*}
	and $(x_1,\dots ,x_8)=(-\al,\al,-\beta,\beta,-\gamma,\gamma,-\dd,\dd)$.
	
	One has $x_1+\cdots +x_8=0$ and \begin{eqnarray*}
		{d\over 2}  x_1^2+{d\over 2}  x_2^2+x_2+{d\over 2}  x_3^2+2x_3+\cdots +{d\over 2}  x_8^2+7x_8 &=& {d\over 4}  (r'{}^2+a^2+b^2+c^2)+r'\cr
		&=&dq'+r'=n \text{ by (\ref{rabc})}.
	\end{eqnarray*} This solves the equation (\ref{EqCo}).
	
	In the case (b), one replaces $r'$ by $r'/2$ to define $a,b,c,\al,\beta,\gamma,\dd$ similarly. Then one takes $(x_1,\dots ,x_8)=(-\al,-\beta,\al,\beta,-\gamma,-\dd,\gamma,\dd)$.
\end{proof}

\bigskip

{}

\bigskip

{}

\bigskip



{}
\bigskip

{}

\def\ch#1{{ (#1)}}

{%
\small\footnotesize

\def\refit#1{\bibitem#1{#1}}

}

\printindex

\end{document}